\documentclass[11pt]{amsart}
\usepackage[utf8]{inputenc}

\usepackage{amsmath,amssymb,amsfonts,amsthm}
\usepackage{graphicx}
\usepackage{xcolor}
\usepackage[colorlinks=true]{hyperref}
\usepackage{epsfig, enumerate}
\usepackage{tikz}
\usetikzlibrary{arrows.meta,decorations.pathreplacing,positioning}

\numberwithin{equation}{section}
\newcommand{\R}{\mathbb{R}}
\newcommand{\N}{\mathbb{N}}
\newcommand{\Z}{\mathbb{Z}}

\newcommand{\W}{\mathcal{W}}
\newcommand{\CZ}{\mathcal{CZ}}
\newcommand{\SL}{\mathrm{SL}}
\newcommand{\GL}{\mathrm{GL}}
\newcommand{\PGL}{\mathrm{PGL}}
\newcommand{\Diff}{\mathrm{Diff}}

\newcommand{\beqnn}{\begin{eqnarray*}}
\newcommand{\eeqnn}{\end{eqnarray*}}
\newcommand{\beqn}{\begin{eqnarray}}
\newcommand{\eeqn}{\end{eqnarray}}
\newcommand{\beq}{\begin{equation}}
\newcommand{\eeq}{\end{equation}}

\theoremstyle{plain}
\newtheorem{thm}{Theorem}[section]
\newtheorem{prop}[thm]{Proposition}

\newtheorem{lem}[thm]{Lemma}
\newtheorem{cor}[thm]{Corollary}
\newtheorem{rmk}[thm]{Remark}
\newtheorem{defi}{Definition}[section]

\begin{document}

\title{Local centralizer rigidity for a non-generic diagonal map}
\author{Zhijing Wendy Wang}
\address{Department of Mathematics, University of Chicago, Chicago, IL 60637, USA}
\email{zhijingw@uchicago.edu}
\author{Amie Wilkinson}
\address{Department of Mathematics, University of Chicago, Chicago, IL 60637, USA}
\email{wilkinso@uchicago.edu}

\begin{abstract}
 In this paper, we prove centralizer rigidity for any non-trivial diagonal map in the compact homogeneous space $\mathrm{SL}_n\R/\Gamma, n\ge 5$, under the assumption that the perturbed volume-preserving diffeomorphism has a regular isomorphic centralizer. 
\end{abstract}

\maketitle
\tableofcontents

\section{Introduction}

The dynamical symmetries of a  diffeomorphism \(f\in \mathrm{Diff}^\infty(X)\) are captured by its
{\em smooth centralizer}
\[
        \mathcal Z(f)
        =
        \{h\in \mathrm{Diff}^\infty(X): hf=fh\}.
\]
From
the point of view of generic dynamics, one expects these symmetries to be
small: Smale \cite{smale} conjectured that generically the centralizer consists of the iterates of the map itself.  Algebraic systems, however, behave very differently.  If
\(f_0\) is an affine map on a homogeneous space, then its centralizer $\mathcal Z(f_0)$ often
contains many algebraic symmetries.  The centralizer rigidity problem asks
whether this distinction is (locally) rigid: if a smooth perturbation of an algebraic
system still has a large centralizer, must the perturbation itself be
algebraic?

In this paper, we consider the compact homogeneous space $X=\mathrm{SL}_n\R/\Gamma$. Fix throughout the paper a Riemannian metric on $X$, which descends from a right-invariant metric on $\mathrm{SL}_n\R$. A left translation $f_0=L_a: X\to X$ is called a {\em diagonal map} if {$a$ belongs to the identity component of a real split Cartan subgroup; equivalently, $a$ is conjugate in $\SL_n(\mathbb R)$ to a positive diagonal matrix}; we say $f_0$ is {\em non-trivial} if $f_0^2\neq id$ (equivalently, the eigenvalues of $a$ do not all have modulus $1 $);  we say $f_0$ is {\em generic} if the eigenvalues of $a$ are distinct.

For a nontrivial diagonal map $f_0$, work of \cite{DWWX} implies that the smooth centralizer agrees with the algebraic centralizer; in particular $\mathcal Z(f_0)\doteq \R^{n-1}$ if $f_0$ is generic, and $\mathcal Z(f_0)\doteq \R^{k-1}\times \mathrm{SL}_{n_1}\R\times \cdots \times \mathrm{SL}_{n_k}\R$, with $\sum_{i=1}^kn_i=n$, if $f_0$ is non-generic. Here, for two groups $G_1$ and $G_2$, we write $G_1\doteq G_2$ if a finite index subgroup of $G_1$ is isomorphic to a finite index subgroup of $G_2$ as Lie groups.

In the case when  the center of the algebraic
centralizer contains a higher rank abelian action $\R^k,k\ge 2$, under suitable assumptions on the centralizer, the first author \cite{Wang,Wang26} established that elements of the higher rank action also satisfy some partially hyperbolic properties; combined with local rigidity theory for higher-rank restrictions of diagonal actions, these results establish centralizer rigidity for these cases. This paper is a follow-up to these earlier works.

The main point of the present paper is to treat the very non-generic diagonal maps, when $\mathcal Z(f_0)\doteq \R\times \mathrm{SL}_{n_1}\R\times \mathrm{SL}_{n_2}\R$. In this case, the higher rank abelian mechanism is no
longer available.  

Our main new contribution is a method for extracting partial hyperbolicity
from this non-abelian centralizer structure, with the assumption that a $C^1$-perturbation $f$ has an isomorphic centralizer. We then upgrade
 measurable information obtained from Zimmer's cocycle superrigidity   to geometric data. This allows us to construct an element
\[
        g\in \mathcal Z(f)
\]
that is partially hyperbolic and whose stable, center, and unstable bundles
are \(C^0\)-close to those of an algebraic diagonal element \(g_0\).  Once
this element is constructed, the pair \(\langle f,g\rangle\) recovers the
coarse Lyapunov geometry needed to apply the geometric rigidity scheme for
higher-rank diagonal actions.

We now state the main result.

Let $\mathcal L$ be a Lie group, and fix a constant $K>1$.
  At a locally free point
$x_0$, endow $\mathfrak l$ with the norm pulled back from
$T_{x_0}(\mathcal L\cdot x_0)$ by the derivative of the orbit map, and
endow $\mathcal L$ with the corresponding left-invariant metric. We say that a smooth, locally free action $\sigma: \mathcal L\times X\to X$ is {\em $K$-regular} if there exists a point $x_0\in X$ such that the orbit map is injective on $B_{\mathcal L}(e,K^{-1})$ and
$\| \sigma(\cdot, x_0)\|_{C^3}+\| (\sigma(\cdot, x_0))^{-1}\|_{C^3}\le K$ on this ball, where the inverse is taken locally from a neighborhood of $x_0\in \mathcal L\cdot x_0$ to $B_{\mathcal L}(e,K^{-1})$.

\begin{thm}\label{cenrignongen}
   Let \(X = \mathrm{SL}_n\R/\Gamma\) where $\;n\ge 5$ and $\Gamma$ is a cocompact lattice. 
    Let $f_0$ be a non-trivial diagonal map, and let \(f\in \mathrm{Diff}^\infty(X)\) be a  \(C^1\)-small perturbation of \(f_0\). Then $\mathcal Z(f)$ is a Lie group and the action $\mathcal{Z}(f)\times X\to X$ is jointly smooth.
    
    Fix $K>1$. If in addition, \(f\in \mathrm{Diff}^\infty(X)\) is a \(C^1\)-small volume-preserving perturbation of \(f_0\) with
    \( \mathcal Z(f) \doteq \mathcal Z(f_0)\), $d_{C^2}(f,f_0)\le K$ and the action $\mathcal{Z}(f)\times X\to X$ is $K$-regular, then \(f\) is smoothly conjugate to a diagonal map.
\end{thm}

\begin{rmk}
    The assumption that the rank is at least 4 ensures that for any non-trivial diagonal map \(f_0\), one of two favorable situations occurs: either \(f_0\) possesses at least two distinct positive roots, or the algebraic centralizer \(\mathcal Z(f_0)\) itself contains a simple Lie group factor of rank at least 2. In the latter case, the rigidity can be proved using the higher-rank action of this subgroup. In fact, Theorem \ref{cenrignongen} extends to any diagonal matrix in homogeneous spaces $\mathrm{SL}_n\R/\Gamma,\,n\ge3$, except for two cases $f_0=diag(e^t,e^t,e^{-2t})$ and $f_0=diag(e^t,e^t,e^{-t},e^{-t})$.
\end{rmk}

\subsection{Historical remarks and related work}

The local rigidity theory of higher-rank algebraic actions has a long history,
beginning with the work of Katok--Spatzier \cite{KS} on higher-rank Anosov actions and
homogeneous actions. For Weyl chamber flows and their restrictions, a
geometric approach was developed by Damjanovi\'c--Katok \cite{DK} and later extended by
Vinhage \cite{Vinhage} and Vinhage--Wang \cite{VinWang}.  These works provide local rigidity results for
higher-rank restrictions of diagonal actions on homogeneous spaces, and form
the main rigidity background for the present paper.

Centralizers have also been studied from a different, more generic point of
view.  A classical question of Smale \cite{smale} asks whether a generic diffeomorphism has
trivial centralizer. Kopell \cite{Kopell} proved this on the circle in the $C^r$ topology, $r\geq 2$, and Bonatti-Crovisier-Wilkinson \cite{BCW} proved this in the
\(C^1\) topology.

For partially hyperbolic systems, centralizer rigidity was first proposed by Damjanovi\'c-Wilkinson-Xu in \cite{DWX}. Centralizer rigidity in tori and nilmanifolds was studied in \cite{DWX23},\cite{DWWX}, \cite{GSXZ}, \cite{Sandfeldt}. In semisimple homogeneous spaces and twisted product of semisimple homogeneous spaces and tori, \cite{Wang} and \cite{Wang26} studied centralizer rigidity of (semi-)generic elements of the Weyl chamber flow.

In \cite{DWWX}, Damjanovi\'c-Wilkinson-Wu-Xu conjectured,
roughly, that for perturbations of affine \(K\)-systems, \begin{enumerate}
 
    \item the centralizer should
be a Lie group;

\item higher-rank centralizers should
force smooth conjugacy to an affine model;

\item isomorphic centralizers should
force smooth conjugacy to an affine model.
 
\end{enumerate}  Theorem \ref{cenrignongen} verifies Conjectures 1 and 3 of \cite{DWWX} for non-trivial diagonal maps on compact quotients of
\(\mathrm{SL}_n(\mathbb R)\), \(n\ge5\), under the quantitative regularity
assumption on the centralizer action.  

 \subsection{Organization of the paper and outline of proof}\label{outl}

In Section \ref{prelim}, we recall some preliminaries on the splitting of the Lie algebra, partial hyperbolicity, smoothness and existence of holonomies, and some Pesin theory. 

Then we proceed with the proof of Theorem \ref{cenrignongen}.

We note that elements in the centralizer of $f$ preserve the center, stable and unstable foliations of $f$.
In Section \ref{sec:cla}, we prove the Lie group property of the centralizer and establish dynamics of the centralizer along the center leaves of $f$. 
 First, we apply results in \cite{Witte} and \cite{Wang} to establish that the center-fixing centralizer $\CZ(f):=\{g\in \mathcal Z(f):g(x)\in \W^c_f(x), \, \forall x\in X\}$ is a finite index subgroup of $\mathcal Z(f)$, and $\CZ(f)$ is a Lie group. Furthermore, results from \cite{Wang26} show that the joint action $\CZ(f)\times \W^c_f(x)\to \W^c_f(x)$ is free, proper, jointly smooth, and ``projectively-$C^0$-close" to the affine action $\CZ(f_0)$ on $\W^c_{f_0}$. The projective estimate from \cite{Wang26}, together with the $K$-regularity hypothesis, then gives an adapted isomorphism between the
Lie algebras of the perturbed and algebraic centralizers. 
 This allows us to associate to a suitable algebraic element $g_0\in \CZ(f_0)$ an element
$g\in\CZ(f)$ with the same hyperbolic behavior along the center leaves.
The remainder of the section establishes the three estimates later used
outside the center: a leaf conjugacy for the common neutral foliation, a
H\"older exponent tending to one with the perturbation, and an expansion to macroscale behavior along the stable and unstable direction, which rules out a zero exponent of $Dg$ restricted to $E_f^s$ or $E_f^u$.

Section~\ref{sec:gph} proves that sufficiently many elements of $\CZ(f)$ are partially hyperbolic. The proof has
four stages. 
In Section \ref{sec:1exp}, we use Zimmer cocycle superrigidity and a low-dimensional
representation argument to show that $Df|_{E_f^s}$ has a single Lyapunov
exponent. 
In Section \ref{sec:gerg}, we first promote measurable stable, neutral, and
unstable subspaces for a centralizer element $g$ to continuous
subbundles. Choose $g\in G_1$ that has no central component on $E^s_f$ and $E^u_f$. We show that any set which is bi-essentially saturated by the Pesin stable and unstable foliations of $g$ is bi-essentially saturated by the stable and unstable foliations of $f$. Then applying Theorem \ref{burnswikerg} on $f$, this shows that $g$ is ergodic. 
Section \ref{subsec:ph} then shows that
the derivative cocycle is reduced, modulo its scalar sign, to a constant
cocycle and Sections~\ref{sec:c0closebund} and \ref{subsec:proofgph}  show that the holonomies of $g$ and $g_0$ are $C^0$-close.

Section~\ref{sec:main-proof} applies the rigidity theorem from \cite{Wang26} to
the action generated by $f$ and $g$, completing the proof of
Theorem~\ref{cenrignongen}. 

Appendix~\ref{app:leafwise-pesin-proof} proves transverse absolute
continuity of Pesin stable and unstable manifolds of $g$ inside $\mathcal W_f^s$ and $\mathcal W_f^u$.

Appendix~\ref{app:topological-perturbations}
recalls the
higher-rank rigidity theorem used in Section~5, and verifies its
hypotheses are satisfied for the action constructed in
Section \ref{sec:gph}.

\subsection{Further questions}
Theorem \ref{cenrignongen} is a first step toward the
centralizer rigidity conjectures stated above for the simplest semisimple case of compact quotients of
\(\mathrm{SL}_n(\mathbb R)\), \(n\ge5\)
under two additional assumptions:
the perturbed centralizer is virtually isomorphic to the algebraic
centralizer, and its action satisfies a uniform local regularity condition.
It is natural to ask to what extent these hypotheses can be weakened or
removed, and whether similar results hold in more generality.

1. Can the \(K\)-regularity assumption in Theorem \ref{cenrignongen} be removed?
 In Theorem \ref{cenrignongen}, \(K\)-regularity is used to obtain
uniform quantitative control of the action near the identity.  One expects
that such control should follow from the structure of the centralizer itself,
at least after passing to a finite-index subgroup or after choosing suitable
coordinates on center leaves.

2. Can the virtual equality of centralizers  be replaced
by a more intrinsic largeness assumption? For example, can one prove centralizer rigidity using only a dimension count as in Theorem 1.1 of \cite{Wang} in the generic case, or a no-rank-one factor condition as in Conjecture 2 of \cite{DWWX}?

3. The remaining low-rank cases also appear interesting.  
What happens for the exceptional low-rank one-root cases 
$
        \operatorname{diag}(e^t,e^t,e^{-2t})\in \mathrm{SL}_3(\mathbb R)
$
and
$
        \operatorname{diag}(e^t,e^t,e^{-t},e^{-t})
        \in \mathrm{SL}_4(\mathbb R)?
$

4. Another natural direction is to go beyond diagonal maps on
\(\mathrm{SL}_n(\mathbb R)/\Gamma\).  For example, can one treat the case of $f_0=L_a: \mathrm{SL}_{2n}\R/\Gamma\to \mathrm{SL}_{2n}\R/\Gamma$ , where $a=\begin{pmatrix}
    eR_1&0\\0& e^{-1}R_2
\end{pmatrix}$  and $R_1,R_2\in \mathrm{SO}(n)$?  In this case, the center dimension will be larger than the centralizer, so different methods are needed to find close-by affine models for elements of the centralizer.

5. Can we extend the rigidity of very non-generic elements to elements of the Weyl chamber flow in other semisimple homogeneous spaces, or twisted products? The methods of this paper extend to many $\R$-split simple homogeneous spaces, but semisimple and non-$\R$-split cases warrant new ideas and mechanisms.

\subsection{Acknowledgements}

The authors thank Kurt Vinhage for explaining his work and for many
useful discussions, particularly concerning the application of Zimmer's
cocycle rigidity to the higher-rank centralizer case treated in
Section~\ref{sec:gph}. They also thank Aaron Brown, David Fisher,
Homin Lee, James Marshall Reber, Sven Sandfeldt and Disheng Xu for
helpful discussions and comments.

The authors used generative AI to assist in revising
the manuscript and in drafting portions of Appendices~A and~B.
All AI-assisted material was reviewed, edited, and mathematically
verified by the authors. The authors assume responsibility for all
content.

\section{Setting and preliminaries}\label{prelim}
In this section, we provide a brief introduction to the diagonal action, the theory of partially hyperbolic dynamical systems, the Mather spectrum and narrow-band condition, Pesin theory, and some results on holonomies of cocycles.

\subsection{Partial hyperbolicity and regularity}\label{prel:ph}
We recall the definitions of dominated splitting and partial hyperbolicity, and introduce the concept of leaf conjugacy, drawing on results from \cite{HPS}.

\begin{defi}[Dominated Splitting]
Let $M$ be a Riemannian manifold, and let $f:M\to M$ be a diffeomorphism. A \emph{dominated splitting} of $f$ is a $Df$-invariant decomposition $TM=E^1\oplus E^2\oplus \cdots \oplus E^k$ satisfying:
\begin{itemize}
\item For each $1\le i\le k$, $Df(E^i(x))=E^i(f(x))$ for any $x\in M$.
\item There exist constants $C>0$ and $0 < \lambda < 1$ such that for any integer $n \ge 1$, $x\in M$, and for any unit vectors $v\in E^i(x), u\in E^{i+1}(x)$,
$$
\|Df^n(x)v\|\le C\lambda^n \|Df^n(x)u\|,
$$
for $1\le i\le k-1$.
\end{itemize}
\end{defi}

\begin{defi}[Partially Hyperbolic Diffeomorphism]\label{defi:ph}
A diffeomorphism $f:M\to M$ of a Riemannian manifold $M$ is \emph{partially hyperbolic} if it admits a \emph{dominated splitting} $TM=E^s\oplus E^c\oplus E^u$, and there exist constants $C>0$ and $0<\mu_u,\mu_s<1$ such that for any integer $n \ge 1$, and any $x\in M$:
$$
\|Df^n(x)v\|\le C\mu_s^n\|v\|,\qquad \text{for any } v\in E^s(x),
$$
and
$$
\|Df^{-n}(x)u\|\le C\mu_u^n\|u\|,\qquad \text{for any } u\in E^u(x).
$$
In this paper, we assume that both the stable and unstable bundles $E^s$ and $E^u$ are non-trivial (i.e., non-zero dimensional).
\end{defi}

For a partially hyperbolic diffeomorphism, the stable and unstable bundles $E^s$ and $E^u$ are uniquely integrable to $f$-invariant stable and unstable foliations $\mathcal{W}^s$ and $\mathcal{W}^u$, respectively. The leaves of these foliations are globally defined by:
$$
\mathcal{W}^s(x)=\{y\in M:\; \exists\; C>0,\; d(f^n(x),f^n(y))\le C\mu_s^n \,\; \forall\; n>0 \},
$$
and
$$
\mathcal{W}^u(x)=\{y\in M: \;\exists\; C>0, \;d(f^n(x),f^n(y))\le C\mu_u^{-n} , \;\forall\; n<0 \}.
$$
Unlike $E^s$ and $E^u$, the center bundle $E^c$ is generally not integrable to a foliation. A partially hyperbolic diffeomorphism is said to be normally hyperbolic if there exists an $f$-invariant center foliation $\mathcal{W}^c$ such that $T\mathcal{W}^c = E^c$.

\begin{defi}[Normally Hyperbolic]
Let $f:M\to M$ be a diffeomorphism and let $\mathcal{F}$ be a foliation in $M$. For $r\ge 1$, we say $(f,\mathcal F)$ is \emph{$r$-normally hyperbolic}, if $f$ preserves $\mathcal{F}$, $f$ is partially hyperbolic with respect to $E^c=T\mathcal{F}$, and there exists an integer $k \ge 1$ such that: 
$$
\sup_{x\in M}\|D_xf^k|_{E^s}\|\cdot \|(D_xf^k|_{E^c})^{-1}\|^r<1
$$
and
$$
\sup_{x\in M}\|(D_xf^k|_{E^u})^{-1}\|\cdot \|D_xf^{k}|_{E^c}\|^r<1.
$$

We say that $(f,\mathcal F)$ is \emph{normally hyperbolic} if it is $1$-normally hyperbolic.
\end{defi}

Note that $r$-normal hyperbolicity is a $C^1$-open condition.

For a normally hyperbolic diffeomorphism $(f,\mathcal{F})$, heuristically, the induced action of $f$ on the leaf space $M/\mathcal{F}$ can be viewed as an Anosov system. Consequently, one might expect structural stability of $f$ up to the leaves of $\mathcal{F}$. Under suitable conditions, perturbations of $f$ are $C^0$-conjugate to $f$ modulo the leaves of $\mathcal{F}$; we call such a perturbation \emph{leaf conjugate} to $f$. Now we define leaf conjugacy precisely.

\begin{defi}[Leaf Conjugacy]
Suppose $(f,\mathcal{F}_f)$ and $(g,\mathcal{F}_g)$ are two diffeomorphisms of $M$ with invariant foliations $\mathcal{F}_f$ and $\mathcal{F}_g$, respectively. Then $(f,\mathcal{F}_f)$ and $(g,\mathcal{F}_g)$ are said to be \emph{leaf conjugate} via a leaf conjugacy $h\in \mathrm{Homeo}(M)$ if
$$
h(\mathcal F_f(x))=\mathcal F_g(h(x))\quad \text{and}\quad h(f(\mathcal F_f(x)))=g(h(\mathcal F_f(x))), \quad \text{for all } x\in M.
$$
\end{defi}

A notion slightly stronger than normal hyperbolicity is dynamical coherence.

\begin{defi}[Dynamically Coherent]
A partially hyperbolic diffeomorphism $f$ is said to be \emph{dynamically coherent} if there exist two $f$-invariant foliations $\mathcal{W}^{cs}$ and $\mathcal{W}^{cu}$ that are tangent to the continuous subbundles $E^{cs}=E^c\oplus E^s$ and $E^{cu}=E^c\oplus E^u$, respectively.

If $f$ is dynamically coherent, the center foliation of $f$, denoted $\mathcal{W}^c$, is defined as the intersection $\mathcal{W}^{cs}\cap \mathcal{W}^{cu}$.
\end{defi}

By definition, $\mathcal{W}^c$ is an $f$-invariant foliation tangent to $E^c$ at every point, which implies that $(f,\mathcal{W}^c)$ is normally hyperbolic. It is important to note that dynamical coherence does not require $E^{cs}$, $E^{cu}$, or $E^c$ to be uniquely integrable to their respective foliations. For a more in-depth discussion, see \cite{BWcoh}.

Now we introduce a central result about leafwise structural stability for normally hyperbolic diffeomorphisms. A more general statement and proof can be found in \cite{BWcoh} and \cite{PSWholrev}.

\begin{thm}[Leafwise structural stability]\label{leafconj}

Let  $f\in\mathrm{Diff}^r(M)$, $r\ge2$, preserve a uniformly $C^r$
foliation $\mathcal F$ and be normally hyperbolic and plaque expansive for
$\mathcal F$.  Every sufficiently small $C^1$ perturbation $f_1\in\mathrm{Diff}^r(M)$ has a
nearby invariant foliation $\mathcal F_1$ and there is a leaf conjugacy
$\phi:(f,\mathcal F)\to(f_1,\mathcal F_1)$ which is $C^0$-close to the
identity and bi-H\"older.  If $f$ is $r$-normally hyperbolic, the restrictions of $\phi$ to the leaves of
$\mathcal F$ are uniformly $C^r$ and converge to the identity in the
leafwise $C^1$ topology as $f_1\to f$ in $C^1$.

The stable and unstable bundles of $f_1$ are uniquely integrable, and their
leaves are $C^r$-smooth.

\end{thm}

\subsubsection{Accessibility}\mbox{}\par
Now we recall some definitions and facts about $su$-holonomies of partially hyperbolic diffeomorphisms.

An \emph{$su$-path} for $f$ is a piecewise $C^1$ curve
$\gamma:[0,1]\to M$  whose legs lie alternately in stable or unstable
 leaves.  More precisely, there  is an integer  $k\ge1$, together with times
$0=t_0<t_1<\cdots<t_k=1$ such that each
$\gamma([t_i,t_{i+1}])$  lies in either
 $\mathcal W^s(\gamma(t_i))$ or  $\mathcal W^u(\gamma(t_i))$.   We call such
a path $k$-legged and sometimes record only its vertices,
$[\gamma(t_0),\gamma(t_1),\ldots,\gamma(t_k)]$.

\begin{defi}[Accessibility]
A partially hyperbolic diffeomorphism $f$ is said to be \emph{accessible} if for any two points $x,y\in M$, there exists an $su$-path from $x$ to $y$.
\end{defi}

For partially hyperbolic systems with center dimension $1$ or $2$, accessibility is known to be an open property (see \cite{Didier} and \cite{AV}). Stable accessibility (meaning accessibility persists under $C^1$-small perturbations) has also been shown to be $C^r$-dense for $r\ge 1$ for center dimension $1$ in \cite{RHU}, and $C^1$-dense for arbitrary center dimension in \cite{DW}. Furthermore, if the center foliation of a partially hyperbolic diffeomorphism $f$ is smooth, then accessibility is $C^1$-open around $f$ by Proposition 1.4 of \cite{GPS}.

\subsubsection{Center-bunching and ergodicity}
We now recall the property of center-bunching and its relation with ergodicity of a partially hyperbolic system.

\begin{defi}[Center Bunching]

We say that a partially hyperbolic diffeomorphism $f:M\to M$ is
center $r$-bunched if $f$ is $r$-normally hyperbolic and there is
an integer $k\ge1$ such that
\begin{align*}
\sup_{x\in M} \|D_xf^k|_{E^s}\|\cdot \|(D_xf^k|_{E^c})^{-1}\|\cdot \|D_xf^k|_{E^c}\|^r &<1, \text{ and } \\
\sup_{x\in M} \|(D_xf^k|_{E^u})^{-1}\|\cdot \|(D_xf^k|_{E^c})^{-1}\|^r\cdot \|D_xf^k|_{E^c}\| &<1.
\end{align*}
We say that $f$ is \emph{center bunched} if it is center $1$-bunched.

\end{defi}

The following theorem of Burns and Wilkinson \cite{BWerg} shows that partial hyperbolicity, center bunching, and accessibility together imply ergodicity.

\begin{thm}[Theorem 0.1, \cite{BWerg}]\label{accerg}
Let $f$ be a $C^2$, partially hyperbolic, volume-preserving diffeomorphism that is center bunched and accessible. Then $f$ is ergodic with respect to the Lebesgue measure.
\end{thm}

The main ingredient of the proof in \cite{BWerg} is the following result, which plays an analogous role to the absolute continuity of stable and unstable foliations in the Hopf argument.

Let $\mathcal F$ be a foliation in a Riemannian manifold $M$. We say that a measurable set $A$ is \emph{$\mathcal F$-saturated} if for any $x\in A$, the entire leaf $\mathcal{F}(x)$ passing through $x$ is contained in $A$. We say that a measurable set $A$ is \emph{essentially $\mathcal F$-saturated} if there exists an $\mathcal F$-saturated set $B$ such that $m(A\triangle B)=0$.

\begin{thm}[Theorem 5.1, \cite{BWerg}]\label{burnswikerg}
Let $f$ be a $C^2$, partially hyperbolic, center-bunched diffeomorphism in a manifold $M$. Let $A\subset M$ be a measurable set that is both essentially $\mathcal{W}^s_f$-saturated and essentially $\mathcal{W}^u_f$-saturated. Then the set of Lebesgue density points of $A$ is both $\mathcal{W}^s_f$-saturated and $\mathcal{W}^u_f$-saturated.
\end{thm}

\subsection{Mather spectrum and the narrow-band condition}
\label{subsec:mather-spectrum}

We recall the spectral term used below.  If a bundle map
$A:E\to E$ over a diffeomorphism $f$ uniformly contracts $E$, its
\emph{Mather spectrum} is the spectrum of the operator
\[
 A_*:\Gamma(E)\to\Gamma(E),\qquad
 (A_*v)(x)=A_{f^{-1}x}v(f^{-1}x).
\]
Here $\Gamma(E)$ is the Banach space of continuous sections with the
supremum norm.
It is a finite union of closed annuli with logarithmic endpoints
\[
 \lambda_1\le\mu_1<\lambda_2\le\mu_2<\cdots
 <\lambda_l\le\mu_l<0.
\]
The spectrum is \emph{narrow band} if
$\mu_i+\mu_l\le\lambda_i$ for $1\le i\le l$; its critical regularity is
$s(A)=\lambda_1/\mu_l$.  For an expanding bundle this terminology is
applied to the inverse bundle map.  

We say that a diffeomorphism $f$ satisfies the narrow-band condition if
$Df|_{E^s_f}$ and $Df^{-1}|_{E^u_f}$ have narrow-band spectrum.

\subsection{The diagonal action}\label{sec:semisimple}

Let $G=\mathrm{SL}_n\R$. Then the Lie algebra is
\[
\mathfrak{g}=\mathfrak{sl}_n\R=\{A\in M_{n\times n}(\R): \mathrm{tr}\,A=0\}.
\]

A Cartan subalgebra of $\mathfrak{g}$ is conjugate to the set of diagonal matrices
\[
\mathfrak{h}\simeq \left\{
\begin{pmatrix}
 t_1 & &&\\
 & t_2&&\\
 && \ddots&\\
 &&& t_n
\end{pmatrix}:\; t_i\in \R, \,\sum_{i=1}^n t_i=0
\right\}.
\]

The adjoint action of $\mathfrak{h}$ induces the splitting $\mathfrak{g}=\mathfrak h\oplus_{i\neq j}\mathfrak{g}_{ij}$.

Here, the eigenspaces of the adjoint action of $\mathfrak{h}$ are
\[
\mathfrak{g}_{ij}=\R e_{ij},\qquad 1\le i\neq j\le n,
\]
where $e_{ij}$ is the matrix with $1$ in the $(i,j)$-entry and zeros elsewhere. The corresponding eigenvalues $\lambda_{ij}$ are given by
\[
\lambda_{ij}
\big(\begin{pmatrix}
 t_1 & &&\\
 & t_2&&\\
 && \ddots&\\
 &&& t_n
\end{pmatrix}\big)
=t_i-t_j.
\] 

Let $f_0$ be a diagonal map on $X=\mathrm{SL}_n\R/\Gamma$, then by identifying tangent spaces at different points by right translation, the differential $Df_0$ is identified with the adjoint action $Ad(f_0):\mathfrak{g}\to \mathfrak{g}$. This induces the following dominated splitting   \begin{equation}\label{eq:f_0split}
    \mathfrak g=\mathfrak g^s_{f_0}\oplus \mathfrak  g^c_{f_0}\oplus \mathfrak  g^u_{f_0},
\end{equation} where $$\mathfrak g^s_{f_0}=\bigoplus_{\lambda_{ij}(f_0)<0}\,\mathfrak{g}_{ij},\;  \mathfrak g^c_{f_0}=\mathfrak h\oplus_{\lambda_{ij}(f_0)=0}\,\mathfrak{g}_{ij},\;\mathfrak g^u_{f_0}=\bigoplus_{\lambda_{ij}(f_0)>0}\,\mathfrak{g}_{ij}.$$ 

We say that $f_0$ has only one non-zero root, if there exists $\lambda\neq 0$, such that $\lambda_{ij}(f_0)\in \{-\lambda,0,\lambda\}$ for any $1\le i\neq j\le n$. We say that $f_0$ has at least two distinct positive roots if $\{\lambda_{ij}(f_0)\}$ takes at least two different positive values.

Equation \ref{eq:f_0split} shows that $f_0$ is partially hyperbolic. Furthermore, $\|D_xf_0^k|_{E^c_{f_0}}\|=1$ for any $k\in \Z$, so $f_0$ is also $r$-normally hyperbolic and center-bunched. Since the bundles are uniquely integrable, $f_0$ is also dynamically coherent.

Any sufficiently small $C^1$-perturbation $f$ of a diagonal map $f_0$ inherits these dynamical properties.

\begin{prop}\label{fparhyp}
Let $f_0$ be a non-trivial diagonal map. Fix an integer $r\ge2$.
Then there exist $N\in\mathbb N$ and $C>0$ such that $f_0$ is
$r$-normally hyperbolic with respect to its center foliation. For any
sufficiently small $C^1$-perturbation $f \in \mathrm{Diff}^r(X)$ of $f_0$:
\begin{enumerate}
    \item $f$ is a partially hyperbolic diffeomorphism with non-trivial stable and unstable bundles.
    \item The stable and unstable bundles $E^u_f,E^s_f$ are uniquely integrable.
    \item $f$ is dynamically coherent, $r$-normally hyperbolic,
    center $r$-bunched, and satisfies the narrow-band
    condition.
    \item $f$ is accessible and any  two
points of $X$ can be joined by an $f$-$su$-path  with at most $N$ legs and
 total leafwise length  at most $C$.

    \item There exists a bi-H\"older leaf conjugacy between $f_0$ and $f$ that is uniformly $C^r$ in the center leaves.
    \item If $f$ is volume-preserving, then $f$ is ergodic with respect to the Lebesgue measure. Furthermore, any $g\in \mathcal Z(f)$ is volume-preserving.
\end{enumerate}
\end{prop}
\begin{proof}
As established in Section \ref{sec:semisimple}, for a non-trivial diagonal map $f_0$, the action of $Ad(f_0)$ has both positive and negative eigenvalues, thus $f_0$ possesses non-trivial stable and unstable bundles. Therefore, by the splitting of the Lie algebra, $f_0$ is $r$-normally hyperbolic with respect to the foliation $\mathcal{W}^c_{f_0}$. Since $f_0$ is plaque expansive, applying Theorem \ref{leafconj} directly yields partial hyperbolicity, unique integrability, dynamical coherence, $r$-normal hyperbolicity, and the asserted leaf conjugacy.  Since $Df_0|_{E^c_{f_0}}$ is isometric, the center $r$-bunching inequalities are strict for every fixed $r$.  Moreover, the Mather bands of $Df_0|_{E^s_{f_0}}$ and $Df_0^{-1}|_{E^u_{f_0}}$ are the singleton negative root values, so the narrow-band inequalities are also strict.  Both properties therefore persist under sufficiently small $C^1$ perturbations.

For property (4), the stable and unstable root subgroups of
$f_0$ generate $G$: their Lie algebras contain the nonzero root spaces of
$f_0$, and their brackets generate the zero-root spaces as well.  Hence a
finite product of local stable and unstable plaques has an endpoint map
whose image contains a neighborhood of the identity.  These finitely
many local paths persist, with uniform plaque lengths, for every
sufficiently small $C^1$ perturbation.  A finite cover of the compact
connected manifold $X$, followed by concatenation, therefore gives uniform
bounds $N$ on the number of legs and $C$ on their total leafwise length.
This is also the quantitative form of the accessibility argument in
Proposition~1.4 of \cite{GPS}.

For property (6), we use that accessibility is open around $f_0$, and center-bunching is an open condition. Therefore, applying Theorem \ref{accerg} shows that $f$ is ergodic with respect to the Lebesgue measure $m$. For any $g\in \mathcal Z(f)$, $g_*m$ is an absolutely continuous $f$-invariant probability measure, so the Radon-Nikodym derivative is an $f$-invariant function. Since $f$ is ergodic, this shows that $g_*m=m$, i.e. $g$ is volume-preserving.
\end{proof}

 \subsection{Holonomy of a foliation}

Let $M$ be a smooth manifold. Let $\mathcal{F}$ be a foliation in $M$ which is a decomposition of $M$ into a collection of disjoint, connected, $C^1$-embedded submanifolds, called \emph{leaves}. For each $x \in M$, we denote the unique leaf containing $x$ by $\mathcal{F}(x)$. 

For a leaf $\mathcal{F}(x)$, we denote by $\mathcal{F}(x,R) = \{ y \in \mathcal{F}(x) : d_{\mathcal{F}(x)}(x,y) < R \}$ a bounded portion of the leaf. A \emph{foliation box} of $\mathcal{F}$ is a foliated region $B \subset M$ that is locally homeomorphic to a product space, such that the leaves of $\mathcal{F}$ within $B$ correspond to one of the product factors.

A \emph{transversal} of $\mathcal{F}$ is an embedded disk $S \subset M$ whose tangent space is transverse to \(T\mathcal F\) at every point
of \(S\).

For two transversals $S_1, S_2$ and a foliation box $B$ that intersects both, the \emph{local holonomy map} $h_{S_1, S_2}^{\mathcal{F}, B}: S_1 \cap B \to S_2$ is defined by following the plaques (leaf segments) within $B$. For general transversals $S_1, S_2$, the \emph{holonomy map} $h_{S_1, S_2}^{\mathcal{F}}: \mathrm{Dom}(h) \to S_2$ is defined on the subset $\mathrm{Dom}(h) \subset S_1$ consisting of points whose leaves also intersect $S_2$. This map is constructed by composing a finite chain of local holonomies across foliation boxes covering a leafwise path between $S_1$ and $S_2$; the resulting map depends only on the homotopy class of the path within the leaf.

\subsubsection{Smoothness of $su$-holonomy}\label{sec:suhol}
In this paper, we mainly consider stable and unstable holonomies, with the center leaves serving as sections.

Suppose $f$ is a dynamically coherent, normally hyperbolic diffeomorphism with center foliation $\mathcal{W}^c$. The stable holonomy maps between center leaves are defined by following stable leaves. Specifically, for $p \in M$ and $q \in \mathcal{W}^s_f(p)$, consider a local product structure where stable leaves are transverse to center leaves. Then there exist sufficiently small neighborhoods $U(p) \subset \mathcal{W}^c_f(p)$ of $p$ and $U(q) \subset \mathcal{W}^c_f(q)$ of $q$, and a constant $R > 0$ (e.g., $R = 2 d_{\mathcal{W}^s_f}(p,q)$), such that for every $x \in U(p)$, the stable leaf segment $\mathcal{W}^s_f(x, R) = \{ y \in \mathcal{W}^s_f(x) : d_{\mathcal{W}^s_f}(x,y) < R \}$ intersects $U(q)$ in exactly one point. The \emph{local stable holonomy} is then defined by
$$
h^s_{f,p,q}: U(p) \to U(q), \qquad x \mapsto U(q) \cap \mathcal{W}^s_f(x,R).
$$
Local unstable holonomies are defined analogously. These stable and unstable holonomies along center leaves are known to possess significant smoothness properties.

\begin{thm}[Theorem B, \cite{PSW}; Theorem 1.1, Section 3.6 \cite{Saghin}]\label{c1hol}
Let $f \in \mathrm{Diff}^{r+1}(M)$, $r \ge 1$, be normally hyperbolic
 and center $r$-bunched.  Then local stable and unstable  holonomies
between center leaves are uniformly $C^r$. These local holonomies and their first
derivatives depend continuously on $f$ in a $C^{r+1}$ neighborhood, under the $C^1$ topology and continuously on the
endpoints of a bounded leg.  
\end{thm}

We may also define the \emph{holonomy of an $su$-path}.   Let
$\gamma=[x_0,x_1,\ldots,x_k]$ and, for $0\le i<k$, choose
 $\star_i\in\{s,u\}$ so that
 $x_{i+1}\in\mathcal W_f^{\star_i}(x_i)$.   The local $su$-holonomy near
$x_0$ is the composition
 $$
h_\gamma^f = h^{\star_{k-1}}_{f,x_{k-1},x_k}
 \circ h^{\star_{k-2}}_{f,x_{k-2},x_{k-1}}
 \circ \cdots \circ h^{\star_0}_{f,x_0,x_1}:
 U(x_0) \longrightarrow U(x_k),
$$
where each $U(x_i)\subset\mathcal W_f^c(x_i)$ is a sufficiently small
neighborhood of $x_i$.

Let $f'$ be a sufficiently small $C^1$  perturbation of $f$, and let
$\phi$ be a leaf conjugacy from $f$ to $f'$.   A bounded $f$--$su$-path
$\gamma=[x_0,\ldots,x_k]$ determines an $f'$--$su$-path only after the
usual center adjustment: starting with $x'_0=\phi(x_0)$, define
$x'_{i+1}$ as the intersection of the corresponding stable/unstable leaf of $f^\prime$
through $x'_i$ with the center leaf $\phi(\mathcal W_f^c(x_{i+1}))$.
Thus $x'_i$ and $\phi(x_i)$ lie on the same center leaf, but need not be
the same point. Theorem~\ref{c1hol}, together with the leafwise
smoothness of $\phi$, shows that the holonomy of this adjusted path is
smooth and converges in the leafwise $C^1$ topology to the model
holonomy if the numbers and lengths of the legs remain bounded.

\subsubsection{Existence of global holonomy}
We say that $f$ has \emph{global holonomy} on
$\mathcal W_f^c(x_0)$ if, for every $su$-path $\gamma$ from $x_0$ to
$x_k$, the local map $h_\gamma^f$ extends uniquely to the whole center
leaf. Equivalently, there is an open cover $\{U_j\}$ of
$\mathcal W_f^c(x_0)$ and compatible local holonomies
 \[
h_{\gamma_j}^f:U_j\longrightarrow\mathcal W_f^c(x_k)
\]
whose restrictions agree on overlaps. They therefore define a global
smooth map
$h_\gamma^f:\mathcal W_f^c(x_0)\to\mathcal W_f^c(x_k)$.

In the case where $f$ is a $C^1$-perturbation of a partially hyperbolic affine model, the following result shows that $f$ always has global holonomy in the universal cover.

\begin{prop}[Proposition 2.8 \cite{Wang}]\label{proP:globhol}

Let $f_0$ be a partially hyperbolic affine diffeomorphism  of a compact
homogeneous space $G/\Gamma$.   Every sufficiently small $C^1$
 perturbation $f$ of $f_0$ has global holonomy in the universal cover.  In
particular, $f$ has global holonomy  on $\mathcal W_f^c(x_0)$ whenever that
center leaf is simply connected.

\end{prop}

\subsection{Pesin theory}
We  recall the  Oseledets theorem  and the specific Pesin-theory facts used
below.

\subsubsection{Oseledets theorem and Pesin blocks}\mbox{}\par

Let $(X,\mu)$ be a probability space, let
$f:X\to X$ be an invertible measure-preserving transformation, and let
$\pi:E\to X$ be a measurable Euclidean vector bundle of dimension $d$.
A measurable bundle map
\[
A:E\longrightarrow E
\]
is called an invertible linear cocycle over $f$ if each fiber map
\[
A_x:E_x\longrightarrow E_{f(x)}
\]
is a linear isomorphism.  For $n\ge1$, write
\[
A_x^n
 =
A_{f^{n-1}(x)}\circ\cdots\circ A_x,
\qquad
A_x^{-n}
 =
\bigl(A_{f^{-n}(x)}^n\bigr)^{-1}.
\]

\begin{thm}[Oseledets multiplicative ergodic theorem {{\cite[Theorems~3.12 and~4.1]{introdyn}}}]\label{Oseled}
Let $A:E\to E$ be an invertible measurable linear cocycle over an
invertible measure-preserving transformation
$f:(X,\mu)\to(X,\mu)$. Assume that
\[
\log^+\|A_x\|,\qquad \log^+\|A_x^{-1}\|
\]
belong to $L^1(\mu)$. Then there exists an $f$-invariant set
$\mathcal R_A\subset X$ of full measure such that, for every
$x\in\mathcal R_A$, there are numbers
\[
\chi_1(x)<\cdots<\chi_{m(x)}(x)
\]
and a direct-sum decomposition
\[
E_x
 =
H_A^1(x)\oplus\cdots\oplus H_A^{m(x)}(x)
\]
with the following properties:
\[
A_xH_A^i(x)=H_A^i(f(x)),
\]
and, for every $0\ne v\in H_A^i(x)$,
\[
\lim_{n\to\pm\infty}
\frac1n\log\|A_x^nv\|
 =
\chi_i(x).
\]
The functions $\chi_i(x)$, their multiplicities, and $m(x)$ are
$f$-invariant. In particular, if $f$ is ergodic, they are constant
almost everywhere.
\end{thm}

A point $x\in\mathcal R_A$ is called an
\emph{Oseledets-regular point}, or simply a \emph{regular point} when
the cocycle is understood.

We now state Pesin's theorem, which shows in some tempered blocks, the cocycle can be straightened to a block-diagonal form with the prescribed growth given by the Lyapunov exponents.

For a linear isomorphism $L$, write \[ m(L)=\|L^{-1}\|^{-1}. \] 

\begin{thm}[Pesin blocks {\cite[Theorem~4.13]{introdyn}}]\label{thm:pesin-blocks} Let $X$ be a compact metric space, let $f:X\to X$ be a homeomorphism, let $\pi:E\to X$ be a continuous Euclidean vector bundle, and let $A:E\to E$ be a continuous bundle automorphism covering $f$. Fix $\varepsilon>0$. For every regular point $x$, there is a linear isomorphism \[ C_\varepsilon(x): \mathbb R^{d_1(x)}\oplus\cdots\oplus \mathbb R^{d_{m(x)}(x)} \longrightarrow E_x \] which maps $\mathbb R^{d_i(x)}$ onto $H_A^i(x)$ and for which \[ B_\varepsilon(x) := C_\varepsilon(f(x))^{-1}A_xC_\varepsilon(x) \] has the block-diagonal form \[ B_\varepsilon(x) = \begin{pmatrix} B_\varepsilon^1(x) & & 0\\ &\ddots&\\ 0&&B_\varepsilon^{m(x)}(x) \end{pmatrix}. \] For every $v\in\mathbb R^{d_i(x)}$, \[ e^{\chi_i(x)-\varepsilon}\|v\| \le \|B_\varepsilon^i(x)v\| \le e^{\chi_i(x)+\varepsilon}\|v\|. \] Equivalently, \[ m\bigl(B_\varepsilon^i(x)\bigr) \ge e^{\chi_i(x)-\varepsilon}, \qquad \bigl\|B_\varepsilon^i(x)\bigr\| \le e^{\chi_i(x)+\varepsilon}. \] The coordinate changes are tempered: there is a measurable function $K_\varepsilon:\mathcal R_A\to[1,\infty)$ such that \[ \|C_\varepsilon(x)\|, \ \|C_\varepsilon(x)^{-1}\| \le K_\varepsilon(x) \] and \[ K_\varepsilon(f^n(x)) \le K_\varepsilon(x)e^{\varepsilon|n|} \qquad \text{for every }n\in\mathbb Z. \] Moreover, the regular set admits a countable exhaustion by compact sets \[ \mathcal R_A = \bigcup_{j\ge1}\Lambda_j \quad\text{modulo a null set}, \] called \emph{Pesin blocks}, such that on each $\Lambda_j$: \begin{enumerate} \item the number of Lyapunov exponents and their multiplicities are constant; \item the functions $\chi_i(x)$, the Oseledets subspaces $H_A^i(x)$, and the Lyapunov coordinate changes $C_\varepsilon(x)$ are continuous; \item there is a constant $K_j\ge1$ such that \[ \|C_\varepsilon(x)\|, \ \|C_\varepsilon(x)^{-1}\| \le K_j \qquad\text{for every }x\in\Lambda_j. \] \end{enumerate} 
\end{thm}

\subsubsection{Stable manifolds and absolute continuity}\mbox{}\par

Based on the Oseledets splitting of a diffeomorphism $f$, we define the subbundles corresponding to negative and positive exponents:
\[
E^-_f(x)=\bigoplus_{\chi_i(x)<0} H^i_f(x), \qquad
E^+_f(x)=\bigoplus_{\chi_i(x)>0} H^i_f(x).
\]
For $\mu$-almost every regular point $x$, Pesin theory associates to
$E^-_f(x)$ and $E^+_f(x)$ local stable and unstable manifolds tangent
to these subspaces at $x$.

The \emph{Pesin stable manifold} of $x$ is defined as
\[
\mathcal{W}^-_f(x)=\left\{y\in M : \limsup_{n\to \infty}\frac{\log d(f^n(x),f^n(y))}{n}<0\right\},
\]
and the \emph{Pesin unstable manifold} of $x$ as
\[
\mathcal{W}^+_f(x)=\left\{y\in M : \limsup_{n\to \infty}\frac{\log d(f^{-n}(x),f^{-n}(y))}{n}<0\right\}.
\]

\begin{thm}[{\cite[Theorem~7.1 and Section~7.3.5]{introdyn}}]\label{pesin}
Let $M$ be a compact Riemannian manifold and let $f:M\to M$ be a $C^2$
diffeomorphism preserving an invariant measure $\mu$. Then for
$\mu$-almost every regular point $x$, if
$
k=\dim E^-_f(x),$ 
there exists a $C^{1+\beta}$ embedded $k$-dimensional disk
$\mathcal{W}^-_{f,\mathrm{loc}}(x)$ centered at $x$ such that
$
T_x\mathcal{W}^-_{f,\mathrm{loc}}(x)=E^-_f(x),
$
and
$
\mathcal{W}^-_{f,\mathrm{loc}}(x)\subset \mathcal{W}^-_f(x).
$
Similarly, there is a local unstable manifold
$\mathcal{W}^+_{f,\mathrm{loc}}(x)$ tangent to $E^+_f(x)$ at $x$.
\end{thm}

We also recall the absolute continuity properties of Pesin stable and unstable manifolds, which we shall use to prove ergodicity of some elements of the centralizer.

\begin{thm}[{\cite[Theorem~8.2 and Remark~8.19]{introdyn}}]\label{pesabcont}
Let $f$ be a $C^2$ volume-preserving diffeomorphism of a compact
Riemannian manifold $M$. In each Pesin block on which
$\dim E^-_f$ is constant, the family of local stable manifolds is
absolutely continuous. More precisely, if $S_1$ and $S_2$ are
sufficiently small smooth disks transverse to these local stable
manifolds inside a stable foliation box, and if
\[
h:D_1\subset S_1\longrightarrow D_2\subset S_2
\]
denotes the stable holonomy map, then $h$ is absolutely continuous with
respect to the induced Riemannian measures on $S_1$ and $S_2$. In fact,
its Jacobian is bounded above and bounded away from zero almost
everywhere.

The analogous statement for unstable manifolds follows by applying the
same result to $f^{-1}$.
\end{thm}

In particular, if a measurable set has zero measure in $M$, then its intersection
with almost every local Pesin stable leaf has zero measure with respect to the Riemannian volume restricted to the leaves; conversely, if
a measurable set has zero leafwise measure on almost every local leaf,
then it has zero measure in $M$.

\subsubsection{Leafwise Pesin theory}\mbox{}\par

We will use Pesin theory not only on the ambient manifold, but also after
restricting to an invariant foliation.  Let $T\in\Diff^2(M)$ preserve a
foliation $\mathcal F$.  For points $x,y$ on the same leaf, write
$d_{\mathcal F}(x,y)$ for their intrinsic leafwise distance.  At a regular
point for the restricted cocycle $DT|_{T\mathcal F}$, let
\[
E^-_{\mathcal F,T}(x)=\bigoplus_{\chi_i(x)<0}H^i_{\mathcal F,T}(x),
\qquad
E^+_{\mathcal F,T}(x)=\bigoplus_{\chi_i(x)>0}H^i_{\mathcal F,T}(x).
\]
The corresponding \emph{relative stable and unstable sets} are
\[
\mathcal W^-_{\mathcal F,T}(x)
 =\left\{y\in\mathcal F(x):
 \limsup_{n\to\infty}\frac1n
 \log d_{\mathcal F}(T^nx,T^ny)<0\right\},
\]
and
\[
\mathcal W^+_{\mathcal F,T}(x)
 =\left\{y\in\mathcal F(x):
 \limsup_{n\to\infty}\frac1n
 \log d_{\mathcal F}(T^{-n}x,T^{-n}y)<0\right\}.
\]

We use the following terminology.  A probability measure on $M$ is
\emph{smooth} if it has a strictly positive $C^1$ density with respect to
Riemannian volume.  A foliation $\mathcal F$ is \emph{uniformly $C^2$} if it
admits a finite foliation atlas whose plaque parametrizations have uniformly
bounded $C^2$ norms.  We say that $\mathcal F$ is \emph{absolutely
continuous} if its local holonomy maps between smooth transversals are
absolutely continuous. 

Two embedded submanifolds $N_1,N_2\subset M$ \emph{meet cleanly} if
$N_1\cap N_2$ is a submanifold and
\[
T_z(N_1\cap N_2)=T_zN_1\cap T_zN_2
\]
at every $z\in N_1\cap N_2$.

\begin{lem}\label{lem:leafwise-pesin-ac}
Let $T\in\Diff^2(M)$ preserve a uniformly $C^2$ foliation $\mathcal F$, and
let $\nu$ be a $T$-invariant smooth probability measure.  Assume that
$\mathcal F$ is absolutely continuous and that at $\nu$-almost every point $DT|_{T\mathcal F}$ has no
zero Lyapunov exponent.  Then:
\begin{enumerate}
\item For $\nu$-almost every $x$, the relative stable and unstable sets
$\mathcal W^-_{\mathcal F,T}(x)$ and
$\mathcal W^+_{\mathcal F,T}(x)$ locally agree with $C^1$ embedded disks
$\mathcal W^-_{\mathcal F,T,\mathrm{loc}}(x)$ and
$\mathcal W^+_{\mathcal F,T,\mathrm{loc}}(x)$ tangent at $x$ to
$E^-_{\mathcal F,T}(x)$ and $E^+_{\mathcal F,T}(x)$, respectively.

\item In every foliation box, for almost every $\mathcal F$-plaque $P$, the
local stable and unstable laminations formed by these relative Pesin disks
have absolutely continuous holonomy inside $P$, with respect to the
Riemannian measure induced on $P$.

\item Suppose moreover that, for $\nu$-almost every $x$ at which the ambient
local Pesin manifold is defined, $\mathcal F_{\mathrm{loc}}(x)$ and
$\mathcal W^{\pm}_{T,\mathrm{loc}}(x)$ meet cleanly near $x$.  Then,
after shrinking the local plaques if necessary,
\[
\mathcal W^{\pm}_{\mathcal F,T,\mathrm{loc}}(x)
 =\mathcal F_{\mathrm{loc}}(x)
  \cap\mathcal W^{\pm}_{T,\mathrm{loc}}(x)
\]
for almost every such $x$.
\end{enumerate}
\end{lem}

\begin{proof}
See Appendix~\ref{app:leafwise-pesin-proof}.
\end{proof}

\subsection{Standard holonomies of cocycles}

In Section \ref{sec:suhol}, we defined $su$-holonomies, which describe how the stable and unstable foliations connect different leaves of the center foliation. In this section, we define holonomies for linear cocycles, which describe how the cocycle varies along stable and unstable directions, forward and backward in time.

Let $f \colon M \to M$ be a diffeomorphism of a manifold $M$. A \emph{cocycle} $\beta$ over $f$ with values in a group $S$ is a map
\[
\beta \colon \mathbb{Z} \times M \to S
\]
such that, for all $m,n \in \mathbb{Z}$ and $x \in M$,
\[
\beta(m+n,x) = \beta\bigl(n, f^m(x)\bigr)\,\beta(m,x).
\]

Two cocycles $\beta$ and $\beta'$ over $f$ with values in $S$ are said to be (measurably, continuously, H\"older, smoothly) \emph{cohomologous} if there exists a (measurable, continuous, H\"older, smooth) transfer map $H \colon M \to S$ such that, for all $x \in M$ and $n \in \mathbb{Z}$,
\[
\beta(n,x) \;=\; H\bigl(f^n(x)\bigr)\,\beta'(n,x)\,H(x)^{-1}.
\]

Now suppose that $f$ is a partially hyperbolic diffeomorphism with exponents $\mu_u,\mu_s$ as in Definition \ref{defi:ph}. Fix a norm $\|\cdot\|$ (e.g. the operator norm) in $\mathrm{GL}_k(\mathbb{R})$. Let $\beta$ be a cocycle over $f$ with values in $\mathrm{GL}_k(\mathbb{R})$ for some $k \in \mathbb{N}$. We say that $\beta$ is \emph{fiber bunched} if there exist $r>0$ and $C>0$ such that:

\begin{itemize}
  \item[(1)] $\beta$ is $r$-H\"older in the base, that is,
  \[
  d\bigl(\beta(1,x),\beta(1,y)\bigr)\;\le\; C\,d(x,y)^r
  \quad\text{for all } x,y \in M;
  \]
  \item[(2)] there exists
$N\ge 1$ such that
\[
\sup_{x\in M}\|\beta(N,x)\|\,\|\beta(N,x)^{-1}\|\,\mu_s^{Nr}<1
\]
and
\[
\sup_{x\in M}\|\beta(-N,x)\|\,\|\beta(-N,x)^{-1}\|\,\mu_u^{Nr}<1.
\]
\end{itemize}

In particular, $Df|_{E^c_f}$ is fiber bunched whenever $f$ is $r$--center-bunched. If $f$ is a $C^1$--perturbation of a left translation $f_0$ on $X = G/\Gamma$, then we have fiber-bunching for $Df|_{E^{c}_f}$. If $f_0$ has only one  nonzero root value up to sign, then we also have fiber-bunching for $Df|_{E^{s}_f}$ and $Df|_{E^{u}_f}$.

\begin{prop}[Proposition~3.4 of \cite{ASV}; Theorem~3.5 of \cite{KalSad}]\label{prop:holcocycle}
Let $\beta$ be a fiber-bunched cocycle over a partially hyperbolic diffeomorphism $f$. Then, for any pair of points $x,y \in M$ in the same local stable leaf of $f$, the limit
\[
\mathcal{H}_{x,y}^s
\;=\;
\lim_{n \to \infty}
\bigl(
  \beta(-n, f^{n}(y)) \,\beta(n, x)
\bigr)
\]
exists in $\mathrm{GL}_k(\mathbb{R})$ and is called the \emph{standard stable holonomy} of the cocycle. Similarly, for any $x,y$ in the same local unstable leaf of $f$,
\[
\mathcal{H}_{x,y}^u
\;=\;
\lim_{n \to \infty}
\bigl(
  \beta(n, f^{-n}(y)) \,\beta(-n, x)
\bigr)
\]
exists and is called the \emph{standard unstable holonomy}. Moreover, the maps
\[
(x,y) \mapsto \mathcal{H}_{x,y}^s, \qquad (x,y) \mapsto \mathcal{H}_{x,y}^u
\]
are H\"older continuous (in the appropriate local stable/unstable sets).
\end{prop}

\section{The center-fixing centralizer and its dynamics }\label{sec:cla}
In this section, we establish some dynamical properties of elements of the centralizer.
\subsection{The center-fixing centralizer}

Let $f_0$ be a diagonal map and let $f\in \mathrm{Diff}^\infty(X)$, $d_{C^2}(f,f_0)<K$ be a sufficiently small $C^1$-perturbation of $f_0$. The following result shows that the center-fixing centralizer of $f$, denoted \[\CZ(f):=\{g\in \mathcal Z(f): g(x)\in \W^c_f(x),\;\forall x\in X\}\] is a finite-index subgroup in the full centralizer $\mathcal Z(f)$.

\begin{thm}[Theorem 4.1, Corollary 4.3 in \cite{Wang}, see also Theorem 1.1 in \cite{Witte}]
   Let $f_0$ be a non-trivial diagonal map, and let $f\in \mathrm{Diff}^\infty(X)$ be a $C^1$ perturbation of $f_0$, then $\CZ(f)$ is a finite index subgroup of $\mathcal Z(f)$.
\end{thm}

From now on, we concentrate on dynamical properties of center-fixing elements $g\in \CZ(f)$.

\subsection{\texorpdfstring{Elements of the centralizer are $C^0$ close to
affine}{Elements of the centralizer are C0 close to affine}}
Following \cite{Wang26}, we estimate the difference between the $su$-holonomies
of $f$ and $f_0$ to show that the elements of $\CZ(f)$ are $C^0$-close to affine
translations along center leaves of $f_0$ in $X$.

 Throughout the rest of the paper, we take $f\in \mathrm{Diff}^\infty(X)$ with $d_{C^2}(f,f_0)<K$ and $d_{C^1}(f,f_0)$ sufficiently small.
 
 Let $\phi_f$ be the  canonical leaf
conjugacy from $(X,\mathcal W^c_{f_0})$ to $(X,\mathcal W^c_f)$
provided by Theorem ~\ref{leafconj}. For
$g\in\CZ(f)$ denote  $\hat g:=\phi_f^{-1}\,g\,\phi_f$. Since $g$ fixes the
center leaves of $f$, the conjugate map $\hat g$ fixes the center leaves of $f_0$.

Denote by $\tilde g$ a lift of $\hat g$ to the cover $G$ of $G/\Gamma$,
and denote by $\widetilde{\mathcal W}^c_{f_0}$ the lift of the center leaves of
$f_0$ to $G$, which are cosets of a subgroup of $G$.

Then we have the following result from \cite{Wang26}.

\begin{prop}[Proposition 4.1 of \cite{Wang26}]\label{gtransl}
For any $0<\epsilon<1$ and $R>0$ there exists $\delta>0$ such that for any
$f\in\mathrm{Diff}^r(X)$ with $d_{C^1}(f,f_0)<\delta$, and for any
$x\in X$ and $g\in\CZ(f)^0$, we have that if
$d_{\widetilde{\mathcal W}^c_{f_0}}(\tilde g(x),x)\le R$, then we have
\[
d_{\widetilde{\mathcal W}^c_{f_0}}(\tilde g(y),g_1(y))
\le \epsilon\; d_{\widetilde{\mathcal W}^c_{f_0}}(\tilde{g}(x),x),
\;\; \forall y\in X.
\]

Here $g_1$ denotes the left translation in $G/\Gamma$ determined by
$\tilde g(x)x^{-1}$, i.e.
$g_1(y)=\big(\tilde g(x)x^{-1}\big)\cdot y$.

\end{prop}

Motivated by Proposition \ref{gtransl} we make the following definition.

\begin{defi}\label{def:projepsclose}
Let $g:X\to X$. We say that $g$ is \emph{projectively $\epsilon$-close} to a left
translation $g_0$ if for every $y\in X$ the lifts satisfy
\[
d_{\widetilde{\mathcal W}^c_{f_0}}(\tilde g(y),g_0(y))
\le \epsilon\; d_{\widetilde{\mathcal W}^c_{f_0}}(g_0(y),y),
\]
where $\tilde g$ is the lift of  $\hat g=\phi_f^{-1} g\phi_f$ to $G$.
We note that $g_0$ is a left translation, so the right hand side
$d_{\widetilde{\mathcal W}^c_{f_0}}(g_0(y),y)$ is a constant not
depending on $y$.
\end{defi}

\subsection{Adapted isomorphism of the Lie algebras}

First we recall a slight variation of Theorem 3.1 in \cite{Wang}, which
implies that $\CZ(f)$ is a Lie group acting freely, properly, and smoothly
on the lifted center leaf $\widetilde{\mathcal W}^c_F$.

\begin{thm}[Theorem 3.1, \cite{Wang}]\label{czlieg}
Let $M$ be a compact Riemannian manifold and let $F$ be a partially
hyperbolic diffeomorphism of $M$. Assume $F$ is accessible and dynamically
coherent, center $r$-bunched with narrow band spectrum, and that a lift of
$F$ has global holonomy on the lift of a center leaf
$\widetilde{\mathcal W}^c_F(x_0)$. Then the $\mathcal W^c$-fixing,
$C^\infty$ centralizer $\CZ(F)$ is a Lie group acting freely and
properly on $\widetilde{\mathcal W}^c_F(x_0)$. Moreover the action
\[
\CZ(F)\times \widetilde{\mathcal W}^c_F(x_0)\to\widetilde{\mathcal W}^c_F(x_0),
\qquad (g,x)\mapsto g(x),
\]
is jointly $C^\infty$.
\end{thm}

Applying this to $f_0$, and passing to the subgroup whose lifts induce the
identity on the deck group, identifies the identity component of
$\CZ(f_0)$ with the connected algebraic centralizer $G^c=Z_G(a)^0$.
Consequently $\mathcal Z(f_0)$ is  virtually the  algebraic centralizer.

 Propositions \ref{fparhyp} and \ref{proP:globhol} verify
the hypotheses of Theorem \ref{czlieg} for  every sufficiently small
$C^1$  perturbation of $f_0$.  Hence $\CZ(f)$ is a Lie group acting freely and
properly on  each lifted center leaf, and the action is jointly smooth.
 Since $\CZ(f)$ has finite index in $\mathcal Z(f)$, the  full centralizer is
a finite  extension of  this Lie group; this proves the first assertion of
Theorem \ref{cenrignongen}.

 Fix the constant $K$ in Theorem ~\ref{cenrignongen}.   Choose a
$K$-regular  base point $z_\ast\in X$, put
$x_\ast=\phi_f^{-1}(z_\ast)$.  The action
 \[
\sigma:\CZ(f)\times \mathcal W^c_f(z_\ast)
       \longrightarrow\mathcal W^c_f(z_\ast)
\]
is jointly $C^\infty$, and  $\phi_f$ is $C^r$ when restricted to center
leaves. Differentiating the action at the basepoint yields a linear map
 \begin{equation}\label{equ:Psi}
\Psi:\mathfrak g(\CZ(f))
\ \xrightarrow{\,D\sigma(\cdot,z_\ast)\,}\ 
T_{z_\ast}\mathcal W^c_f(z_\ast)
\ \xrightarrow{\,D\phi_f^{-1}\,}\ 
T_{x_\ast}\mathcal W^c_{f_0}(x_\ast)\simeq\mathfrak g(G^c),
\end{equation}
where $\mathfrak g(\CZ(f))$ denotes the Lie algebra of $\CZ(f)$ and
$\mathfrak g(G^c)$ the Lie algebra of $G^c$.

We claim that there  is a  Lie-algebra isomorphism $\Xi$  close to $\Psi$.
The  next lemma makes this precise.   This is the  only use of the
 $K$-regularity hypothesis.

\begin{lem}\label{liealg}

For  every $K>1$ and  $\epsilon_0>0$ there  is a $C^1$-neighborhood of
$f_0$ with the following  property.  If
 $\mathcal Z(f)\doteq\mathcal Z(f_0)$ and the  centralizer action is
$K$-regular, then  there is a  Lie-algebra isomorphism
 \[
 \Xi:\mathfrak g(\CZ(f))\longrightarrow\mathfrak g(G^c)
 \quad\text{such that}\quad
 \|\Xi-\Psi\|_{\mathrm{op}}<\epsilon_0.
\]

\end{lem}

\begin{proof}

{ \textbf{Step 1: $\Psi$ almost preserves brackets.}
}

The assumed $K$-regularity of $\mathcal Z(f)$  gives the required
quantitative bounds for the orbit map of $\CZ(f)$  at $z_\ast$.

Since $\mathcal Z(f)\doteq \mathcal Z(f_0)$ and $\mathcal{Z}(f)/\CZ(f),\mathcal{Z}(f_0)/\CZ(f_0)$ are finite, the Lie algebras  $\mathfrak g(\CZ(f))$ and $\mathfrak{g}(G^c)$ are isomorphic.

 Use the norm  on $\mathfrak g(\CZ(f))$ induced by the  derivative of  the
orbit map,
 \[
 \|X\|=\lim_{t\to0}
 \frac{d(\sigma(\exp(tX),z_\ast),z_\ast)}{|t|}.
\]
Let  $X_1,X_2$ be unit vectors and  write
\[
g_i^t:=\exp(tX_i)\in \CZ(f),\qquad i=1,2.
\]
 For $0<t<\frac{1}{K}$, set
the left translations in $G/\Gamma$ 
 \[
L_i^t(y):=(\tilde g_i^t(x_\ast)x_\ast^{-1})\cdot y.
\]

The commutator estimate supplied by Proposition \ref{gtransl} is
\begin{equation}\label{eq:commutator-c0}
d\Bigl(\hat g_1^{-t}\hat g_2^{-t}\hat g_1^t\hat g_2^t(x_\ast),\,
(L_1^t)^{-1}(L_2^t)^{-1}L_1^tL_2^t(x_\ast)\Bigr)
\le C_1\epsilon\,t,
\end{equation}
for  a constant $C_1$ depending only on the fixed  local geometry.

To prove \eqref{eq:commutator-c0}, apply Proposition~\ref{gtransl}
successively to the four factors.  For the negative-time
factors, compose the estimate from Proposition~\ref{gtransl} with
$(L_i^t)^{-1}$; after enlarging a uniform constant $C_0$, this compares
$\hat g_i^{-t}$ with $(L_i^t)^{-1}$.  The last error is zero at
$x_\ast$ by the definition of $L_2^t$:

$$d\Bigl(\hat g_1^{-t}\hat g_2^{-t}\hat g_1^t\hat g_2^t(x_\ast),\,
(L_1^t)^{-1}\hat g_2^{-t}\hat g_1^t\hat g_2^t(x_\ast)\Bigr)
\le C_0\epsilon d(\tilde g_1^t(x_\ast),x_\ast);$$
\begin{align*}
&d\Bigl((L_1^t)^{-1}\hat g_2^{-t}\hat g_1^t\hat g_2^t(x_\ast),
(L_1^t)^{-1}(L_2^t)^{-1}\hat g_1^t\hat g_2^t(x_\ast)\Bigr)\\
&\qquad\le C_0\epsilon d(\tilde g_2^t(x_\ast),x_\ast)
 \cdot \|(L_1^t)^{-1}\|_{C^1}.
\end{align*}
\begin{align*}
&d\Bigl((L_1^t)^{-1}(L_2^t)^{-1}\hat g_1^t\hat g_2^t(x_\ast),
 (L_1^t)^{-1}(L_2^t)^{-1}L_1^t\hat g_2^t(x_\ast)\Bigr)\\
&\qquad\le \epsilon d(\tilde g_1^t(x_\ast),x_\ast)
 \cdot \|(L_1^t)^{-1}\|_{C^1}\cdot \|(L_2^t)^{-1}\|_{C^1}.
\end{align*}
and note that
 $$(L_1^t)^{-1} (L_2^t)^{-1}  L_1^t\hat g_2^t(x_\ast)=(L_1^t)^{-1}(L_2^t)^{-1} L_1^t L_2^t(x_\ast).$$

For left translations,  $\|DL\|$ and $\|DL^{-1}\|$ tend uniformly  to $1$
as  $L\to\mathrm{id}$, so the displayed estimates prove
\eqref{eq:commutator-c0}.

 The commutator
displacement at $z_\ast$ is $O(t^2)$;
differentiability of $\phi_f^{-1}$ at $z_\ast$ sends its quadratic term
to $\Psi([X_1,X_2])$, while the $C^3$ bound for the original orbit map
controls the $O(t^3)$ remainder.  Thus, the $C^3$ part of $K$-regularity and
\eqref{eq:commutator-c0} give, after division by $t^2$,  
 \begin{equation}
\|\Psi([X_1,X_2])-[\Psi(X_1),\Psi(X_2)]\|
\le C_1 \frac{\epsilon }{t}+C_2K^2t
\end{equation}
Here $C_1,C_2$ are uniform constants, and the
$K$-regularity estimate is applied at its stated base point:
\[
\sup_{d(g,\mathrm{id})<1/K}
\bigl(\|D^2\sigma(g,z_\ast)\|+\|D^3\sigma(g,z_\ast)\|\bigr)
\le K.
\]
Taking $t$ comparable to $\sqrt\epsilon$ (with constants depending on
$K$) gives
\begin{align}
&\|\Psi([X_1,X_2])-[\Psi(X_1),\Psi(X_2)]\|_{\mathfrak g(G^c)}
\nonumber\\
&\qquad\le C_0\sqrt{\epsilon}\,
\|X_1\|_{\mathfrak g(\CZ(f))}
\|X_2\|_{\mathfrak g(\CZ(f))}.
\label{eq:approx-Lie-hom}
\end{align}
By homogeneity, the estimate holds for all $X_1,X_2\in \mathfrak g(\CZ(f))$.

\textbf{Step 2: find a nearby Lie-algebra isomorphism.}

Fix $\epsilon_0>0$.  Using the fixed identification $
\mathfrak g(\CZ(f))\simeq V:=\mathfrak g(G^c)$ given by virtual isomorphism between $\mathcal Z(f)$ and $\mathcal Z(f_0)$,
we regard $\Psi$ as an invertible linear map from $V$ to itself.  The
$K$-regularity assumption gives, after enlarging $K$ by a uniform
constant if necessary,
\[
\|\Psi\|_{\mathrm{op}}\le K,
\qquad
\|\Psi^{-1}\|_{\mathrm{op}}\le K.
\]
Consider the compact set
\[
\mathcal L_K
 :=
 \left\{
 T\in\GL(V):
 \|T\|_{\mathrm{op}}\le K,\ 
 \|T^{-1}\|_{\mathrm{op}}\le K
 \right\}.
\]
Define the bracket defect
\[
\Delta(T)
 :=
 \sup_{\|X\|=\|Y\|=1}
 \bigl\|
 T([X,Y])-[T(X),T(Y)]
 \bigr\|,
 \qquad T\in\mathcal L_K,
\]
and let
\[
\mathcal L_K^0
 :=
 \{T\in\mathcal L_K:\Delta(T)=0\}.
\]
The function $\Delta$ is continuous, and every element of
$\mathcal L_K^0$ is an invertible Lie-algebra homomorphism, hence a
Lie-algebra isomorphism.

By compactness of $\mathcal L_K$, for every $\eta>0$, there exists
$\delta=\delta(K,\eta)>0$ such that
\begin{equation}\label{eq:defect-distance}
\Delta(T)<\delta
\quad\Longrightarrow\quad
d(T,\mathcal L_K^0)<\eta
\end{equation}
for every $T\in\mathcal L_K$. 

By \eqref{eq:approx-Lie-hom},
$
\Delta(\Psi)\le C_0(K)\sqrt{\epsilon}.
$
Choose the $C^1$-neighborhood of $f_0$ sufficiently small that
\[
C_0(K)\sqrt{\epsilon}<\delta(K,\epsilon_0).
\]
Then \eqref{eq:defect-distance} gives an element
$\Xi\in\mathcal L_K^0$ satisfying $
\|\Xi-\Psi\|_{\mathrm{op}}<\epsilon_0.$ 

Since $\Xi: \mathfrak g(\CZ(f))\longrightarrow\mathfrak g(G^c)$ is invertible and preserves brackets, it is a Lie-algebra
isomorphism.
This proves the lemma.

\end{proof}
\subsection{Dynamics in the center foliation}\label{sec:gcendyn}
Now suppose $f_0$ is a diagonal map,
$\mathcal Z(f)\doteq \mathcal Z(f_0)$ and the $\mathcal Z(f)$ action is $K$-regular. Then the Lie algebra homomorphism
of Lemma \ref{liealg} is an isomorphism.

\begin{lem}\label{lem:finflow}
    Suppose $f_0$ is a diagonal map and $f$ is a sufficiently small $C^1$-perturbation of $f_0$, such that
$\mathcal Z(f)\doteq \mathcal Z(f_0)$ and the $\mathcal Z(f)$ action is $K$-regular, then there exists a flow $f^t\in \mathcal Z(f)$ such that $f^1=f$ and $\mathcal {CZ}(f)\subset \mathcal Z(f^t)$ for any $t\in \R$.
\end{lem}

\begin{proof}
We first observe that $\mathcal{CZ}(f)$ is connected. By the leaf-conjugacy between $f$ and $f_0$,
$f$ fixes every leaf of $\mathcal W_f^c$, because $f_0$ acts trivially
on the model center-leaf space, so we have $f\in \CZ(f)$.

Let $L$ be a connected center leaf on which $\CZ(f)$ acts freely and
properly. Since $\mathcal Z(f)/\CZ(f)$ is finite and
$\mathcal Z(f)\doteq\mathcal Z(f_0)$,
$
\dim \CZ(f)
 =\dim\mathcal Z(f_0)
 =\dim\mathcal W_{f_0}^c
 =\dim L.$
 
It follows that the orbit $\CZ^0(f)x$ of the identity component is open in
$L$. It is also closed, since the $\CZ^0(f)$-action is proper. Hence
$\CZ(f)^0x=L$. Thus freeness of the $\CZ(f)$-action gives $
\CZ(f)=\CZ(f)^0.$

Let
$
v_f=\log f_0\in Z\bigl(\mathfrak g(\mathcal{CZ}(f_0))\bigr)$ 
and let
\[
\Xi:\mathfrak g(\CZ(f))\longrightarrow
\mathfrak g(\mathcal{CZ}(f_0))
\]
be the adapted Lie-algebra isomorphism constructed in Lemma \ref{liealg}. Set
\[
X_0=\Xi^{-1}(v_f),
\qquad
a_f=\exp_{\CZ(f)}(X_0).
\]
Since Lie-algebra isomorphisms preserve centers,
$
X_0\in Z\big(\mathfrak g(\CZ(f))\big),$
and hence $a_f\in Z(\CZ(f))$.

The same comparison used in Proposition~\ref{prop:g_0tog}, applied
with $g_0=f_0$, shows that $a_f$ converges to $f_0$ in the center-leaf
orbit coordinates. The map $f$ itself also converges to $f_0$ in
these coordinates. The $K$-regularity of the orbit map therefore
implies that
\[
a_f^{-1}f\longrightarrow e
\qquad\text{in }\CZ(f)
\]
as $d_{C^1}(f,f_0)\to0$. After decreasing the neighborhood, we may
therefore write uniquely
\[
a_f^{-1}f=\exp_{\CZ(f)}(Y)
\]
for some sufficiently small $Y\in\mathfrak g(\CZ(f))$.

Every element of $\CZ(f)$ commutes with $f$, by the definition of
$\mathcal{CZ}(f)$, and $a_f$ is central in $\CZ(f)$. Thus
$a_f^{-1}f$ is central in $\CZ(f)$. Consequently, $Y\in Z\big(\mathfrak g(\CZ(f))\big)$. Hence we may define
\[
X_f:=X_0+Y\in  Z\big(\mathfrak g(\CZ(f))\big)
\]
and, since $X_0$ and $Y$ commute,
\[
f=a_f\exp_{\CZ(f)}(Y)
  =\exp_{\CZ(f)}(X_0+Y)
  =\exp_{\CZ(f)}(X_f).
\]
This defines the required flow
\[
f^t=\exp_{\CZ(f)}(tX_f)\in Z(\CZ(f)).
\]
\end{proof}

By the isomorphism $\Xi: \mathfrak {g}(\CZ(f))\to \mathfrak {g}(G^c)$ given by Lemma \ref{liealg}, for any  semisimple $g\in\CZ(f)^0$ the adjoint action of $g$ on
$\mathfrak g(\CZ(f))$ induces a splitting
\begin{equation}\label{eq:liealgsplit}
    \mathfrak g(\CZ(f))=E^s_g(\mathfrak g)\oplus E^c_g(\mathfrak g)\oplus E^u_g(\mathfrak g),
\end{equation}

given by the linear transformation $\operatorname{Ad}(g)$ on
$\mathfrak g(\CZ(f))$, where $E^s_g$ (resp.\ $E^c_g$, $E^u_g$) is the sum
of generalized eigenspaces of $\operatorname{Ad}(g)$ with eigenvalues of
modulus $<1$ (resp.\ $=1$, $>1$). Passing to the tangent space of
$\mathcal W^c_f(x_0)$ and using Lemma \ref{liealg}, we obtain at each
$x_0\in X$ the splitting
\begin{equation}\label{eq:Ecfsplit}
    E^c_f=E^{c,s}_{f,g}\oplus E^{c,c}_{f,g}\oplus E^{c,u}_{f,g},
\end{equation}

where $E^{c,*}_{f,g}=D\sigma(E^*_g(\mathfrak g))$ for $*=s,c,u$.

By construction the splitting $E^c_f=E^{c,s}_{f,g}\oplus E^{c,c}_{f,g}\oplus E^{c,u}_{f,g}$
is $g$-invariant. Moreover, since  $\phi_f$ and the action map
$\sigma$ in \eqref{equ:Psi} are $C^r$ in center leaves, $Dg$ is
continuously cohomologous via $D\sigma$ to the constant linear map induced by
$\operatorname{Ad}(g)$ on $\mathfrak g(\CZ(f))$. Hence, if $E^{c,s}_{f,g},E^{c,u}_{f,g}$ are non-trivial, then $Dg$ is partially
hyperbolic on the center bundle $E^c_f$ with respect to the above
splitting in Equation \ref{eq:Ecfsplit}.

We now show the bundles in Equation \ref{eq:Ecfsplit} are $C^0$-close to the corresponding bundles of
$f_0$ and the model element $g_0\in\CZ(f_0)$.

\begin{prop}\label{prop:g_0tog}
Let  $g_0=\exp_G(v_0)\in\CZ(f_0)^0$, where $v_0$ is real-split semisimple and $g_0\ne e$, and let $\epsilon>0$. Suppose $f$ is a sufficiently
$C^1$-small perturbation of the diagonal map $f_0$,
$\mathcal Z(f)\doteq \mathcal Z(f_0)$ and the action $\mathcal Z(f)\times X\to X$ is $K$-regular. (Here $g_0$ is chosen first, and the  required smallness of  $d_{C^1}(f,f_0)$ may depend on $g_0$.) Then there exists $g\in\CZ(f)$ satisfying:

\begin{enumerate}
\item $g$ is projectively $\epsilon$-close  to $g_0$ (see Definition \ref{def:projepsclose}).

\item The subbundles $E^{c,*}_{f,g}\subset E^c_f$ are $C^0$-close to
  $E^{c,*}_{f_0,g_0}=E^c_{f_0}\cap E^*_{g_0}$ for $*=s,c,u$.

\item The subbundles $E^{c,*}_{f,g}$ uniquely integrate to sub-foliations
  $\mathcal W^{c,*}_{f,g}$ of $\mathcal W^c_f$.

\item There  is a bi-H\"older  homeomorphism $h_c:X\to X$ mapping the leaves  of
$\mathcal W^{c,c}_{f_0,g_0}$ onto the leaves of
$\mathcal W^{c,c}_{f,g}$.  Moreover, $h_c$ is  simultaneously a leaf
conjugacy  from $(g_0,\mathcal W^{c,c}_{f_0,g_0})$ to
 $(g,\mathcal W^{c,c}_{f,g})$ and from
$(f_0,\mathcal W^{c,c}_{f_0,g_0})$ to
$(f,\mathcal W^{c,c}_{f,g})$.
\end{enumerate}
\end{prop}

\begin{proof}
Recall that the base point $x_\ast$ and the Lie-algebra isomorphism
$\Xi:\mathfrak g(\CZ(f))\to\mathfrak g(G^c)$ were fixed above. Set
 \[
g=\exp_{\CZ(f)}\bigl(\Xi^{-1}(v_0)\bigr).
\]
Choose  $m\ge1$ so that  $u_0=v_0/m$ lies in the $K$-regular exponential
ball.  Since $m$ and $g_0$ are fixed, continuity of the exponential,
$\Xi\to\Psi$, and Proposition ~\ref{gtransl},  first applied to
 $\exp_{\CZ(f)}(\Xi^{-1}u_0)$ and then to  its $m$th power,  give
 \[
d_G\bigl(\tilde g(x_\ast),\tilde g_0(x_\ast)\bigr)
 <\tfrac\epsilon4d_G(\tilde g_0(x_\ast),x_\ast).
\]
Here $\tilde g$ is the lift of  $\phi_f^{-1}g\phi_f$ and  $\tilde g_0$ is
the  corresponding model lift.
 Applying Proposition ~\ref{gtransl} once more,  with error $\epsilon/4$,
to  the left translation determined at $x_\ast$ shows that $g$ is
projectively $\epsilon$-close to $g_0$. This proves (1).

By construction, $E^*_g(\mathfrak g)$ from Equation \ref{eq:liealgsplit} corresponds to the constant bundle
$E^c_{f_0}\cap E^*_{g_0}$ for $*=s,c,u$. Therefore, since $\Xi$ is close to  $\Psi=D\phi_f^{-1} \circ D\sigma$ and the leaf conjugacy  $\phi_f$ restricted to the center is $C^1$-close to identity, the
sub-bundles $E^{c,*}_{f,g}=D\sigma(E^*_g(\mathfrak g))$ are close  at the
base point.  This closeness is uniform on $X$: bounded $f$--$su$ paths
transport the orbit-map differentials equivariantly, and their stable and
unstable holonomies are uniformly $C^1$-close to the corresponding model
holonomies.  Hence $E^{c,*}_{f,g}$ is $C^0$--close to
$E^{c,*}_{f_0,g_0}=E^c_{f_0}\cap E^*_{g_0}$ for $*=s,c,u$, proving (2).

Since  $g$ is real-split semisimple, these bundles uniquely integrate to
the orbits of  its stable, central, and unstable subgroups.  Concretely,
 \[
\mathcal W^{c,s}_{f,g}(x)=G_g^s\cdot x,
\qquad G_g^s=\{h\in\CZ(f)^0:g^nhg^{-n}\to e\},
\]
\[
\mathcal W^{c,u}_{f,g}(x)=G_g^u\cdot x,
\qquad G_g^u=\{h\in\CZ(f)^0:g^{-n}hg^n\to e\},
\]
and
 \[
\mathcal W^{c,c}_{f,g}(x)=Z_{\CZ(f)^0}(g)^0\cdot x.
\]
Their tangent spaces are respectively
$E^{c,s}_{f,g}$, $E^{c,u}_{f,g}$, and $E^{c,c}_{f,g}$, proving (3).

It remains to prove (4). 

We use the same methods as in Section 9 of \cite{PSWholrev}. Fix a small $0<r\ll1$. 

Choose $k\in\mathbb N$ large enough so
that
\[
g^k\big(\mathcal W^{c,s}_{f,g}(x,r)\big)\subset
\mathcal W^{c,s}_{f,g}\big(g^k(x),r/2\big),
\qquad
g^{-k}\big(\mathcal W^{c,u}_{f,g}(x,r)\big)\subset
\mathcal W^{c,u}_{f,g}\big(g^{-k}(x),r/2\big)
\]
(the choice of such $k$ only depends on $g_0$). 

 After conjugation by $\phi_f$, the  restrictions of $g$ and $g_0$ to
model center leaves are uniformly $C^1$-close, and  their respective
leafwise foliations are $C^0$-close.  Thus we
may define the ``amalgam'' map $a:X\to X$ by
 \[
a(x)=\mathcal W^{c,s}_{f,g}(g^k(x),r)\cap
\phi_f\!\big(\mathcal W^{c,cu}_{f_0,g_0}(g_0^k(\phi_f^{-1}(x)),r)\big),
\]
where $\mathcal W^{c,cu}_{f_0,g_0}(\cdot,r)$ denotes the joint
integration of $\mathcal W^{c,c}_{f_0,g_0}$ and $\mathcal W^{c,u}_{f_0,g_0}$
in a small radius $r$. Because $E^{c,s}_{f,g}$ is $C^0$-close to
$E^{c,s}_{f_0,g_0}$ and  $\phi_f$ is $C^1$ inside $\W^c_f$-leaves with differential close to the
identity, for any
$y\in\mathcal W^{c,s}_{f,g}(x,r)$,
\[
d\big(a(x),a(y)\big)\le \frac{2}{3}d(x,y).
\]
By the local-product structure of $\W^{c,s}_{f,g}$ and $\phi_f(\W^{c,cu}_{f_0,g_0})$, the amalgam map $a$ is a homeomorphism with inverse map also given by local intersection of the foliations, and $a$ is $C^1$-along the $\W^{c,s}_{f,g}$-leaves.

Define the fiber bundle
\[
X_1=\{(x,y):x\in X,\ y\in\mathcal W^{c,s}_{f,g}(x,r)\},
\]
with projection $\pi:X_1\to X$.

If $f$ is chosen
sufficiently close to $f_0$, then  $\phi_f^{-1}g^k\phi_f$ is  uniformly close to  $g_0^k$ along
model center leaves, and hence $d_{\mathcal W^{c,s}_{f,g}}
(a(x),g^kx)<r/3$ uniformly and
\[
g^k\big(\mathcal W^{c,s}_{f,g}(x,r)\big)\subset
\mathcal W^{c,s}_{f,g}\big(g^k(x),r/2\big)\subset
\mathcal W^{c,s}_{f,g}\big(a(x),r\big).
\]
Thus the bundle map
\[
F:X_1\to X_1,\qquad F(x,y)=(a(x),g^k(y))
\]
is well defined and contracts fibers over $a$. Similarly
\[
G:X_1\to X_1,\qquad G(x,y)=(g^k(x),a(y))
\]
is well defined and contracts fibers over $g^k$: indeed,
$a(y)$ is on the stable leaf of $g^kx$, and the bounds
$d(a(y),a(x))\le2r/3$ and $d(a(x),g^kx)<r/3$ give the stated radius
inclusion. 

By the invariant-section theorem for fiber contractions over
the homeomorphism $a$, $F$ admits a unique invariant
section $\sigma(x)=(x,s(x))$ and $G$ admits a unique invariant section
$(x,t(x))$. The invariance relations imply
$g^k(s(x))=s(a(x))$ and $a(t(x))=t(g^k(x))$. For any
$y\in\mathcal W^{c,cu}_{f,g}(x,r)$, plaque expansiveness gives
$t(y)\in\phi_f\mathcal W^{c,cu}_{f_0,g_0}(\phi_f^{-1} t(x))$.
Hence  $\phi_f^{-1}\circ t\big(\mathcal W^{c,cu}_{f,g}(x)\big)\subset
\mathcal W^{c,cu}_{f_0,g_0}(\phi_f ^{-1} \circ t(x))$. Moreover
$s\circ t\circ g^k=s\circ a\circ t=g^k\circ s\circ t$ and
$s\circ t(x)\in\mathcal W^{c,s}_{f,g}(x,2r)$. Iterating shows
$d(g^{kn}(s\circ t(x)),g^{kn}(x))\le 2r$ for all $n\in\mathbb Z$. Since
$g$ exponentially expands stable leaves as $n\to-\infty$, it follows
that $s\circ t=\mathrm{id}_X$.  Invariance of domain and compactness
show that  $t$ is a  homeomorphism with  inverse $s$. Thus
 $\phi_f^{-1}\circ t$ is a leaf conjugacy
$(g^k,\mathcal W^{c,cu}_{f,g})$ to $(g_0^k,\mathcal W^{c,cu}_{f_0,g_0})$.

Put $H_0=s\circ\phi_f$.  Repeat the same construction with
$g^{-k}$, with stable and unstable exchanged, and with $H_0$ in place
of $\phi_f$.  Thus the second amalgam map is
\[
b(x)=\mathcal W^{c,u}_{f,g}(g^{-k}x,r)\cap
H_0\!\left(\mathcal W^{c,cs}_{f_0,g_0}
 (g_0^{-k}H_0^{-1}x,r)\right),
\]
and the fiber bundle is
\[
X_u=\{(x,y):y\in\mathcal W^{c,u}_{f,g}(x,r)\}.
\]
The two bundle maps are
\[
(x,y)\longmapsto(b(x),g^{-k}y),\qquad
(x,y)\longmapsto(g^{-k}x,b(y)).
\]
The same radius estimates as above make them well defined fiber
contractions.  Every fiber $\mathcal W^{c,u}_{f,g}(x,r)$ lies in the
already matched $\mathcal W^{c,cu}_{f,g}$-leaf, while $b$ and $g^{-k}$
preserve that leaf.  Hence both bundle maps restrict to the bundle over
each such leaf.  Uniqueness in the
invariant-section theorem therefore shows that the restricted and
unrestricted constructions give the same section.  Consequently the resulting correction $s_u$ preserves
every already matched center-unstable leaf and matches the
center-stable leaves. Therefore
 \[
h_c=s_u\circ s\circ\phi_f
\]
maps the leaves of  $\mathcal W^{c,c}_{f_0,g_0}$ onto the leaves of
 $\mathcal W^{c,c}_{f,g}$.  The H\"older  invariant-section estimate in
 \cite[Proposition~11 and Section~9]{PSWholrev}, applied also to the  inverse
construction, makes $h_c$ bi-H\"older.

Finally, by our construction of the foliation and Lemma \ref{lem:finflow}, leaves of $\W^{c,c}_{f,g}$ are fixed by $f$ and $g$, and leaves of $\W^{c,c}_{f_0,g_0}$ are fixed by $f_0$ and $g_0$. Therefore, $h_c$ is a leaf conjugacy. This
proves (4).
\end{proof}

\subsection{H\"older exponent of leaf conjugacy}
Here we show that in certain cases the H\"older exponent of the leaf
conjugacy from $(f_0,\mathcal W^c_{f_0})$ to $(f,\mathcal W^c_{f})$ can be
taken arbitrarily close to $1$.

Let $\mu=\lambda_{ij}(f_0)\in\R$ be  the value of a root of $f_0$.
Denote by $\mathcal W^\mu_{f_0}$ the cosets given by the exponential of
 $\oplus_{\lambda_{ij}(f_0)=\mu}\mathfrak g_{ij}$. Then for any
$\epsilon>0$, if $f$ is sufficiently close to $f_0$, there is a splitting
 \[
T_xX=E^c_f(x)\oplus
 \bigoplus_{\mu\in \{\lambda_{ij}(f_0)\}\setminus\{0\}}E^\mu_f(x)
\]
such that
\[
e^{\mu-\epsilon}\|v\|\le \|D_xf(v)\|\le e^{\mu+\epsilon}\|v\|,
\qquad\forall v\in E^\mu_f.
\]

\begin{lem}\label{lem:holexp1}
Let  $\phi_f$ be the  canonical leaf conjugacy from
$(f_0,\mathcal W^c_{f_0})$ to $(f,\mathcal W^c_f)$. Suppose for every
$\mu\in\{\lambda_{ij}(f_0)\}\setminus\{0\}$, the bundles $E^\mu_f$ and
$E^c_f$ are jointly integrable to a foliation $\mathcal W^{\mu,c}_f$,
and
 \[
\phi_f(\mathcal W^{\mu,c}_{f_0})=\mathcal W^{\mu,c}_f.
\]
Then, for  every $\theta_0<1$,  the maps $\phi_f$ and $\phi_f^{-1}$ are
$\theta_0$-H\"older if $f$ is sufficiently  $C^1$-close to $f_0$ .
\end{lem}

This is the quantitative conclusion of
\cite[Theorem~A and Sections~4 and~9]{PSWholrev}.  In the model,
$Df_0|_{E^c_{f_0}}$ is isometric, while the normal contraction and
expansion inequalities are strict.  Hence, for every fixed
$\sqrt\theta_0<1$, the two spectral inequalities for the center leaf
conjugacy in Section~4 of that paper hold after the $C^1$-neighborhood
of $f_0$ is reduced.  The suspension and holonomy-intersection
construction in the proof of Theorem~A then shows that the canonical center
leaf conjugacy and its inverse are $\theta_0$-H\"older.

\subsection{Expansion along the stable and unstable bundles}

In this subsection we prove the following lemma, which shows that many elements of
$\CZ(f)$ exhibit expansion in prescribed cones inside the stable bundle
$E^s_{f_0}$ up to macroscale. We will use this in Section \ref{sec:1exp} to show that some
elements of $\CZ(f)$ are partially hyperbolic.

\begin{lem}\label{lem:coneexpan}
Let $g_0\in\CZ(f_0)$ be a diagonal map.
Suppose there exist constants $\mu>\lambda>1$ and a family of cones
$C(x)\subset E^s_{f_0}(x)$ with the following property: for any left
translation $g_1$ which is projectively $\epsilon$-close to $g_0$ and for any
$u\in C(x)$ we have  
\[Dg_1(u)\in C(g_1(x)),\quad\text{and}\quad
\mu\,|u|\ge |Dg_1(u)|\ge \lambda\,|u|.
\]

Let $g\in\CZ(f)$ be a diffeomorphism projectively $\epsilon$-close to $g_0$ and
let $\tilde g$ be the lift of  $\hat g=\phi_f^{-1}\,g\,\phi_f$ to $G$. Then for any $y\in\exp_x(C(x))$ and $n\in \N$ with
$d(x,y)<\mu^{-n}\epsilon$, we have
\[
d_G\big(\tilde g^n(x),\widetilde{\mathcal W}^c_{f_0}(y)\big)
\ge \tfrac{1}{4}\,\lambda^{n}\, d(x,y).
\]
\end{lem}

\begin{proof}
We work on the cover $G$ of $X=G/\Gamma$. Consider the center-stable leaf through
the identity $\widetilde{\mathcal W}^{cs}_{f_0}(\mathrm{id})$. Every
$z\in\widetilde{\mathcal W}^{cs}_{f_0}(\mathrm{id})$ can be written
uniquely as a product
\[
z = g^c(z)\cdot \mathfrak n(z),
\]
with $\mathfrak n(z)\in\widetilde{\mathcal W}^s_{f_0}(\mathrm{id})$ and
$g^c(z)\in \widetilde\W^c_{f_0}(\mathrm{id})$. Using this decomposition we compute
\[
\tilde g^n(y)\,(\tilde g^n(x))^{-1}
=\tilde g^n(y)y^{-1}\,(y x^{-1})\,(\tilde g^n(x)x^{-1})^{-1}.
\]
Since the center  subgroup normalizes the stable subgroup, the  factor
coming from the $y$-orbit changes only the center  coordinate. Put
 $r_i=\tilde g^i(y)(\tilde g^i(x))^{-1}=c_i s_i$ in center--stable
coordinates.  Then the stable  coordinates give the  exact recurrence
$s_{i+1}=\tau_{i+1}s_i\tau_{i+1}^{-1}$.  Hence the  stable component is
\[
\mathfrak n\big(\tilde g^n(y)\,(\tilde g^n(x))^{-1}\big)
= \tau_n\tau_{n-1}\cdots\tau_1\;\mathfrak n(yx^{-1})\;
\tau_1^{-1}\cdots\tau_n^{-1},
\]
where
\[
\tau_i=\tilde g\big(\tilde g^{\,i-1}(x)\big)\cdot\big(\tilde g^{\,i-1}(x)\big)^{-1}.
\]

By assumption $\mathfrak n(yx^{-1})\in C(\mathrm{id})$. The cone
invariance and expansion hypothesis imply inductively that each conjugate
$\tau_i\cdots\tau_1\;\mathfrak n(yx^{-1})\;\tau_1^{-1}\cdots\tau_i^{-1}$
remains in the cone and that its norm grows at least by a factor
$\lambda$ at each step. Thus for every $n\in\mathbb N$,
\[
\mathfrak n\big(\tilde g^n(y)\,(\tilde g^n(x))^{-1}\big)\in C(\mathrm{id}),
\qquad
\big\|\mathfrak n\big(\tilde g^n(y)\,(\tilde g^n(x))^{-1}\big)\big\|
\ge \lambda^{n}\big\|\mathfrak n(yx^{-1})\big\|.
\]

For any $z\in\widetilde{\mathcal W}^{cs}_{f_0}(\mathrm{id})$ with
$\|\mathfrak n(z)\|$ small we have $d_G(z,\mathrm{id})\ge
\tfrac{1}{2}\,\|\mathfrak n(z)\|$ (this is the standard comparison of
group distance and stable coordinate in a small neighborhood).  The upper cone bound and $d(x,y)<\mu^{-n}\epsilon$ keep every
stable coordinate through time $n$ in this neighborhood.  Since
$\tilde g$ fixes each model center leaf, we therefore get
\[
d_G\big(\tilde g^n(x),\widetilde{\mathcal W}^c_{f_0}(y)\big)
\ge \tfrac{1}{2}\,\big\|\mathfrak n\big(\tilde g^n(y)\,(\tilde g^n(x))^{-1}\big)\big\|
\ge \tfrac{1}{2}\,\lambda^{n}\big\|\mathfrak n(yx^{-1})\big\|.
\]
Finally, for small separations there is the estimate
$\|\mathfrak n(yx^{-1})\|\ge \tfrac{1}{2}d(x,y)$, so combining the
inequalities gives the claimed bound
\[
d_G\big(\tilde g^n(x),\widetilde{\mathcal W}^c_{f_0}(y)\big)
\ge \tfrac{1}{4}\,\lambda^{n}\,d(x,y).
\]
This completes the proof.
\end{proof}

As a corollary, if $g_0$ has sufficiently large root values on the
stable directions, then one obtains exponential separation in the stable
foliation for iterates of $g$.

\begin{cor}\label{cor:gexpandinEsf}
Let $g_0$ be a diagonal map such that
$|\lambda_{ij}(g_0)|>10$ for every root $\lambda_{ij}$ with $\lambda_{ij}(f_0)<0$. Let
$g\in\CZ(f)$ be an element projectively $\epsilon$-close to $g_0$. Then there
exist constants $\mu>\lambda>1$ (depending only on $g_0$) and constants
$C_1,C_2>0$ (depending only on the H\"older conjugacy and $\epsilon$)
such that for every $x\in X$ and every $y\in\mathcal W^s_f(x)$ with
$d(x,y)\le C_2\mu^{-n}$ either
\[
d\big(g^n(x),g^n(y)\big)\ge C_1\lambda^{n}\, d(x,y)^{1/\theta^2}
\] or \[
d\big(g^{-n}(x),g^{-n}(y)\big)\ge C_1\lambda^{n}\, d(x,y)^{1/\theta^2},
\]
where $\theta\in(0,1]$ is the H\"older exponent of the leaf conjugacy
between $f$ and $f_0$.
\end{cor}

\begin{proof}
By the hypothesis on $g_0$ there exist cone families  $\mathcal C_+(x),\mathcal C_-(x)\subset
E^s_{f_0}(x)$ and constants $\mu_1>\lambda_1>1$ such that
 $\mathcal C_+(x)\cup\mathcal C_-(x)=E^s_{f_0}(x)$,  $\mathcal C_+$ satisfies the assumptions of
Lemma \ref{lem:coneexpan} for $g_0$, and  $\mathcal C_-$ satisfies the assumptions
of Lemma \ref{lem:coneexpan} for $g_0^{-1}$.

Let  $\phi_f$ be the leaf conjugacy from $(f_0,\mathcal W^c_{f_0})$ to $(f,\mathcal W^c_f)$. For any $y\in\mathcal W^s_f(x)$ we have
 $\phi_f^{-1}(y)\in\mathcal W^{cs}_{f_0}(\phi_f^{-1}(x))$. Put  $x_1=\phi_f^{-1}(x)$. Since
 $\mathcal C_+(x_1)\cup\mathcal C_-(x_1)=E^s_{f_0}(x_1)$ there exists
 $y_1\in\exp(\mathcal C_+(x_1))\cup\exp(\mathcal C_-(x_1))$ with
 $\phi_f^{-1}(y)\in\mathcal W^c_{f_0}(y_1)$. Assume  $y_1\in\exp(\mathcal C_+(x_1))$ (the
other case is analogous using $g_0^{-1}$). If
 $d(y_1,\phi_f^{-1}(x))\le \mu_1^{-n}\epsilon$, Lemma \ref{lem:coneexpan}
gives
\[
d\big(\hat g^n(x_1),\mathcal W^c_{f_0}(y_1)\big)\ge \tfrac{1}{4}\lambda_1^n d(x_1,y_1),
\]
where  $\hat g=\phi_f^{-1} g\phi_f$.

Since  $\phi_f$ is $\theta$-bi-H\"older, there is a constant $C>0$ such that
\[
d\big(g^n(x),g^n(y)\big)
\ge C\,d\big(\hat g^n(x_1),\mathcal W^c_{f_0}(y_1)\big)^{1/\theta}
\ge C'\,\lambda_1^{n/\theta} d(x_1,y_1)^{1/\theta}.
\]
Using the bi-H\"older bounds again to relate $d(x_1,y_1)$ to $d(x,y)$
gives
\[
d\big(g^n(x),g^n(y)\big)\ge C''\,\lambda_1^{n/\theta} d(x,y)^{1/\theta^2},
\]
provided $d(x,y)\le \epsilon^{2/\theta}\mu_1^{-n/\theta}$. Choosing
$\mu,\lambda,C_1,C_2$ appropriately (absorbing the powers of $\theta$ and
the constants) yields the stated inequality. The alternative inequality
for backward iterates is handled similarly when $y_1$ lies in the cone
 $\mathcal C_-(x_1)$.
\end{proof}

\begin{lem}\label{gnozeroexp}
Assume that $f_0$ has only one nonzero root value up to sign.
 Let $g_0$ be a diagonal map such that
$|\lambda_{ij}(g_0)|\neq 0$ for every root $\lambda_{ij}$ with $\lambda_{ij}(f_0)<0$. Then there exists  $\epsilon_0>0$ such that if
$g\in\CZ(f)$ is projectively  $\epsilon_0$-close to $g_0$, then $Dg|_{E^s_f}$ cannot have a zero Lyapunov exponent. The analogous statement holds  on $E^u_f$. 
\end{lem}
\begin{proof}
Suppose $g$ has a zero Lyapunov exponent  on $E^s_f$. By  Oseledets'
theorem, for every $\epsilon>0$ there are a regular point $x$ in
the  zero-exponent set,  a nonzero vector $u(x)\in E^s_f(x)$, and a
constant  $K(x)>0$ such that
 \[
\|Dg^n_xu(x)\|\le K(x)e^{\epsilon|n|}\|u(x)\|,
\qquad n\in\mathbb Z.
\]

Choosing $f$ sufficiently $C^1$-close to
$f_0$, the one-root splitting gives the jointly integrable
center--stable and center--unstable foliations, and the canonical leaf
conjugacy carries the corresponding model foliations to them.  Thus the
hypotheses of Lemma~\ref{lem:holexp1} hold, and we may assume that
the leaf conjugacy has H\"older exponent
$\theta$ arbitrarily close to $1$. We shall show that this
sub-exponential bound contradicts Corollary ~\ref{cor:gexpandinEsf}.
 Take a fixed power  $g_0^k$, if necessary, so that
$|\lambda_{ij}(g_0^k)|>10$ for  every root  nonzero on $f_0$.  After
decreasing $\epsilon_0$, the  corresponding fixed power $g^k$ is
projectively  close enough to $g_0^k$.

 By Corollary~\ref{cor:gexpandinEsf}, for  every sufficiently  small
$y\in\mathcal W^s_f(x)$ and every $n$, there is a  sign
$\sigma_n\in\{-1,1\}$ such that
 \[
d\bigl(g^{\sigma_nkn}(y),g^{\sigma_nkn}(x)\bigr)
\ge C_1\lambda^n d(x,y)^{1/\theta^2},
\qquad d(x,y)<C_2\mu^{-n},
\]
where  $\mu>\lambda>1$.

Apply Theorem~\ref{thm:pesin-blocks} to the cocycle $Dg|_{E^s_f}$ at the regular point $x$, using tempering constant $\eta\ll \log \mu$. In the corresponding tempered Pesin charts along the orbit segment $\{g^jx:|j|\le kn\}$, the derivative is block diagonal with respect to the negative, zero, and positive Oseledets spaces, the zero block has growth between $e^{-\eta}$ and $e^\eta$, and the nonlinear terms have arbitrarily small derivative. Since the Pesin chart radii are tempered while $\mu^{-n}$ decreases exponentially, we may choose points $y_n\in\mathcal W_f^s(x)$ in the coordinates of the zero blocks in the Pesin block, so that for sufficiently large $n$, \[ C_2\mu^{-n}>d(x,y_n)>C_2\mu^{-n-1} \]  and \[ \frac{\exp_{x,\mathcal W_f^s}^{-1}(y_n)} {\|\exp_{x,\mathcal W_f^s}^{-1}(y_n)\|} \longrightarrow \frac{u(x)}{\|u(x)\|}. \] Taking $j=\pm kn$ in the preceding finite-time estimate gives  \[ d\bigl(g^{\pm kn}(y_n),g^{\pm kn}(x)\bigr) \le C(x)e^{\epsilon kn}d(x,y_n). \]
Taking the sign
$\sigma_n$ in the lower bound gives
 \[
C_1\lambda^n d(x,y_n)^{1/\theta^2}
\le C(x)e^{\epsilon kn}d(x,y_n),
\]
which implies that
 \[
C_1(\lambda e^{-k\epsilon})^{n}\le C(x)d(x,y_n)^{1-1/\theta^2}
\le C(x)C_2^{1-1/\theta^2}
(\mu^{(1-\theta^2)/\theta^2})^{n+1}.
\]
This holds for every  sufficiently large $n$.  Pick $\epsilon$ sufficiently small so that $e^{k\epsilon}<\lambda \mu^{-\frac{1-\theta^2}{\theta^2}}$; this yields a contradiction since $\mu$ and $\lambda$ are fixed before the neighborhood of $f_0$ is reduced once more, and $\theta$ can be made arbitrarily close to $1$  by Lemma \ref{lem:holexp1}.
\end{proof}

\section{Partially hyperbolic elements in the centralizer}\label{sec:gph}
In this section, we shall prove that many elements in $\CZ(f)$ are partially hyperbolic if $f_0$ has only one root up to sign.

We start with some definitions.
For the diagonal $\mathbb Z^2$-actions in this section, a root $\beta$ of
$\SL_n(\mathbb R)$ restricts to the linear functional
\[
 \chi_\beta(m,n)=\beta\bigl(m\log f_0+n\log g_0\bigr)
 \quad\text{on }\mathbb R^2.
\]
The nonzero restrictions, modulo positive proportionality, are the
\emph{coarse Lyapunov functionals}; the associated coarse Lyapunov
foliations are the maximal intersections of stable foliations with the
corresponding sign pattern.  An element $a\in\mathbb Z^2$ is
\emph{regular for the restricted action} if $\chi(a)\ne0$ for every
nonzero restricted Lyapunov functional $\chi$.  When we say ``regular'' in this section, applied
to $a\in\mathbb Z^2$, it has this restricted meaning, which is much weaker than an element of $\SL_n(\mathbb R)$ being generic.  

A finite set $F\subset\mathbb Z^2$ is called \emph{sufficient} if it is
symmetric, contains a generating set, each element is regular for the restricted action, and contains a common contracting element
for each pair of nonproportional coarse Lyapunov foliations.

An action $\alpha'$ on $Y$ is a \emph{factor} of an action $\alpha$ on
$X$ if there is a surjection $p:X\to Y$ with
$p\circ\alpha(a)=\alpha'(a)\circ p$ for every $a$.  Such a factor is
\emph{rank one} if $\alpha'(\mathbb Z^2)$ contains a virtually cyclic
subgroup whose coarse Lyapunov foliations are the same as those of the
full factor action.  A restriction of the diagonal action is
\emph{genuinely higher rank} if its lifted algebraic action on the ambient
group $G=\mathrm{SL}_n\R$ has no rank-one factor.  This is the terminology of
\cite[Definition~3.2 and Section~3.2]{VinWang}; see also
\cite[Section~5]{Wang26}.

Let
$
\alpha_0:\mathbb Z^2\longrightarrow\Diff^\infty(X)
$
be a genuinely higher-rank restriction of the diagonal action.
The coarse Lyapunov foliations of $\alpha_0$ are the
maximal intersections of stable foliations of
the maps $\alpha_0(a)$, $a\in \Z^2$.

In this section we prove the following.

\begin{thm}\label{thm:highrankgph}
Fix $K>1$. Let $f_0$ be a diagonal map with only one nonzero root value
up to sign on
$
X=\SL_n(\mathbb R)/\Gamma,\qquad n\ge5,
$
and let $\mathfrak g_1$ be a maximal semisimple subalgebra of
$\mathfrak g(\mathcal Z(f_0))$ whose simple factors have real rank at
least two.

Choose a diagonal element
$g_0\in\exp(\mathfrak g_1)$ such that $E^c_{g_0}\subset E^c_{f_0}$, and a finite sufficient set
$(0,1)\in F\subset\mathbb Z^2$.
Then the following holds.

For every $\epsilon>0$ there exists $\delta>0$ such that, if
\[
f\in\Diff^\infty_{\mathrm{vol}}(X),\qquad
d_{C^1}(f,f_0)<\delta,\qquad
d_{C^2}(f,f_0)\le K,\qquad
\mathcal Z(f)\doteq\mathcal Z(f_0),
\]
and the action of $\mathcal Z(f)$ is $K$-regular, then there exists
$g\in\mathcal Z(f)$ such that, for
$
\alpha(m,n)=f^m g^n,
$
the following hold.

\begin{enumerate}
\item For every $a\in F$, the map $\alpha(a)$ is
partially hyperbolic, and all these elements have the same center
foliation, denoted $\mathcal W_\alpha^c$.

\item Every leaf of $\mathcal W_\alpha^c$ is fixed by the whole action
$\alpha$, and there is a $C^0$-small bi-H\"older leaf conjugacy
\[
h_c:(X,\mathcal W_{\alpha_0}^c)\longrightarrow
     (X,\mathcal W_\alpha^c)
\]
which conjugates the induced actions on leaf spaces.

\item For every $a\in F$ and every $\eta>0$, there is
$C=C(a,\eta)>0$ such that
\[
\sup_{x\in X}
\left\|D\alpha(na)|_{E_\alpha^c(x)}\right\|
\le Ce^{\eta|n|}
\qquad\text{for every }n\in\mathbb Z.
\]

\item Put $\widehat\alpha=h_c^{-1}\alpha h_c$. For every
$a\in F$, the stable holonomies of $\widehat\alpha(a)$
between local $\mathcal W_{\alpha_0}^c$-plaques are uniformly
$C^0$--$\epsilon$-close to those of $\alpha_0(a)$, for stable legs
of length at most $\epsilon^{-1}$.

\item For every $a\in F$, $\alpha(a)$ and $\alpha_0(a)$ are
projectively $\epsilon$-close in the sense that the center
translation of $\widehat\alpha(a)$ at any point of $\W_{\alpha_0}^c$
lies in the $\epsilon$-cone about the model translation by $\alpha_0(a)$.
\end{enumerate}
\end{thm}

In this section, let $p\ge q$ be the two multiplicities of the diagonal entries of $f_0$,
so that $p+q=n$ and $\dim E_f^s=pq$.  Let $\mathfrak g_1$ be the
higher-rank semisimple subalgebra chosen in
Theorem~\ref{thm:highrankgph}, and set
$
G_1<\CZ(f)^0$ to be the connected analytic subgroup with
Lie algebra $\Xi^{-1}(\mathfrak g_1).$

Choose a diagonal element $g_0\in\exp(\mathfrak g_1)$ such that $E^c_{g_0}\subset E^c_{f_0}$.  We take the corresponding $g\in G_1$ as in Proposition \ref{prop:g_0tog}. We shall prove in this section that $g$ is
partially hyperbolic, and its invariant foliations are $C^0$-close to those of
$g_0$.

\subsection{\texorpdfstring{$Df|_{E^s_f}$ only has one Lyapunov exponent}
{The stable derivative has only one Lyapunov exponent}}\label{sec:1exp}

We first prove that $Df|_{E_f^s}$ has only one Lyapunov exponent.  If it
had two or more, grouping its Oseledets spaces would give a measurable
$\CZ(f)$-invariant splitting
\[
E_f^s=E_f^1\oplus E_f^2,
\]
with both summands nonzero.  We assume such a splitting and derive a
contradiction.

For a real vector space $W$, write
\[
Q(W)=\GL(W)/\{\pm I\}.
\]

For $L\in\GL(W)$, we write $[L]_\pm$ for its class in $Q(W)$; this
notation is used throughout the section.

We identify $Q(W)$ with the real algebraic image of
$\GL(W)\to \GL(W\otimes W) :A\mapsto A\otimes A$.  Thus Zimmer's theorem applies directly to
$Q(W)$, and norms, condition numbers, invariant subspaces, and the
moduli of eigenvalues are independent of the choice of either lift of
an element of $Q(W)$.  

We shall repeatedly use the following form of the Lusin--recurrence
argument.

\begin{lem}\label{lem:observable-recurrence}
Let $T$ be an invertible measure-preserving transformation of a standard
probability space, and let $\mathcal O$ be a measurable map to a
second countable metric space.  For almost every $x$ there are sequences
$n_j\to+\infty$ and $m_j\to-\infty$ such that
\[
\mathcal O(T^{n_j}x)\longrightarrow\mathcal O(x),\qquad
\mathcal O(T^{m_j}x)\longrightarrow\mathcal O(x).
\]
The statement also holds simultaneously for any finite family of
measurable maps.
\end{lem}

\begin{proof}
Take a countable basis in the target containing neighborhoods of
arbitrarily small diameter.  For each basis element $U$, Poincar\'e
recurrence applies to $\mathcal O^{-1}(U)$.  Intersecting the resulting
countably many full-measure sets, and then choosing successively smaller
$U\ni\mathcal O(x)$, gives the positive sequence.  Apply the same
argument to $T^{-1}$ for negative times and to the product map for a
finite family of observables.
\end{proof}

We also record the linear-algebra consequence used with this lemma.

\begin{lem}\label{lem:positive-conjugation}
Let $A,C\in\GL(W)$ be real diagonalizable with positive eigenvalues and
let $B\in Q(W)$.  If
\[
{[C]_\pm}^{n_j}B{[A]_\pm}^{-n_j}\longrightarrow B
\quad\text{in }Q(W)
\]
for some $n_j\to\infty$, then ${[C]_\pm}=B{[A]_\pm} B^{-1}$.  In particular, if
$C=A$, then $B$ commutes with ${[A]_\pm}$.
\end{lem}

\begin{proof}

Choose a lift $\widetilde B$ of $B$.  After passing to a subsequence, the
hypothesis lifts to
$C^{n_j}\widetilde B A^{-n_j}\to\varepsilon\widetilde B$ for one fixed
$\varepsilon\in\{\pm1\}$.  Let $P_c$ and $Q_a$ be the spectral
projections of $C$ and $A$ for the positive eigenvalues $c$ and $a$.
The $(c,a)$ block of the left-hand side is
\[
 (c/a)^{n_j}P_c\widetilde BQ_a.
\]
If $c>a$, boundedness forces this block to vanish; if $c<a$, its limit is
zero, so convergence to $\varepsilon\widetilde B$ again forces it to
vanish.  Hence $P_c\widetilde BQ_a=0$ unless $c=a$.  Since
$\widetilde B$ is invertible, the matching eigenspaces have the same
dimensions and $C\widetilde B=\widetilde BA$.  Thus
$C=\widetilde B A\widetilde B^{-1}$.  (The nonzero matching blocks also
force $\varepsilon=1$.)

\end{proof}

\paragraph{\textbf{Step 1. Apply Zimmer's cocycle superrigidity to $Dg|_{E^1_f}$}}

Put $W=\mathbb R^m$.  After choosing a measurable
orthonormal trivialization of $E_f^1$, consider the
quotient differential cocycle
\[
\beta:\CZ(f)\times X\to Q(W),\qquad
\beta(g,x)=[Dg|_{E^1_f(x)}]_{\pm},
\]
where $m=\dim E^1_f$.

We use the following nonergodic form of Zimmer's cocycle superrigidity
theorem.  The symbol $\omega(x)$ records the ergodic component of $x$;
it is unrelated to the multiplicities $p,q$ fixed above.

\begin{thm}[{\cite[Theorem 4.4]{FMW}}]\label{thm:zimmer-nonerg}

Let $S$ be a finite product of connected, simply connected semisimple
real groups whose simple factors all have real rank at least $2$, and
let $S\curvearrowright(X,\mu)$ be a probability-preserving action. Let
$\mathbf H$ be a real algebraic group and let
$\beta:S\times X\to\mathbf H(\mathbb R)$ be an integrable Borel cocycle.
Let $\omega:X\to\Omega$ be the ergodic-decomposition map.

Then there are the following objects: a measurable family of homomorphisms
$\pi:\Omega\times S\to\mathbf H(\mathbb R)$, a measurable transfer
map $H:X\to\mathbf H(\mathbb R)$, a measurable family of compact
groups
\[
K_\xi\subset Z_{\mathbf H(\mathbb R)}(\pi_\xi(S)),
\]
and a measurable cocycle $c:S\times X\to\mathbf H(\mathbb R)$
with $c(a,x)\in K_{\omega(x)}$, such that
\[
\beta(a,x)=H(ax)\pi_{\omega(x)}(a)c(a,x)H(x)^{-1}
\]
for every $a\in S$ and almost every $x$.

\end{thm}

The derivative cocycle is uniformly bounded, hence
integrable.  Proposition~\ref{fparhyp} shows that $G_1$ preserves
volume.  Pulling the action and cocycle back to the simply connected
cover of $G_1$ does not change the invariant subbundles.  Apply
Theorem~\ref{thm:zimmer-nonerg} with $\mathbf H(\mathbb R)=Q(W)$.  It
gives a measurable quotient frame
$H(x)\in\operatorname{Iso}(W,E_f^1(x))/\{\pm I\}$, componentwise
representations $\pi_{\omega(x)}:G_1\to Q(W)$, and a compact-valued
error $c$.  Thus, for every $a\in G_1$ and almost every $x$,
\begin{equation}\label{eq:zimmer-decomposition}
[Da|_{E_f^1(x)}]_{\pm}
 =H(ax)\pi_{\omega(x)}(a)c(a,x)H(x)^{-1},
\end{equation}
where
\[
c(a,x)\in K_{\omega(x)}\subset
Z_{Q(W)}(\pi_{\omega(x)}(G_1)).
\]
Since the cover of $G_1$ is simply connected, each $\pi_\xi$ lifts to
$\GL(W)$.  Whenever eigenvalues or vectors are used below, we use this
lift and retain the notation $\pi_\xi$.

We next show that $\pi_{\omega(x)}$ is nontrivial almost everywhere.
Choose a regular semisimple $g_0\in\exp_{\CZ(f_0)}(\mathfrak g_1)$ whose
weights on $E_{f_0}^s$ are all nonzero, and let $g\in G_1$ be the
projectively close element supplied by
Proposition~\ref{prop:g_0tog}.  We claim more precisely that
$\pi_{\omega(x)}(g)$ has no eigenvalue of modulus one for almost every
$x$.

Otherwise there is an invariant positive-measure set $A$ on which such a modulus-one
eigenvalue occurs.  Since $g$ is semisimple and the compact error
$c(g,\cdot)$ commutes with $\pi_{\omega(\cdot)}(g)$, one can choose a
measurable unit vector $v(x)$ in the corresponding subspace and a
measurable function $C_g(x)$ such that
\[
\sup_{n\in\mathbb Z}
 \|\pi_{\omega(x)}(g^n)c(g^n,x)v(x)\|\le C_g(x).
\]
Choose an Oseledets-regular $x\in A$ for which
Lemma~\ref{lem:observable-recurrence} applies to $H$.  Along the
resulting sequences $n_j\to+\infty$ and $m_j\to-\infty$,
Equation~\eqref{eq:zimmer-decomposition} and compactness of the error
give
\[
\sup_j\|Dg^{n_j}(x)H(x)v(x)\|<\infty,
\qquad
\sup_j\|Dg^{m_j}(x)H(x)v(x)\|<\infty.
\]
At an Oseledets-regular point, bounded subsequences in both time
directions force the Lyapunov exponent of $H(x)v(x)$ to be zero.  This
contradicts Lemma~\ref{gnozeroexp}.  Hence
$\pi_{\omega(x)}(g)$ has no eigenvalue of modulus one almost everywhere;
in particular, $\pi_{\omega(x)}$ has no trivial summand.

\paragraph{\textbf{Step 2: Unique exponent using representation input}}

We now classify the possible representations
$\pi_{\omega(x)}:G_1\to\GL_m(\mathbb R)$.  Recall that
\[
\mathfrak g(\CZ(f))\simeq
\mathfrak{sl}_p(\mathbb R)\oplus
\mathfrak{sl}_q(\mathbb R)\oplus\mathbb R,
\qquad p\ge q,\quad p+q=n,
\]
and that $\dim E_f^s=pq$.

If $q=1$, then $G_1$ has Lie algebra
$\mathfrak{sl}_{n-1}(\mathbb R)$.  Step~1 shows that
$\pi_{\omega(x)}$ is nontrivial, while $m\le n-1$.  Hence $m=n-1$
and $\pi_{\omega(x)}$ is the standard representation or its dual.
Thus $E_f^2=0$.  The centralizer of this representation in
$\GL(W)$ is $\mathbb R^\times I$, so its compact centralizer in
$Q(W)$ is trivial.  Hence $c(a,x)=e$ almost everywhere.

If $p,q>2$, take
$\mathfrak g(G_1)=\mathfrak{sl}_p(\mathbb R)\oplus
\mathfrak{sl}_q(\mathbb R)$.  Applying Step~1 to regular split elements
in each simple factor rules out every irreducible summand on which one
factor acts trivially.  Every remaining irreducible summand has dimension
at least $pq=\dim E_f^s$.  Consequently $m=pq$,
$\pi_{\omega(x)}$ is irreducible, $E_f^2=0$, and again
$c(a,x)=e$ almost everywhere in $Q(W)$.

It remains to consider $q=2$, for which
$\mathfrak g(G_1)=\mathfrak{sl}_{n-2}(\mathbb R)$.  Nontriviality gives
$m\ge n-2$.  Applying the same argument to the complementary bundle
$E_f^2$ shows that either $E_f^2=0$, or
\[
\dim E_f^1=\dim E_f^2=n-2.
\]
In the latter case the representation on either summand is standard or
dual and its compact error is trivial in $Q(W)$.  Step~3 rules out this
last possibility.

\paragraph{\textbf{Step 3: The case
$\mathfrak g(G_1)=\mathfrak{sl}_{n-2}(\mathbb R)$.}}

Set $r=n-2$. Suppose, toward a contradiction, that
$Df|_{E_f^s}$ has more than one Lyapunov exponent. By Step~2, there is
then a measurable $\CZ(f)$-invariant splitting
\[
E_f^s=E_f^1\oplus E_f^2,\qquad
\dim E_f^1=\dim E_f^2=r,
\]
and the representation
$\pi_{\omega(x)}:G_1\to\GL_r(\mathbb R)$ on $E_f^1$ is, for almost every
$x$, either the standard representation or its dual. Moreover,
$c(a,x)=e$ in $Q(W)$ for every $a\in G_1$ and almost every $x$.

Let $\mu$ be Lebesgue measure and define the logarithmic determinant
homomorphism
\[
\mathcal D:\CZ(f)\longrightarrow\mathbb R,\qquad
\mathcal D(h)=\int_X\log\left|
\det\left(Dh|_{E_f^1(x)}\right)\right|\,d\mu(x).
\]

Since $E_f^1$ is $\CZ(f)$-invariant and every element of $\CZ(f)$
preserves $\mu$, the cocycle identity for the determinant gives
$\mathcal D(h_1h_2)=\mathcal D(h_1)+\mathcal D(h_2)$. The restriction of $\mathcal D$ to any
finite-dimensional Lie subgroup of $\CZ(f)$ is a Borel homomorphism,
hence is continuous. In particular, let $S_2<\CZ(f)^0$ be
the connected Lie subgroup with Lie algebra
$\Xi^{-1}(\mathfrak{sl}_2(\mathbb R))$.
Then the restriction $\mathcal D|_{S_2}$ is trivial, since $S_2$ is connected
and semisimple. Thus
\begin{equation}\label{eq:det-S2-zero}
\mathcal D(h)=0\qquad\text{for every }h\in S_2.
\end{equation}

We define the $Q(W)$-valued cocycle
\[
\tau(h,x)=H(hx)^{-1}[Dh|_{E_f^1(x)}]_{\pm}H(x).
\]
For $h\in \CZ(f)$, the superrigidity decomposition and the conclusion of
Step~2 give, for $g\in G_1$,
\[
\tau(g,x)=\pi_{\omega(x)}(g).
\]
We shall use the following lemma, whose proof is given in Step~4.

\begin{lem}\label{lem:tauginv}
Suppose that $c(g,x)=e$ in $Q(W)$ for every $g\in G_1$ and almost
every $x$. Then, for every $h\in Z_{\CZ(f)}(G_1)$ and every
$g\in G_1$,
\[
\tau(h,gx)=\tau(h,x)
\]
for almost every $x$.
\end{lem}

Fix $h\in S_2$. We show that all Lyapunov exponents of
$Dh|_{E_f^1}$ are equal. We use the following elementary
linear-algebra observation: if $A_1,\ldots,A_{r^2}$ form a basis of
$M_r(\mathbb R)$, then there is a finite constant
$C_{\mathrm{test}}(A_1,\ldots,A_{r^2})$ such that
\begin{equation}\label{eq:matrix-testing}
\kappa(T):=\|T\|\,\|T^{-1}\|
\le
C_{\mathrm{test}}(A_1,\ldots,A_{r^2})
\max_i\|TA_iT^{-1}\|
\end{equation}
for every $T\in\GL_r(\mathbb R)$. Indeed, this follows by writing every
matrix of norm at most one in the basis $\{A_i\}$ and taking the
supremum over the unit ball of $M_r(\mathbb R)$.

Choose a countable dense subgroup $\Lambda_1<G_1$ (the new symbol
avoids confusion with the lattice $\Gamma<G$) and enumerate all
ordered $r^2$-tuples in $\Lambda_1$ as
\[
\mathbf a^\ell=(a_1^\ell,\ldots,a_{r^2}^\ell),
\qquad \ell\in\mathbb N.
\]
Since $\pi_{\omega(x)}$ is the standard representation or its dual, it is
absolutely irreducible, and hence
\[
\operatorname{span}_{\mathbb R}\pi_{\omega(x)}(G_1)=M_r(\mathbb R).
\]
Thus, for almost every $x$, there is some $\ell$ for which
$\pi_{\omega(x)}(a_1^\ell),\ldots,
\pi_{\omega(x)}(a_{r^2}^\ell)$ form a basis of
$M_r(\mathbb R)$. Let $X_\ell$ be the measurable set on which $\ell$ is the
first such index, and for $x\in X_\ell$ set
\[
C_\ell(x)=
C_{\mathrm{test}}\bigl(\pi_{\omega(x)}(a_1^\ell),\ldots,
       \pi_{\omega(x)}(a_{r^2}^\ell)\bigr).
\]

Fix an Oseledets-regular $x\in X_\ell$ for which
Lemma~\ref{lem:observable-recurrence} applies simultaneously to
$H$ and the finitely many maps $\tau(a_i^\ell,\cdot)$.  Let
$n_j\to\infty$ be the resulting sequence and put
$T_j=\tau(h^{n_j},x)$. Since $h^{n_j}$
centralizes $G_1$, the cocycle identity gives
\[
\tau(a_i^\ell,h^{n_j}x)T_j
=
\tau(h^{n_j},a_i^\ell x)\tau(a_i^\ell,x).
\]
By Lemma~\ref{lem:tauginv},
$\tau(h^{n_j},a_i^\ell x)=T_j$, while
$\tau(a_i^\ell,x)=\pi_{\omega(x)}(a_i^\ell)$. Therefore
\begin{equation}\label{eq:Tj-conjugation}
\tau(a_i^\ell,h^{n_j}x)
=
T_j\pi_{\omega(x)}(a_i^\ell)T_j^{-1}
\quad\text{in }Q(W).
\end{equation}
Notice that the matrices in the middle of
\eqref{eq:Tj-conjugation} are all evaluated at the fixed initial point
$x$ and hence do not depend on $j$.

The left-hand sides of \eqref{eq:Tj-conjugation} converge and hence are
uniformly bounded. Applying
\eqref{eq:matrix-testing} to the basis
$\{\pi_{\omega(x)}(a_i^\ell)\}$ and to arbitrary lifts of $T_j$ gives
\[
\sup_j\kappa(T_j)<\infty.
\]
On the other hand,
\[
[Dh^{n_j}|_{E_f^1(x)}]_{\pm}
=H(h^{n_j}x)T_jH(x)^{-1}.
\]
Since $H(h^{n_j}x)\to H(x)$, it follows that
\[
\sup_j\kappa\left(Dh^{n_j}|_{E_f^1(x)}\right)<\infty.
\]
At an Oseledets-regular point,
\[
\lambda_{\max}(h,x)-\lambda_{\min}(h,x)
=
\lim_{n\to\infty}\frac1n
\log\kappa\left(Dh^n|_{E_f^1(x)}\right).
\]
Evaluating this limit along the subsequence $n_j$ yields
\[
\lambda_{\max}(h,x)=\lambda_{\min}(h,x).
\]
Thus all Lyapunov exponents of $Dh|_{E_f^1}$ coincide almost
everywhere.

Because $h$ commutes with $f$, these Lyapunov exponents are
$f$-invariant measurable functions. Since $f$ is ergodic, their common
value is almost everywhere constant; denote it by $\chi(h)$. By
Oseledets' theorem and \eqref{eq:det-S2-zero},
\[
r\chi(h)
=
\int_X\log\left|
\det\left(Dh|_{E_f^1(x)}\right)\right|\,d\mu(x)
=
\mathcal D(h)=0.
\]
Hence every Lyapunov exponent of $Dh|_{E_f^1}$ is zero for every
$h\in S_2$.

Finally, choose a regular real-split element
$h_0$ in the $\SL_2(\mathbb R)$-factor of $\CZ(f_0)^0$
whose weights on $E_{f_0}^s$ are all nonzero, and choose
$h\in S_2$ projectively sufficiently close to $h_0$. The preceding
argument shows that $Dh|_{E_f^1}$ has only zero Lyapunov exponents,
whereas Lemma~\ref{gnozeroexp} says that
$Dh|_{E_f^s}$ has no zero Lyapunov exponent. This contradiction rules
out the splitting $E_f^s=E_f^1\oplus E_f^2$. Therefore
$Df|_{E_f^s}$ has only one Lyapunov exponent.

\paragraph{\textbf{Step 4: Proof of Lemma~\ref{lem:tauginv}.}}
Let $h\in Z_{\CZ(f)}(G_1)$.  The cocycle identities give, for
$a\in G_1$ and $n\in\mathbb Z$,
\begin{align}
\tau(h,a^nx)
&=\tau(a^n,hx)\tau(h,x)\tau(a^n,x)^{-1}\nonumber\\
&=\pi_{\omega(hx)}(a)^n\tau(h,x)
        \pi_{\omega(x)}(a)^{-n}.
\label{equ:taugdirection}
\end{align}
Lemma~\ref{lem:observable-recurrence}, applied to
$x\mapsto\tau(h,x)$, gives for almost every $x$ a sequence
$n_j\to\infty$ for which $\tau(h,a^{n_j}x)\to\tau(h,x)$. Hence
\begin{equation}\label{tauhconv}
\pi_{\omega(hx)}(a)^{n_j}\tau(h,x)
 \pi_{\omega(x)}(a)^{-n_j}\longrightarrow\tau(h,x)
 \quad\text{in }Q(W).
\end{equation}

Take $g\in G_1$ real diagonalizable and apply
Lemma~\ref{lem:positive-conjugation} to $a=g^2$.  The two matrices
$\pi_{\omega(x)}(g^2)$ and $\pi_{\omega(hx)}(g^2)$ have positive
eigenvalues, so \eqref{tauhconv} gives
\[
\tau(g^2,hx)=\tau(h,x)\tau(g^2,x)\tau(h,x)^{-1}.
\]
Substitution into \eqref{equ:taugdirection} yields
\[
\tau(h,g^2x)=\tau(h,x).
\]
The squares of real-diagonalizable elements contain a nonempty open
subset of $G_1$, and every nonempty open subset generates the connected
group $G_1$.  The cocycle identity therefore extends the last relation
to every $a\in G_1$:
\begin{align}
\tau(a,hx)&=\tau(h,x)\tau(a,x)\tau(h,x)^{-1},\nonumber\\
\tau(h,ax)&=\tau(h,x).
\label{taufgequ}
\end{align}
This proves Lemma~\ref{lem:tauginv}.

\subsection{Ergodicity}\label{sec:gerg}
We next prove that for our choice of $g_0$, the corresponding $g\in \CZ(f)$ is ergodic.

\paragraph{\textbf{Step 5. The splitting of $g$ is continuous}}

Fix $g\in G_1$.  Recall that we have proved  $E_f^1=E_f^s$, and we may extend the
notation $\tau$ to the group generated by $G_1$ and $f$ by setting
\[
\tau(h,x)=H(hx)^{-1}[Dh|_{E_f^s(x)}]_{\pm}H(x).
\]
For $\xi\in\Omega$, let
\[
V_g^s(\xi)\oplus V_g^c(\xi)\oplus V_g^u(\xi)=\mathbb R^{pq}
\]
be the sums of the generalized eigenspaces of $\pi_\xi(g)$ with
eigenvalues of modulus less than, equal to, and greater than one,
respectively.  The compact error centralizes $\pi_\xi(G_1)$ and
therefore preserves these three subspaces.  Since $f$ commutes with
$g$, the cocycle identity gives
\[
\tau(f,x)V_g^\ast(\omega(x))=V_g^\ast(\omega(fx)),
\qquad \ast=s,c,u.
\]
Consequently the measurable subbundles
\begin{equation}\label{eq:measurable-g-splitting}
E^{s,\ast}_{f,g}(x)=H(x)V_g^\ast(\omega(x)),
\qquad \ast=s,c,u,
\end{equation}
are invariant under $Df|_{E_f^s}$ as well as under $Dg$.

We use the following theorem from \cite{KalSad1exp}.

\begin{thm}[Theorem 3.3, \cite{KalSad1exp}; see also Theorem B,
\cite{ASV}]\label{contbund1exp}
Let $\mathcal E\to M$ be a finite-dimensional H\"older vector bundle,
and let $\mathcal A:\mathcal E\to\mathcal E$ be a fiber-bunched
H\"older cocycle over a $\mu$-ergodic, partially hyperbolic,
center-bunched, accessible $C^2$ map.  If $\mathcal A$ has only one
Lyapunov exponent, then every measurable $\mathcal A$-invariant
subbundle agrees almost everywhere with a continuous subbundle.  If the
base is locally H\"older accessible, the resulting subbundle is H\"older
\cite[Corollary~3.7]{KalSad1exp}.
\end{thm}

For the affine model, $Df_0|_{E_{f_0}^s}$ is conformal on the single root
space.  Thus, after fixing a H\"older exponent and shrinking the
$C^1$-neighborhood, the strict fiber-bunching inequalities persist for
$\mathcal A=Df|_{E_f^s}$.  Section~\ref{sec:1exp} shows that this
cocycle has one Lyapunov exponent.  Applying
Theorem~\ref{contbund1exp} to the three bundles in
\eqref{eq:measurable-g-splitting} makes each of them continuous.

Finally, Lemma~\ref{lem:observable-recurrence} applied to $H$ shows that the Lyapunov
exponents of $g$ on $E^{s,s}_{f,g}$ are strictly negative, those on
$E^{s,c}_{f,g}$ are zero, and those on $E^{s,u}_{f,g}$ are strictly
positive.  Indeed, along a recurrent subsequence the contribution of
$H(g^nx)$ is bounded, whereas the three rates are those of the
constant matrix $\pi_{\omega(x)}(g)$.  We have therefore obtained the
continuous splitting
\[
E_f^s=E^{s,s}_{f,g}\oplus E^{s,c}_{f,g}\oplus E^{s,u}_{f,g}
\]
with the asserted Lyapunov signs.

\paragraph{\textbf{Step 6. $g$ is ergodic.}}

We now show that the element $g\in G_1$ corresponding to the regular
model element $g_0$ is ergodic. The proof
uses the Hopf argument together with the ergodicity criterion of
Burns--Wilkinson \cite{BWerg}.

Applying the construction of Step~5 to the unstable bundle
$E^u_f$ (instead of $E^s_f$) yields, for any $g\in G_1$, a
continuous splitting
\[
E^u_f=E^{u,s}_{f,g}\oplus E^{u,c}_{f,g}\oplus E^{u,u}_{f,g}.
\]
Choose $g_0$ with no zero weight on
$E_{f_0}^s\oplus E_{f_0}^u$, as in
Theorem~\ref{thm:highrankgph}, and take the associated $g$.  Step~1 and
Lemma~\ref{gnozeroexp} give
$E^{s,c}_{f,g}=E^{u,c}_{f,g}=0$.

In the center direction $E^c_f$ of $f$, Lemma \ref{liealg} implies that
$Dg|_{E^c_f}$ is conjugate to the adjoint action $\operatorname{Ad}(g)$
on the Lie algebra $\mathfrak g(\CZ(f))$. Since $g$ is semisimple in
$\CZ(f)$, there is a smooth splitting
\[
E^c_f=E^{c,s}_{f,g}\oplus E^{c,c}_{f,g}\oplus E^{c,u}_{f,g}
\]
and constants $C_c>0$, $\nu_{c,s},\nu_{c,u}\in(0,1)$ such that
for all $n\in\N$:
\[
\begin{aligned}
|Dg^n(v)|&\le C_c\nu_{c,s}^n|v| &&(v\in E^{c,s}_{f,g}),\\
|Dg^{\pm n}(v)|&\le C_c|v| &&(v\in E^{c,c}_{f,g}),\\
|Dg^{-n}(v)|&\le C_c\nu_{c,u}^n|v| &&(v\in E^{c,u}_{f,g}).
\end{aligned}
\]

Set
\[
E^s_g=E^{s,s}_{f,g}\oplus E^{c,s}_{f,g}\oplus E^{u,s}_{f,g},\qquad
E^u_g=E^{s,u}_{f,g}\oplus E^{c,u}_{f,g}\oplus E^{u,u}_{f,g},\qquad
E^c_g=E^{c,c}_{f,g}.
\]
Then $TX=E^s_g\oplus E^c_g\oplus E^u_g$ is a
continuous splitting. The constants in the center estimates above may be
chosen uniform because the orbit-coordinate transfer is continuous on the
compact space $X$. In contrast, the contraction and expansion estimates on
the off-center Lyapunov bundles are nonuniform: there are measurable
functions $C_g(x)>0$ and $\kappa_g(x)\in(0,1)$ such that for a.e.\
$x\in X$ and all $n\in\N$:
\begin{equation}\label{gbund}
    \begin{aligned}
\|Dg^n(v)\|&\le C_g(x)\,\kappa_g(x)^n\|v\| &&(v\in E^s_g(x)),\\
\|Dg^{\pm n}(v)\|&\le C_g(x)\,\|v\| &&(v\in E^c_g(x)),\\
\|Dg^{-n}(v)\|&\le C_g(x)\,\kappa_g(x)^n\|v\| &&(v\in E^u_g(x)).
\end{aligned}
\end{equation} Here the estimate in the $E^c_g$ bundles comes from the dynamics in the center foliation given in Section \ref{sec:gcendyn}, not from the measurable transfer map $H$.

Hence this splitting coincides a.e.\ with the Oseledets splitting for $g$
(Theorem \ref{Oseled}). Therefore for a.e.\ $x$ the distributions $E^s_g$
and $E^u_g$ are tangent to the Pesin stable and unstable manifolds
$\mathcal W^-_g(x)$ and $\mathcal W^+_g(x)$ given by Theorem
\ref{pesin}. By Theorem \ref{pesabcont} these Pesin foliations are
transversely absolutely continuous.

Apply Lemma~\ref{lem:leafwise-pesin-ac}, with
$T=g$ and $\mathcal F=\mathcal W_f^s$.  Commutation of $f$ and $g$ makes
$\mathcal W_f^s$ invariant under $g$; its leaves are uniformly smooth,
and the stable foliation of the partially hyperbolic map $f$ is absolutely
continuous.  The restriction $Dg|_{E_f^s}$ has no zero exponent and has
the Oseledets splitting
\[
 E_f^s=E^{s,s}_{f,g}\oplus E^{s,u}_{f,g}.
\]
Denote the relative stable and unstable manifolds supplied by that lemma
by $\mathcal W^{s,s}_{f,g}$ and $\mathcal W^{s,u}_{f,g}$.  They are
embedded plaques inside $\mathcal W_f^s$, tangent respectively to
$E^{s,s}_{f,g}$ and $E^{s,u}_{f,g}$, and they are transversely absolutely
continuous with respect to the Riemannian volumes on almost every
$\mathcal W_f^s$-leaf. 

We next prove that the leaves of $\mathcal W^{s,s}_{f,g}$ and $\mathcal W^{s,u}_{f,g}$ have uniform radius.  For a Pesin regular point $x$, let $r_{ss}(x)$ be the
supremum of the intrinsic radii of embedded disks about $x$ in the
relative global stable manifold $\mathcal W^{s,s}_{f,g}(x)$.  The
leafwise Pesin lemma gives $r_{ss}(x)>0$ almost everywhere.  Choose
$N\ge1$ so that $Df^{-N}$ expands $E_f^s$ by a factor at least $2$.
Because $f$ commutes with $g$, it preserves the relative stable relation,
and hence
\[
 r_{ss}(f^{-N}x)\ge2r_{ss}(x).
\]
As $f^{-N}$ preserves volume, for every $t>0$ this gives
\[
 \mu\{r_{ss}>t\}\le\mu\{r_{ss}>2t\}\le\mu\{r_{ss}>t\}.
\]
Consequently, for every $m\ge1$,
$\mu\{r_{ss}>t\}=\mu\{r_{ss}>2^mt\}$.  Letting $m\to\infty$ gives
$\mu\{r_{ss}>t\}=\mu\{r_{ss}=\infty\}$; then letting $t\downarrow0$ and
using $r_{ss}>0$ almost everywhere shows that
$r_{ss}(x)=\infty$ almost everywhere.  The same argument gives an
infinite radius for $\mathcal W^{s,u}_{f,g}$. Since their
radii are infinite and the angle between the continuous complementary
bundles $E^{s,s}_{f,g}$ and $E^{s,u}_{f,g}$ is uniformly bounded below,
the relative plaques form a local product structure on one fixed scale
inside almost every $\mathcal W_f^s$-leaf.

Applying the same argument with $\mathcal F=\mathcal W_f^u$ gives
relative manifolds $\mathcal W^{u,s}_{f,g}$ and
$\mathcal W^{u,u}_{f,g}$, their leafwise absolute continuity, and their
global-radius local product structure inside almost every
$\mathcal W_f^u$-leaf.

We now apply the Hopf argument to prove ergodicity of $g$. Fix a
continuous function $\psi:X\to\R$ and define the forward and backward
Birkhoff averages
\[
\psi^+(x):=\limsup_{n\to\infty}\frac{1}{n}\sum_{i=1}^n\psi(g^i(x)),
\qquad
\psi^-(x):=\limsup_{n\to\infty}\frac{1}{n}\sum_{i=1}^n\psi(g^{-i}(x)).
\]
By the exponential contraction and expansion in $\mathcal W^-_g$ and $\mathcal W^+_g$, we have
$\psi^+(x)=\psi^+(y)$ for all $y\in\mathcal W^-_g(x)$ and a.e.\ $x$, and
$\psi^-(x)=\psi^-(y)$ for all $y\in\mathcal W^+_g(x)$ and a.e.\ $x$. By
Birkhoff's ergodic theorem the two averages agree almost everywhere, so
for any $t\in\R$ the set
\[
X_t:=\{x:\psi^+(x)\le t\}\cap\{x:\psi^-(x)\le t\}
\]
is both essentially $\mathcal W^-_g$-saturated and essentially
$\mathcal W^+_g$-saturated. Consequently $X_t$ is essentially both
$\mathcal W^{s,s}_{f,g}$- and $\mathcal W^{s,u}_{f,g}$-saturated.

Since $\W^{s,s}_{f,g}$ and $\W^{s,u}_{f,g}$ are absolutely continuous and form a local product structure in $\W^s_f$, the standard Hopf argument shows that $X_t$ is
essentially $\mathcal W^s_f$-saturated. 

The same argument in the
unstable direction shows $X_t$ is essentially $\mathcal W^u_f$-saturated.

Now we invoke Theorem \ref{burnswikerg} (the Burns--Wilkinson criterion) on $f$ (instead of $g$):
the Lebesgue density points of a measurable set which is both
$\mathcal W^s_f$- and $\mathcal W^u_f$-saturated must form a set which
is saturated by both foliations. Since $f$ is accessible (see
Proposition \ref{fparhyp}), and center-bunched, Theorem \ref{burnswikerg} implies that any set
saturated by both $\mathcal W^s_f$ and $\mathcal W^u_f$ has either full
measure or zero measure. Therefore $X_t$ has either full measure or zero
measure for each $t\in\R$, which forces the functions $\psi^+$ and
$\psi^-$ to be almost everywhere constant. Hence $g$ is ergodic.

\subsection{Partial hyperbolicity}\label{subsec:ph}
We now establish continuity of the coboundary and partial hyperbolicity of $g$.
\paragraph{\textbf{Step 7.  $Df|_{E_f^s}$ is cohomologous to a
constant modulo $\{\pm I\}$.}}

Apply the notation of Step~1 with $E_f^1=E_f^s$: thus $W$ now denotes
a fixed real vector space of dimension $d=\dim E_f^s$.
We regard a measurable transfer modulo sign as a measurable section
\[
\bar H(x)\in
\operatorname{Iso}(W,E_f^s(x))/\{\pm I\}.
\]
We shall prove that there are such a section $\bar H$, a homomorphism
$\bar\pi:G_1\to Q(W)$, and an element $A_f\in Q(W)$ such that, for
every $a\in G_1$ and almost every $x$,
\begin{equation}\label{eq:Q-cohomology}
[Da|_{E_f^s(x)}]_\pm
 =\bar H(ax)\bar\pi(a)\bar H(x)^{-1},
\qquad
[Df|_{E_f^s(x)}]_\pm
 =\bar H(fx)A_f\bar H(x)^{-1}.
\end{equation}

Suppose first that $q=1$ or $q>2$. Apply the construction of
Step~1 with $E_f^1=E_f^s$. By Step~2, the corresponding
representation $\pi:G_1\to Q(W)$ is absolutely irreducible after
lifting to $\GL(W)$, and the compact error is trivial in $Q(W)$. Hence
\[
\tau(a,x)=\pi(a).
\]
Set $\bar H=H$ and $\bar\pi=\pi$. Lemma~\ref{lem:tauginv}, applied with
$h=f$, gives $\tau(f,ax)=\tau(f,x)$. Taking the ergodic
element $g\in G_1$ constructed in Step~6, we conclude that
$\tau(f,x)$ is almost everywhere constant. This gives
\eqref{eq:Q-cohomology} in these two cases.

It remains to consider $q=2$. Put $r=n-2$, so that $r\ge3$ and
$d=2r$, and let $G_{12}<\CZ(f)^0$ be the connected Lie
subgroup with Lie algebra
$\Xi^{-1}(\mathfrak{sl}_r(\mathbb R)\oplus
\mathfrak{sl}_2(\mathbb R))$.
We pull the action back to the algebraically simply connected cover
\[
\SL_r(\mathbb R)\times\SL_2(\mathbb R)
\]
and use the same notation for the resulting action.

We use Lee's dynamical cocycle superrigidity theorem in the following
specialized form.

\begin{thm}[{\cite[Theorem~4.3]{Lee}}]\label{thm:lee}
Let
\[
S=\SL_r(\mathbb R)\times\SL_2(\mathbb R),\qquad r\ge3,
\]
act measurably on a standard probability space $(X,\mu)$, preserving
$\mu$, and assume that the $\SL_r(\mathbb R)$-factor acts ergodically.
Let $V$ be a finite-dimensional real vector space and let
$\beta:S\times X\to\GL(V)$ be a measurable cocycle such that
\[
\log\max\{\|\beta(a,\cdot)\|,\|\beta(a,\cdot)^{-1}\|\}\in L^2(X,\mu)
\]
for every $a\in S$. Then, after passing to a finite measurable
extension $\mathfrak q:\widehat X\to X$, there exist a measurable
transfer $\Phi:\widehat X\to\GL(V)$, a rational homomorphism
$\pi:S\to\GL(V)$, and a measurable cocycle
$\zeta:S\times\widehat X\to A_0(\mathbb R)$ with values in an amenable
real algebraic subgroup of $\GL(V)$ such that
\[
\beta(a,\mathfrak q(\hat x))
 =\Phi(a\hat x)\pi(a)\zeta(a,\hat x)\Phi(\hat x)^{-1}.
\]
Moreover, $\pi(S)$ commutes with $\zeta(a,\hat x)$, and for every
$a\in S$ the Lyapunov spectrum of the cocycle generated by
$\beta(a,\cdot)$ agrees with that of the constant matrix $\pi(a)$.
\end{thm}

\begin{rmk}
Lee states Theorem~4.3 for a weakly irreducible, ergodic action and
writes the superrigidity homomorphism in a complex algebraic ambient
group, with rational homomorphisms understood to be defined over
$\mathbb R$.  Here the only rank-one factor is $\SL_2(\mathbb R)$, so
weak irreducibility is exactly ergodicity of the complementary
$\SL_r(\mathbb R)$-factor; this also implies ergodicity of the full
$S$-action.  In the statement of Lee's Theorem~4.3, $\pi$ is written as
a rational homomorphism into $\GL(V_{\mathbb C})$.  In the proof,
however, the algebraic hull is a real algebraic group
$\mathbf L=\mathbf L(\mathbb R)$, a Levi subgroup $\mathbf F$ is chosen
over $\mathbb R$ with
$\mathbf L(\mathbb R)=\mathbf F(\mathbb R)\ltimes\mathbf U(\mathbb R)$.
Writing $\mathbf H$ for the noncompact semisimple factor of
$\mathbf F$, \cite[Theorem~4.1]{Lee} produces
\[
\pi:S\longrightarrow\mathbf H(\mathbb R)
 \subset\mathbf F(\mathbb R)\subset\GL(V).
\]
The remaining steps only add the commuting amenable error.  This gives
exactly the specialized statement above.
\end{rmk}

The ergodic element constructed in Step~6 belongs to the
$\SL_r(\mathbb R)$-factor, so that factor acts ergodically.  Take $V=W$
in Theorem~\ref{thm:lee}, choose a measurable orthonormal trivialization
of $E_f^s$, and apply the theorem to
\[
\beta^s(a,x)=Da|_{E_f^s(x)}.
\]
The $L^2$ condition follows from compactness of $X$ and smooth
dependence on $a$.

We claim that the representation $\pi$ is a product of the standard/dual representations on the simple components
\[
\pi\simeq V_r^\varepsilon\otimes V_2,
\qquad \varepsilon\in\{1,\vee\}.
\]
Indeed, rational representations of $S$ are completely reducible.
If an irreducible summand were trivial on one of the two simple
factors, then a regular positive real-split element in that factor
would have a zero exponent on this summand. Lee's equality of spectra
would then give a zero exponent for the corresponding derivative
cocycle, contradicting Lemma~\ref{gnozeroexp}. Thus every irreducible
summand is nontrivial on both factors and has dimension at least
$2r=d$. Since the total dimension is $d$, the representation is
irreducible and is the tensor product above. In particular, it is
absolutely irreducible, so
\[
Z_{\GL(W_\mathbb C)}(\pi(S))=\mathbb C^\times I.
\]
Since the error takes values in a real algebraic subgroup, it follows
that $\zeta(a,\hat x)\in\mathbb R^\times I$.

Projectivizing over $\R$ gives
\begin{equation}\label{eq:Lee-projective}
\mathbb P\beta^s(a,\mathfrak q(\hat x))
 =\widehat H(a\hat x)\mathbb P\pi(a)\widehat H(\hat x)^{-1},
\qquad \widehat H=\mathbb P\Phi.
\end{equation}

We now show that $\widehat H$ descends to $X$. Disintegrate the
invariant measure on $\widehat X$ as
$\hat\mu=\int\hat\mu_x\,d\mu(x)$ and equip
\[
Y=\widehat X\times_X\widehat X
\]
with the relatively independent measure
$\mu_Y=\int\hat\mu_x\otimes\hat\mu_x\,d\mu(x)$. Define
\[
R(\hat x_1,\hat x_2)
 =\widehat H(\hat x_2)^{-1}\widehat H(\hat x_1)
 \in\PGL(W).
\]
Equation~\eqref{eq:Lee-projective} gives
\begin{equation}\label{eq:R-covariance}
R(a\hat x_1,a\hat x_2)
 =\mathbb P\pi(a)R(\hat x_1,\hat x_2)\mathbb P\pi(a)^{-1}.
\end{equation}
We apply Lemmas~\ref{lem:observable-recurrence}
and~\ref{lem:positive-conjugation}, the latter after projecting from
$Q(W)$ to $\PGL(W)$.

Choose positive real-split elements $a_1,\ldots,a_N\in S$ such that
\[
\bigcap_{i=1}^N
Z_{\PGL(W)}(\mathbb P\pi(a_i))=\{e\}.
\]
Such a finite family exists because positive real-split elements are
Zariski dense, the projective centralizer of $\pi(S)$ is trivial,
and centralizers are algebraic. Applying
Lemma~\ref{lem:observable-recurrence} to $R$ for each $a_i$, and then
using \eqref{eq:R-covariance} and
Lemma~\ref{lem:positive-conjugation}, gives
\[
R(y)\in Z_{\PGL(W)}(\mathbb P\pi(a_i))
\]
for almost every $y\in Y$. Intersecting the finitely many
full-measure sets gives $R=e$ almost everywhere. Hence
$\widehat H$ is constant on almost every fiber of $\mathfrak q$ and descends to
a measurable projective frame
\[
H_{\mathrm{proj}}(x)\in
\operatorname{PIsom}(W,E_f^s(x))
\]
satisfying
\begin{equation}\label{eq:projective-base}
\mathbb P(Da|_{E_f^s(x)})
 =H_{\mathrm{proj}}(ax)\mathbb P\pi(a)
  H_{\mathrm{proj}}(x)^{-1}.
\end{equation}

We now treat the determinant part.

\begin{lem}[{\cite[Proposition~2.2.10 and Theorem~2.12.4]{BHV}}]
\label{lem:real-cocycle-T}
Let $L$ be a locally compact, second countable group.  Suppose that
$L$ has property {\rm (T)} and acts by measure-preserving transformations on
$(X,\mu)$. Let $\delta:L\times X\to\mathbb R$ be a measurable
additive cocycle. Suppose that $\delta(a,\cdot)\in L^2(X,\mu)$ for
every $a\in L$ and that $a\mapsto\delta(a,\cdot)$ is continuous as a
map to $L^2(X,\mu)$. Then there exists $u\in L^2(X,\mu)$ such that
\[
\delta(a,x)=u(ax)-u(x)
\]
for every $a\in L$ and almost every $x$.
\end{lem}

This follows from the Delorme--Guichardet characterization of property
${\rm (T)}$.  The implication used here, property ${\rm (T)}\Rightarrow
{\rm (FH)}$, is due to Delorme \cite[Theorem~V.1]{Delorme}; Guichardet
proved the converse for $\sigma$-compact locally compact groups
\cite{Guichardet}.  See also
\cite[Proposition~2.2.10 and Theorem~2.12.4]{BHV} for a more modern summary.

For $a\in G_1=\SL_r(\mathbb R)$, define an additive cocycle
\[
\delta(a,x)=\frac1d
 \log\left|\det(Da|_{E_f^s(x)})\right|.
\]
 Since $E_f^s$ is continuous and the
action depends smoothly on $a$, the map
$a\mapsto\delta(a,\cdot)$ is continuous in $L^2$. For $r\ge3$,
$\SL_r(\mathbb R)$ has property {\rm (T)} \cite{Kazhdan}; therefore
Lemma~\ref{lem:real-cocycle-T} gives
$\delta(a,x)=u(ax)-u(x)$.

Fix volume densities on $W$ and $E_f^s$. For each $x$, the map
\[
[L]_\pm\longmapsto
\bigl(\mathbb P L,|\det L|\bigr)
\]
identifies
$\operatorname{Iso}(W,E_f^s(x))/\{\pm I\}$ with
$\operatorname{PIsom}(W,E_f^s(x))\times\mathbb R_{>0}$; similarly,
\[
Q(W)\simeq\PGL(W)\times\mathbb R_{>0}.
\]
Since $\det\pi(a)=1$ on the connected semisimple group $G_1$,
Equation~\eqref{eq:projective-base} and the determinant coboundary
combine to give a measurable quotient frame $\bar H$ and a
homomorphism $\bar\pi:G_1\to Q(W)$ satisfying the first identity in
\eqref{eq:Q-cohomology}. In these coordinates,
\[
\bar H(x)=\bigl(H_{\mathrm{proj}}(x),e^{du(x)}\bigr),
\qquad
\bar\pi(a)=\bigl(\mathbb P\pi(a),1\bigr).
\]

Define
\[
\mathfrak A_f(x)=\bar H(fx)^{-1}[Df|_{E_f^s(x)}]_\pm\bar H(x)\in Q(W).
\]
The commutativity of $f$ with $G_1$ gives
\begin{equation}\label{eq:script-Af-covariance}
\mathfrak A_f(ax)=\bar\pi(a)\mathfrak A_f(x)\bar\pi(a)^{-1}.
\end{equation}
Choose the ergodic element $g\in G_1$ in a positive real-split
chamber. Lemma~\ref{lem:observable-recurrence} gives, for almost every
$x$, a sequence $n_j\to\infty$ such that
$\mathfrak A_f(g^{n_j}x)\to \mathfrak A_f(x)$. Under
$Q(W)\simeq\PGL(W)\times\mathbb R_{>0}$, the second factor is central,
while the first factor satisfies the projective conjugation lemma.
Thus $\mathfrak A_f(x)$ commutes with $\bar\pi(g)$. Equation
\eqref{eq:script-Af-covariance} then gives
$\mathfrak A_f(gx)=\mathfrak A_f(x)$, and ergodicity
of $g$ implies that $\mathfrak A_f(x)$ is almost everywhere equal to a constant
$A_f$. This proves \eqref{eq:Q-cohomology} in the case $q=2$.

The same argument, applied to $f^{-1}$ and $E_f^u$, gives the
corresponding quotient cohomology on the unstable bundle.
\paragraph{\textbf{Step 8. Continuity of the transfer and partial
hyperbolicity of $g$.}}

We use the following result of Kalinin--Sadovskaya.

\begin{thm}[{\cite[Theorem~1.4]{KalSad23}}]\label{KS23}
Let $f:X\to X$ be a $C^2$, accessible, center-bunched, partially
hyperbolic diffeomorphism preserving a smooth volume.  Let $\mathcal A$
be a constant $\GL(D,\mathbb R)$-valued cocycle with one Lyapunov
exponent, and let $\mathcal B$ be an su-$\vartheta$--H\"older,
fiber-bunched $\GL(D,\mathbb R)$-valued cocycle over $f$.  Then every
measurable conjugacy between $\mathcal A$ and $\mathcal B$ agrees almost
everywhere with an su-$\vartheta$--H\"older conjugacy.
\end{thm}

Let
\[
\mathsf D_x=Df|_{E_f^s(x)}.
\]
Step~7 gives a measurable quotient frame $\bar H$ and an element
$A_f\in Q(W)$ such that
\[
[\mathsf D_x]_\pm=\bar H(fx)A_f\bar H(x)^{-1}.
\]
Choose a lift $A\in\GL(W)$ of $A_f$ and consider the linear cocycles
\[
\mathcal A=A\otimes A
\quad\text{on }X\times(W\otimes W),\qquad
\mathcal B_x=\mathsf D_x\otimes\mathsf D_x
\quad\text{on }E_f^s\otimes E_f^s.
\]
The map
\[
\mathcal T(x)=\bar H(x)\otimes\bar H(x):
W\otimes W\longrightarrow E_f^s(x)\otimes E_f^s(x)
\]
is well defined, independent of the representative of $\bar H(x)$,
and satisfies
\begin{equation}\label{eq:tensor-bundle-cohomology}
\mathcal B_x=\mathcal T(fx)\mathcal A\mathcal T(x)^{-1}
\end{equation}
almost everywhere.

We first show that $\mathcal A$ has one Lyapunov exponent. Pass to an
$f$-invariant full-measure set on which
\eqref{eq:tensor-bundle-cohomology} holds for every iterate. Choose an
Oseledets-regular point $x$ for which
Lemma~\ref{lem:observable-recurrence} applies to $\mathcal T$, and let
$n_j\to\infty$ be the resulting sequence.  For some $M=M(x)>0$ we then
have
\[
\kappa(\mathcal A^{n_j})
 \le M^4\kappa(\mathcal B_x^{n_j})
 =M^4\kappa(\mathsf D_x^{n_j})^2,
\qquad \text{where }
\kappa(T)=\|T\|\|T^{-1}\|.
\]
Since $Df|_{E_f^s}$ has one Lyapunov exponent,
$n^{-1}\log\kappa(\mathsf D_x^n)\to0$. On the other hand, for a fixed matrix
the limit
\[
\lim_{n\to\infty}\frac1n\log\kappa(\mathcal A^n)
\]
is the difference between the logarithms of the largest and smallest
moduli of its eigenvalues. It follows that all eigenvalues of
$\mathcal A$ have the same modulus.

The cocycle $\mathcal B$ is su-$\vartheta$--H\"older. Shrinking the
$C^1$-neighborhood of $f_0$ if necessary, it is fiber bunched. Indeed,
$Df_0|_{E_{f_0}^s}$ is conformal on the single root space, so
$\mathcal B$ is fiber bunched at $f_0$; the required strict
fiber-bunching inequalities for an iterate persist under small
$C^1$ perturbations.  The standing hypotheses on the base follow from
Proposition~\ref{fparhyp}: $f$ is smooth,
volume-preserving, accessible, center-bunched, and partially
hyperbolic.  The algebraic bundle $E_{f_0}^s$ is smoothly trivial, and
orthogonal projection from $E_f^s$ to $E_{f_0}^s$ is a H\"older bundle
isomorphism for $f$ sufficiently close to $f_0$.  We use the induced
trivialization of $E_f^s\otimes E_f^s$ to regard $\mathcal B$ as a
$\GL(W\otimes W)$-valued cocycle.  Therefore
Theorem~\ref{KS23} applies and gives an su-$\vartheta$--H\"older
conjugacy
\[
\mathcal T_0(x):W\otimes W\longrightarrow
E_f^s(x)\otimes E_f^s(x)
\]
which agrees almost everywhere with $\mathcal T$.

We now recover a continuous quotient frame from $\mathcal T_0$. For each $x$,
consider
\[
\iota_x:
\operatorname{Iso}(W,E_f^s(x))/\{\pm I\}
 \longrightarrow
\operatorname{Iso}(W\otimes W,E_f^s(x)\otimes E_f^s(x)),
\qquad
[L]_\pm\longmapsto L\otimes L.
\]
This map is injective. Its image is closed: in local trivializations,
if $L_j\otimes L_j$ converges to an invertible map, then
$\|L_j\|^2$ and $\|L_j^{-1}\|^2$ remain bounded, so a subsequence of
$L_j$ converges to an invertible $L$, and the limit is $L\otimes L$.
The same argument shows that $\iota_x$ is a homeomorphism onto its
image.

The equality set $\{\mathcal T_0=\mathcal T\}$ has full measure and is therefore dense.
In local bundle trivializations, continuity of $\mathcal T_0$ and closedness
of the tensor-square image imply that
$\mathcal T_0(x)\in\operatorname{Im}\iota_x$ for every $x$. Consequently,
\[
\bar H_0(x)=\iota_x^{-1}(\mathcal T_0(x))
\]
is a continuous section of
$\operatorname{Iso}(W,E_f^s)/\{\pm I\}$ agreeing almost everywhere
with $\bar H$. We replace $\bar H$ by $\bar H_0$.

For every fixed $a\in G_1$, both sides of the identities in
\eqref{eq:Q-cohomology} are now continuous in $x$. Since they agree
almost everywhere and the volume has full support, they agree for
every $x$:
\begin{equation}\label{eq:continuous-Q-cohomology}
[Da|_{E_f^s(x)}]_\pm
 =\bar H(ax)\bar\pi(a)\bar H(x)^{-1},
\qquad
[Df|_{E_f^s(x)}]_\pm
 =\bar H(fx)A_f\bar H(x)^{-1}.
\end{equation}

Apply the first equation to the regular positive real-split element
$g\in G_1$ chosen in Step~6. Choose a real-diagonalizable lift
$A_g\in\GL(W)$ of $\bar\pi(g)$ with positive eigenvalues, and let
\[
W=W_g^s\oplus W_g^c\oplus W_g^u
\]
be the sums of eigenspaces with eigenvalues respectively less than,
equal to, and greater than $1$. Define
\[
E^{s,*}_{f,g}(x)=\bar H(x)W_g^*,
\qquad *=s,c,u.
\]
These subspaces are well defined because multiplication of a frame by
$-I$ does not change its image on the Grassmannian. They form a
continuous $Dg$-invariant splitting of $E_f^s$.

The quantities $\|L\|$ and $\|L^{-1}\|$ are invariant under replacing
a representative $L$ of $\bar H(x)$ by $-L$. Hence compactness of
$X$ gives a uniform bound for both. Iterating
\eqref{eq:continuous-Q-cohomology}, for arbitrary representatives
$L_x$ and $L_{g^nx}$, gives
\[
Dg_x^nL_x=\varepsilon_n(x)L_{g^nx}A_g^n,
\qquad \varepsilon_n(x)\in\{\pm1\}.
\]
It follows that there are $C_s>0$ and
$\nu_s,\nu_u\in(0,1)$ such that, for $n\ge0$,
\[
\begin{aligned}
\|Dg^nv\|&\le C_s\nu_s^n\|v\|,
 &&v\in E^{s,s}_{f,g},\\
\|Dg^{\pm n}v\|&\le C_s\|v\|,
 &&v\in E^{s,c}_{f,g},\\
\|Dg^{-n}v\|&\le C_s\nu_u^n\|v\|,
 &&v\in E^{s,u}_{f,g}.
\end{aligned}
\]
For the chosen $g$, there are no zero weights on $E_f^s$, and hence
$E^{s,c}_{f,g}=0$.

Repeating Steps~7 and~8 for $f^{-1}$ and $E_f^u$ gives a continuous
$Dg$-invariant splitting
\[
E_f^u=E^{u,s}_{f,g}\oplus E^{u,c}_{f,g}\oplus E^{u,u}_{f,g}
\]
with the analogous uniform estimates; again
$E^{u,c}_{f,g}=0$.

On $E_f^c$, Section~\ref{sec:gcendyn} gives a smooth invariant
splitting
\[
E_f^c=E^{c,s}_{f,g}\oplus E^{c,c}_{f,g}\oplus E^{c,u}_{f,g},
\]
with uniform exponential contraction on $E^{c,s}_{f,g}$, uniform
exponential contraction for $Dg^{-1}$ on $E^{c,u}_{f,g}$, and
uniformly bounded iterates in both directions on
$E^{c,c}_{f,g}$.

Set
\[
\begin{aligned}
E_g^s&=E^{s,s}_{f,g}\oplus E^{c,s}_{f,g}
       \oplus E^{u,s}_{f,g},\\
E_g^c&=E^{c,c}_{f,g},\\
E_g^u&=E^{s,u}_{f,g}\oplus E^{c,u}_{f,g}
       \oplus E^{u,u}_{f,g}.
\end{aligned}
\]
Then $TX=E_g^s\oplus E_g^c\oplus E_g^u$ is a continuous
$Dg$-invariant splitting. Combining the estimates above, there are
$C_{\mathrm{ph}}>0$ and $\nu_{\mathrm{ph}}\in(0,1)$ such that, for every $n\ge0$,
\[
\|Dg^n|_{E_g^s}\|\le C_{\mathrm{ph}}\nu_{\mathrm{ph}}^n,\qquad
\|Dg^{\pm n}|_{E_g^c}\|\le C_{\mathrm{ph}},\qquad
\|Dg^{-n}|_{E_g^u}\|\le C_{\mathrm{ph}}\nu_{\mathrm{ph}}^n.
\]
After passing to an iterate, these inequalities give the domination
between the three bundles. Thus $g$ is partially hyperbolic.

\subsection{Cone control of the fine foliations}\label{sec:c0closebund}

We prove the plaque estimate used below. Let
$\hat g=\phi_f^{-1}\,g\,\phi_f$ be the conjugate of $g$ by the leaf
conjugacy of $f$, and write $\hat{\mathcal W}^{*}_{\hat f,\hat g}=\phi_f^{-1}(\mathcal W^*_{f,g})$
for $*=s,s;s,u;u,s;u,u$; these are the topological stable/unstable
subfoliations of $\hat g$ inside those of $\hat f=\phi_f^{-1} f\phi_f$.
Consider $\hat{\mathcal W}^{s,s}_{\hat f,\hat g}$. Since $\phi_f$ is
$\theta$--H\"older, for any $x\in X$ and
$y\in\hat{\mathcal W}^{s,s}_{\hat f,\hat g}(x)$ there are
$C_1>0$ and $\rho>1$ such that
\[
d\big(\hat g^n(x),\hat g^n(y)\big)
 \le C_1\rho^{-\theta n}d(x,y)^{\theta^2}\qquad(n>0).
\]

Fix $\delta>0$. Replacing $g_0$ by a suitable power $g_0^k$ and taking
$\epsilon_1>0$ sufficiently small, there exists a family of cones
$C(x)\subset E^s_{f_0}(x)$ satisfying the hypothesis of Lemma
\ref{lem:coneexpan}; namely, for every left translation $g_1$ projectively
$\epsilon$-close to $g_0^k$ and every $u\in C(x)$,
\[
Dg_1(u)\in C(g_1(x)),\qquad \mu|u|\ge |Dg_1(u)|\ge\lambda|u|
\]
for fixed $\mu>\lambda>1$. Moreover, we may choose the cones so that any
vector $v\in E^s_{f_0}(x)\setminus C(x)$ lies in the $\delta$--cone of
the subspace $E^{s,s}_{f_0,g_0}=E^s_{f_0}\cap E^s_{g_0}$.

For any $\epsilon>0$, applying Proposition
\ref{gtransl} to $R>d_{\W^c_{f_0}}(\tilde g_0^k(x),x)$ and taking
$d_{C^1}(f_0,f)$ sufficiently small, the iterate $g^k$ is projectively
$\epsilon$-close to $g_0^k$. Therefore, Lemma~\ref{lem:coneexpan}
implies that for any $y_1\in\exp(C(x))$ with
$d(x,y_1)<\mu^{-n}\epsilon_1$,
\[
d_G\big(\tilde g^{kn}(x),\tilde \W^c_{f_0}(y_1)\big)
 >\tfrac14\lambda^n d(x,y_1).
\]

Fix $N>0$ sufficiently large. For any $x\in X$ and any
$y\in\hat{\mathcal W}^{s,s}_{\hat f,\hat g}(x)$ with
\[
\epsilon_1\mu^{-n-1}<d(x,y)<\epsilon_1\mu^{-n},\qquad n\ge N,
\]
take
$y_1=\W^c_{f_0}(y,\mathrm{loc})\cap\W^s_{f_0}(x,\mathrm{loc})$.

If $y_1\in\exp C(x)$, then combining the inequalities above gives
\[
C_1\rho^{-\theta kn}d(x,y)^{\theta^2}
 \ge d\big(\hat g^{kn}(x),\hat g^{kn}(y)\big)
 \ge d_G\big(\tilde g^{kn}(x),\tilde \W^c_{f_0}(y_1)\big)
 >\tfrac14\lambda^n d(x,y_1).
\]

On the other hand, we also have
$\phi_f(y_1)\in\W^c_f(\phi_f(y))$ and
$\phi_f(y)\in\W^s_f(\phi_f(x))$. Since
$d(\phi_f(y),\phi_f(x))\ll1$ and the angle between $E^c_f$ and $E^s_f$
is uniformly bounded away from zero, the local product estimate gives
\[
d(\phi_f(x),\phi_f(y_1))
 \ge d_G\big(\phi_f(x),\W^c_f(\phi_f(y_1),\mathrm{loc})\big)
 \ge C_2d(\phi_f(x),\phi_f(y)).
\]
The bi-H\"older bounds therefore imply, for uniform constants
$C_3,C_4,C_5>0$, that
\[
d(x,y_1)\ge C_3d(\phi_f(x),\phi_f(y_1))^{1/\theta}
 \ge C_4d(\phi_f(x),\phi_f(y))^{1/\theta}
 \ge C_5d(x,y)^{1/\theta^2}.
\]
Plugging this relation back, we get
\[
C_1\rho^{-\theta kn}d(x,y)^{\theta^2}
 \ge C_6\lambda^n d(x,y)^{1/\theta^2}.
\]
Using $d(x,y)\ge\epsilon_1\mu^{-n-1}$, we obtain an upper bound of
order $\mu^{n(1/\theta^2-\theta^2)}$ for the left-hand ratio.
By Lemma~\ref{lem:holexp1}, we may choose $\theta$ sufficiently close to
$1$ so that
\[
\mu^{1/\theta^2-\theta^2}<\rho^{\theta k}\lambda.
\]
Choosing $N$ sufficiently large yields a contradiction.
This shows that for any
$y\in\hat{\mathcal W}^{s,s}_{\hat f,\hat g}(x)$ with
$d(x,y)<\epsilon_1\mu^{-N}$, we must have $y_1\notin\exp(C(x))$.
Therefore, the foliation $\hat \W^{s,s}_{\hat f,\hat g}$ is inside the
$\delta$--cone of $E^{s,s}_{f_0,g_0}$ after projecting in the
$\W^c_{f_0}$ foliation to $\W^s_{f_0}$.

\subsection{Completion of the proof of Theorem~\ref{thm:highrankgph}}\label{subsec:proofgph}

\begin{proof}[Proof of item~(5)]
The model element $g_0$ and the finite set $F$ are chosen first.  Fix
$a=(m,n)\in F$ and choose compatible lifts.  Since $m$ and $n$ are
fixed, Proposition~\ref{prop:g_0tog}(1), the convergence
$\phi_f^{-1}f\phi_f\to f_0$, and continuity of a fixed finite
composition imply
\[
 \widetilde{\widehat{\alpha(a)}}(x_\ast)
 \longrightarrow \alpha_0(a)(x_\ast).
\]
In particular, the center displacement at $x_\ast$ is uniformly
bounded.  Apply Proposition~\ref{gtransl} to $\alpha(a)\in\CZ(f)^0$ and
let $L_a$ be the left translation determined by this base-point
displacement.  It gives uniform projective closeness of $\alpha(a)$ to
$L_a$, while the displayed convergence gives $L_a\to\alpha_0(a)$.
Because $F$ is finite and its elements are regular,
\[
 \min_{a\in F}
 d_{\widetilde{\mathcal W}^c_{f_0}}
   (\alpha_0(a)(x_\ast),x_\ast)>0.
\]
The triangle inequality therefore gives the required relative
projective estimate simultaneously for every $a\in F$ after shrinking
the $C^1$-neighborhood.  This neighborhood may depend on $g_0$ and
$F$; no estimate for $a\notin F$ is asserted or needed.
\end{proof}

\begin{proof}[Proof of items~(1), (2), and~(4)]
For $a=(m,n)\in F$, the continuous quotient cohomologies from
Steps~7--8 give
\[
[D\alpha(a)|_{E_f^s(x)}]_\pm
 =\bar H(\alpha(a)x)A_f^m\bar\pi(g)^n\bar H(x)^{-1},
\]
and the analogous identity on $E_f^u$.  By item~(5) and
Lemma~\ref{gnozeroexp}, these constant matrices have no eigenvalue of
modulus one.  Together with the orbit-coordinate splitting on $E_f^c$
from Section~\ref{sec:gcendyn}, this gives a continuous uniformly
partially hyperbolic splitting for every $\alpha(a)$, $a\in F$.  Since
the model elements in $F$ are regular, their neutral spaces are the
common center space of the action; the same orbit-coordinate description
therefore gives the common center foliation $\mathcal W^c_\alpha$.
This proves item~(1).

Proposition~\ref{prop:g_0tog}(4) supplies the same bi-H\"older map $h_c$
as a leaf conjugacy for both generators $f_0,f$ and $g_0,g$.  It
therefore conjugates the induced $\mathbb Z^2$-actions on their common
center leaf spaces, proving item~(2).

Finally, repeat the cone argument of Section~\ref{sec:c0closebund} for
the four stable/unstable fine foliations and for every $a\in F$.  Item~(5)
supplies the required projective closeness, and finiteness of $F$ makes
the plaque radii and cone widths uniform.  The projected plaques then lie
in arbitrarily small cones about the corresponding model plaques. On
legs of bounded length, uniform local transversality makes the associated
local product holonomies uniformly $C^0$-close to the model holonomies.
Combining with the fact that the leaf conjugacy $h_c$ is $C^0$-small, this proves
item~(4).
\end{proof}

\begin{proof}[Proof of item~(3)]
By Proposition~\ref{prop:g_0tog}(3), the common center leaves are the
orbits of $Z_{\CZ(f)^0}(g)^0$.  Both $f$ and $g$ commute with this
subgroup.  If $\mathcal B_c(x)$ denotes the differential of its orbit
coordinate, equivariance gives
\[
D\alpha(na)|_{E^c_\alpha(x)}
 =\mathcal B_c(\alpha(na)x)\mathcal B_c(x)^{-1},
 \qquad a\in F,\quad n\in\mathbb Z.
\]
The orbit frames and their inverses are uniformly bounded on the compact
manifold $X$.  Hence the left-hand side is uniformly bounded in $n$,
which is stronger than item~(3).
\end{proof}

This completes the proof of Theorem \ref{thm:highrankgph}.

\section{Proof of Theorem \ref{cenrignongen}}\label{sec:main-proof}

With the partial hyperbolicity structure established by Theorem \ref{thm:highrankgph}, we prove the main theorem using the geometric method of proving rigidity of higher rank restrictions of the Weyl chamber flow, following the ideas of \cite{DK,Vinhage,VinWang,Wang}.

We use the following rigidity statement, which combines
Theorem~6.1, Proposition~6.2, and the argument of Section~6.5 in
\cite{Wang26}.

\begin{thm}[Rigidity of a topological perturbation]
\label{thm:wang-rigidity}

Let $\alpha_0$ and $\alpha$ be the model and perturbed
$\mathbb Z^2$-actions supplied by Theorem~\ref{thm:highrankgph}.
Then $\alpha$ is
$C^\infty$ conjugate to a restriction of the diagonal action.

\end{thm}

The proof follows the arguments of \cite{Wang26,VinWang}.  We
give the application here and recall the details in
Appendix~\ref{app:proof-wang-rigidity}.  First, the conclusions of
Theorem~\ref{thm:highrankgph} imply that $\alpha$ is a sufficiently
small topological perturbation of $\alpha_0$ on the finite sufficient
set $F$ in the sense of Definition~\ref{def:topological-perturbation}.

\begin{prop}\label{prop:alpha-topological}

For every sufficiently small $\epsilon>0$, after reducing the
$C^1$-neighborhood in Theorem~\ref{thm:highrankgph}, the action
$\alpha=\langle f,g\rangle$ constructed there is an
$(\epsilon,r)$-topological perturbation on $F$ of
$\alpha_0=\langle f_0,g_0\rangle$, where $r>0$ depends only on the
model action.  The proof is given in
Appendix~\ref{app:verify-topological-perturbation}.

\end{prop}

Then, Theorem~6.1 and Proposition~6.2 of \cite{Wang26} give a H\"older
conjugacy to a homogeneous action, and the argument of
\cite[Section~6.5]{Wang26} upgrades this conjugacy to $C^\infty$.
A more detailed proof and explanation are recalled in
Appendix~\ref{app:proof-wang-rigidity}.

\begin{proof}[Proof of Theorem \ref{cenrignongen}]
The assertion that $\mathcal Z(f)$ is a Lie group and that the action
\[
\mathcal Z(f)\times X\longrightarrow X
\]
is jointly smooth follows from
\cite[Theorem~3.1]{Wang}; see Theorem~\ref{czlieg}.

Suppose first that $f_0$ has more than one nonzero root up to sign.
By the algebraic classification above, the center of
$\mathcal Z(f_0)$ has dimension at least two. The conclusion then follows directly from
\cite[Theorem~1.2]{Wang26}.

We therefore assume that $f_0$ has only one nonzero root value up to
sign. Up to finite
index,
\[
\mathcal Z(f_0)
 \doteq
 \SL_p(\mathbb R)\times\SL_q(\mathbb R)\times\mathbb R,
\qquad p\ge q,\qquad p+q=n,
\]
and we take
\[
\mathfrak g_1=
\begin{cases}
\mathfrak{sl}_p(\mathbb R)\oplus
\mathfrak{sl}_q(\mathbb R),&q>2,\\
\mathfrak{sl}_p(\mathbb R),&q\le2.
\end{cases}
\]

Choose $\log g_0$ in a split Cartan subalgebra of
$\mathfrak g_1$, outside the finitely many root
hyperplanes on which a root nonzero on $\mathfrak g_1$ vanishes.  Then
$E^c_{g_0}\subset E^c_{f_0}$ and all weights of $g_0$ on
$E^s_{f_0}\oplus E^u_{f_0}$ are nonzero.  Since $\log f_0$ is in the
central direction complementary to $\mathfrak g_1$, the two logarithms
are linearly independent; the resulting diagonal $\mathbb Z^2$-action
$\alpha_0=\langle f_0,g_0\rangle$ is genuinely higher rank.

The nonzero restricted Lyapunov functionals have only finitely many
walls in $\mathbb R^2$.  Choose a primitive lattice vector outside all
walls; among the lattice vectors completing it to a unimodular basis,
all but finitely many also avoid the walls.  For each coarse Lyapunov
functional choose a regular lattice point in its negative half-space,
and for each pair of nonproportional functionals choose a regular lattice
point in the intersection of their negative half-spaces.  Add their
negatives.  The resulting finite
symmetric set $F$ contains a generating set, is sufficient, and every
$a\in F$ is regular.  In particular
$E^c_{g_0}=E^c_{\alpha_0(a)}\subset E^c_{f_0}$ is the common center space for every $a\in F$.

Then up to reducing the
$C^1$-neighborhood of $f_0$, Theorem \ref{thm:highrankgph} gives
$g\in\mathcal Z(f)$ and an action
\[
\alpha(m,n)=f^m g^n
\]
with the properties stated there.

Theorem~\ref{thm:wang-rigidity}  then gives a smooth conjugacy
of $\alpha$ to a diagonal restriction.  In particular, $f$ is smoothly
conjugate to a diagonal map.  This proves
Theorem~\ref{cenrignongen} in the one-root case.

\end{proof}

\appendix
\section{Proof of the leafwise Pesin lemma}
\label{app:leafwise-pesin-proof}

We prove Lemma~\ref{lem:leafwise-pesin-ac} by applying the usual Pesin
construction inside the leaves of $\mathcal F$.  We separate the argument
into three lemmas: existence of the relative Pesin manifolds, absolute
continuity of their holonomies inside $\mathcal F$, and comparison with the
ambient Pesin manifolds.

\begin{lem}
\label{lem:app-relative-pesin}
Let $T\in\Diff^2(M)$ preserve a uniformly $C^2$ foliation $\mathcal F$,
and let $\nu$ be a $T$-invariant probability measure.  Assume that
$DT|_{T\mathcal F}$ has no zero Lyapunov exponent at $\nu$-almost every
point.  There is a countable family of compact Pesin blocks
$\{\Lambda_j\}$ covering the regular set modulo a $\nu$-null set such that, on
each $\Lambda_j$, the following data are uniform:
\begin{enumerate}
\item the dimensions of $E^-_{\mathcal F,T}$ and $E^+_{\mathcal F,T}$;
\item local $C^1$ disks
$\mathcal W^-_{\mathcal F,T,\mathrm{loc}}(x)$ and
$\mathcal W^+_{\mathcal F,T,\mathrm{loc}}(x)$ tangent at $x$ to
$E^-_{\mathcal F,T}(x)$ and $E^+_{\mathcal F,T}(x)$;
\item a radius $r_j>0$ and constants $C_j>0$, $\gamma_j>0$ such that the
relative stable and unstable disks have intrinsic radius at least $r_j$ and
satisfy the usual exponential contraction estimates.
\end{enumerate}
In particular, these disks are contained in the relative stable and unstable
sets $\mathcal W^-_{\mathcal F,T}(x)$ and
$\mathcal W^+_{\mathcal F,T}(x)$ defined above.
\end{lem}

\begin{proof}
Let $\mathcal R$ be the regular set for the restricted cocycle
$DT|_{T\mathcal F}$.  Since there are no zero exponents at $\nu$-almost
every point, modulo a $\nu$-null set we may decompose $\mathcal R$ into a countable union, where each set has Lyapunov exponents of absolute value at least $q^{-1}$
and constant $\dim E^-_{\mathcal F,T}$.  For each such set,
choose a Lyapunov-coordinate tolerance $\varepsilon_q>0$ sufficiently small
compared with $q^{-1}$.  Applying Theorem~\ref{thm:pesin-blocks} with this
tolerance and refining its compact blocks by these two conditions gives a
countable compact exhaustion of $\mathcal R$ modulo a null set.

Fix one such block $\Lambda$.  Write $\lambda>0$ for its uniform lower bound
on the absolute values of the restricted Lyapunov exponents and write
$\varepsilon>0$ for its chosen Lyapunov-coordinate tolerance.  We choose the
refinement so that $\varepsilon$ is as small compared with $\lambda$ as
needed below.  For $x\in\Lambda$
and $x_k=T^kx$, let
$
 C_k=C_\varepsilon(x_k)
$
be the coordinate change from Theorem~\ref{thm:pesin-blocks}.  Its domain
has the fixed splitting
$
 \mathbb R^{d^-}\oplus\mathbb R^{d^+},
$
and
\[
 C_{k+1}^{-1}\,DT_{x_k}|_{T\mathcal F}\,C_k
 =
 \begin{pmatrix}
 A_k&0\\
 0&D_k
 \end{pmatrix},
\]
where
\begin{equation}\label{eq:leafwise-linear-hyperbolicity}
 \|A_k\|\le e^{-\lambda+\varepsilon},
 \qquad
 \|D_k^{-1}\|\le e^{-\lambda+\varepsilon}.
\end{equation}
The tempered estimate in Theorem~\ref{thm:pesin-blocks} gives, for a
constant $K_\Lambda$ independent of $x\in\Lambda$,
\begin{equation}\label{eq:leafwise-tempered-coordinates}
 \|C_k\|,\ \|C_k^{-1}\|
 \le K_\Lambda e^{\varepsilon|k|}.
\end{equation}

Use leafwise exponential coordinates
\[
 \Phi_k(z)=\exp^{\mathcal F}_{x_k}(C_kz),
 \qquad
 F_k=\Phi_{k+1}^{-1}\circ T\circ\Phi_k.
\]
Uniform $C^2$ bounds for the leaves of $\mathcal F$ and for $T$, together
with \eqref{eq:leafwise-tempered-coordinates}, give radii $\rho_k>0$ on
which
\[
 \|DF_k(z)-DF_k(0)\|<\varepsilon_0,
 \qquad \|z\|<\rho_k,
\]
where $\varepsilon_0>0$ may be fixed arbitrarily small.  Moreover there is
a constant $c_0$, depending only on the fixed foliation atlas and the
uniform leafwise $C^2$ bounds, such that $\rho_k^{-1}$ grows at most like
$e^{c_0\varepsilon|k|}$.  Choose
$c_0\varepsilon<\eta<\lambda/10
$
and replace $\rho_k$ by the tempered minorant
\[
 r_k=\inf_{\ell\in\mathbb Z}e^{\eta|\ell-k|}\rho_\ell.
\]
Then
\begin{equation}\label{eq:leafwise-tempered-radii}
 e^{-\eta}\le \frac{r_{k+1}}{r_k}\le e^\eta,
\end{equation}
and $r_0$ is bounded below uniformly for $x\in\Lambda$.  On the balls
$\|z\|<r_k$ we still have
\begin{equation}\label{eq:leafwise-nonlinear-smallness}
 \|DF_k(z)-DF_k(0)\|<\varepsilon_0.
\end{equation}

Write $z=(u,v)\in\mathbb R^{d^-}\oplus\mathbb R^{d^+}$.  Then
\[
 F_k(u,v)
 =\bigl(A_ku+a_k(u,v),\,D_kv+b_k(u,v)\bigr),
\]
with
\[
 a_k(0)=b_k(0)=0,
 \qquad Da_k(0)=Db_k(0)=0.
\]
Choose $\varepsilon_0$ so that
\[
 e^{-\lambda+\varepsilon}+2\varepsilon_0<e^{-\eta}.
\]
After multiplying all $r_k$ by one fixed small constant, the backward graph
transform preserves graphs
\[
 v=\varphi_k(u),
 \qquad \operatorname{Lip}\varphi_k\le\frac13,
\]
defined on a fixed fraction of the ball of radius $r_k$, and is a contraction
in the graph metric.  This follows directly from
\eqref{eq:leafwise-linear-hyperbolicity},
\eqref{eq:leafwise-tempered-radii}, and
\eqref{eq:leafwise-nonlinear-smallness}.

Starting with the zero graph at time $N$ and pulling backward, then letting
$N\to\infty$, gives a unique invariant sequence of graphs.  Their images
under $\Phi_k$ are the local relative stable manifolds.  On $\Lambda$ there
are uniform constants $C_\Lambda>0$ and $\gamma_\Lambda>0$ such that
\begin{equation}\label{eq:leafwise-stable-contraction}
 d_{\mathcal F}(T^ny,T^nx)
 \le C_\Lambda e^{-\gamma_\Lambda n}d_{\mathcal F}(y,x),
 \qquad n\ge0,
\end{equation}
for $y$ in the local relative stable disk through $x$.  Its radius is
bounded below uniformly on $\Lambda$.  The global relative stable manifold
is obtained by saturation under backward iterates:
\[
 \mathcal W^-_{\mathcal F,T}(x)
 =\bigcup_{n\ge0}
 T^{-n}\mathcal W^-_{\mathcal F,T,\mathrm{loc}}(T^nx).
\]
Applying the same construction to $T^{-1}$ gives the relative unstable
manifolds and the corresponding uniform estimates.
\end{proof}

\begin{lem}
\label{lem:app-relative-ac}
On the blocks from Lemma~\ref{lem:app-relative-pesin}, the local relative
stable and unstable laminations have absolutely continuous holonomy inside
$\mathcal F$.  More precisely, fix two blocks $\Lambda,\Lambda'$ and a
foliation box.  On the part of a relative stable holonomy whose initial
point lies in $\Lambda$ and terminal point lies in $\Lambda'$, the
leafwise measures of a set and of its holonomy image are comparable, with
constants depending only on the two blocks and the foliation box.
\end{lem}

\begin{proof}
We give the stable case.  Work inside one $\mathcal F$-plaque and let
$\Sigma_0,\Sigma_1$ be sufficiently small smooth disks transverse to the
relative stable plaques.  Let
\[
 h:D_0\subset\Sigma_0\longrightarrow D_1\subset\Sigma_1
\]
be the relative stable holonomy.  Fix two blocks $\Lambda,\Lambda'$ and
restrict to
\[
 D_{\Lambda,\Lambda'}
 =\{z\in D_0:z\in\Lambda,\ h(z)\in\Lambda'\}.
\]
Since there are only countably many blocks, it is enough to obtain a uniform
absolute-continuity estimate on each such measurable piece.  After a
countable further decomposition, all charts, cone bounds, and local
extensions used below may be taken uniformly on the piece.

After shrinking the foliation box, the tangent spaces of $\Sigma_0$ and
$\Sigma_1$ lie in a common cone transverse to the stable direction.  Forward
graph transform narrows this cone.  At the same time, if $z$ and $h(z)$ are
paired by stable holonomy, then their forward iterates approach one another
at the exponential rate in \eqref{eq:leafwise-stable-contraction}.  Hence,
in the time-$n$ Lyapunov chart, the relevant pieces of
$T^n\Sigma_0$ and $T^n\Sigma_1$ are $C^1$-close graphs over the same
$E^+$-coordinate plane.  Let
\[
 \pi_n:T^n\Sigma_0\longrightarrow T^n\Sigma_1
\]
be the map obtained by matching the $E^+$ coordinates.  Then
\[
 \|D\pi_n-I\|\longrightarrow0
\]
uniformly on each fixed measurable chart piece.

Define the approximate holonomy
\[
 h_n=(T^n|_{\Sigma_1})^{-1}\circ\pi_n\circ(T^n|_{\Sigma_0}).
\]
Then $h_n\to h$ uniformly.  For $z\in D_{\Lambda,\Lambda'}$, put
\[
 w_n=h_n(z),\qquad z_k=T^kz,\qquad w_{k,n}=T^kw_n,
\]
and
\[
 P_{0,k}=DT^k(T_z\Sigma_0),
 \qquad
 P_{1,k,n}=DT^k(T_{w_n}\Sigma_1).
\]
The stable contraction, the forward and backward cone estimates, and the
$C^1$ regularity of $DT|_{T\mathcal F}$ give constants $C,c>0$, depending
only on the two blocks and the chosen chart piece, such that
\begin{equation}\label{eq:leafwise-distortion}
 \left|
 \log\bigl|\det(DT(z_k)|_{P_{0,k}})\bigr|
 -
 \log\bigl|\det(DT(w_{k,n})|_{P_{1,k,n}})\bigr|
 \right|
 \le C\bigl(e^{-ck}+e^{-c(n-k)}\bigr)
\end{equation}
for $0\le k<n$.  The first term compares the two true stable orbits, while
the second accounts for pulling backward the short projection made at time
$n$.

Summing \eqref{eq:leafwise-distortion} and using the uniform Jacobian bounds
for $\pi_n$ and $\pi_n^{-1}$ gives
\begin{equation}\label{eq:leafwise-holonomy-jacobian}
 C_1^{-1}\le\operatorname{Jac}h_n\le C_1,
\end{equation}
with the same estimate for $h_n^{-1}$ and with $C_1$ independent of $n$.
By the area formula, for every Borel set
$B\subset D_{\Lambda,\Lambda'}$,
\[
 \operatorname{vol}_{\Sigma_1}(h_n(B))
 \le C_1\operatorname{vol}_{\Sigma_0}(B).
\]
Interchanging $\Sigma_0$ and $\Sigma_1$ gives the corresponding
estimate for $h_n^{-1}$ and also $h_n^{-1}\to h^{-1}$ uniformly.  Thus
\[
 (h_n)_*\operatorname{vol}_{\Sigma_0}
 \le C_1\operatorname{vol}_{\Sigma_1}
\]
on the measurable piece, and similarly for the inverse push-forward.
Passing to weak limits and using the uniform convergence $h_n\to h$ gives
\[
 \operatorname{vol}_{\Sigma_1}(h(B))
 \le C_1\operatorname{vol}_{\Sigma_0}(B).
\]
Thus $h$ sends null sets to null sets on
$D_{\Lambda,\Lambda'}$.  Taking the countable union over pairs of blocks,
chart pieces, transversals, and foliation boxes proves absolute continuity
of the relative stable lamination inside $\mathcal F$.  The unstable case
follows by applying the same argument to $T^{-1}$.
\end{proof}

\begin{lem}
\label{lem:app-relative-intersection}
Let $x$ be regular for both the ambient derivative cocycle $DT$ and the
restricted cocycle $DT|_{T\mathcal F}$.  Suppose that
$\mathcal F_{\mathrm{loc}}(x)$ and
$\mathcal W^-_{T,\mathrm{loc}}(x)$ meet cleanly near $x$.  Then, after
shrinking the local plaques,
\[
 \mathcal W^-_{\mathcal F,T,\mathrm{loc}}(x)
 =\mathcal F_{\mathrm{loc}}(x)
  \cap\mathcal W^-_{T,\mathrm{loc}}(x).
\]
The analogous statement holds for unstable manifolds.
\end{lem}

\begin{proof}
By \eqref{eq:leafwise-stable-contraction}, points in the relative stable
plaque converge exponentially in leafwise distance, hence also in ambient
distance.  After shrinking the plaque, all forward iterates remain in the
ambient Pesin charts.  The local uniqueness part of the ambient stable
manifold theorem therefore gives
\[
 \mathcal W^-_{\mathcal F,T,\mathrm{loc}}(x)
 \subset
 \mathcal W^-_{T,\mathrm{loc}}(x).
\]
Since $T\mathcal F$ is $DT$-invariant, the Oseledets splitting for the
restricted cocycle is obtained by intersecting the ambient Oseledets
splitting with $T\mathcal F$; in particular,
\[
 E^-_{\mathcal F,T}(x)=T_x\mathcal F\cap E^-_T(x).
\]
Clean intersection therefore implies that
$\mathcal F_{\mathrm{loc}}(x)\cap\mathcal W^-_{T,\mathrm{loc}}(x)$ is a
submanifold with the same dimension and tangent space at $x$ as the relative
stable plaque.  The preceding inclusion is consequently a local equality
after shrinking.  Applying the same argument to $T^{-1}$ proves the
unstable statement.
\end{proof}

\begin{proof}[Proof of Lemma~\ref{lem:leafwise-pesin-ac}]
Let $\mathcal R$ be the regular set for $DT|_{T\mathcal F}$.  It has full
$\nu$-measure.  Since $\mathcal F$ is absolutely continuous and $\nu$ is
smooth, disintegration in a countable family of foliation boxes shows that
$\mathcal R$ has full leafwise Riemannian measure in almost every
$\mathcal F$-plaque.

Lemma~\ref{lem:app-relative-pesin} gives the relative stable and unstable
manifolds on a countable family of blocks covering $\mathcal R$ modulo a
null set.  Lemma~\ref{lem:app-relative-ac}, applied to every pair of these
blocks, gives the plaque-wise absolute continuity asserted in part~(2) of
Lemma~\ref{lem:leafwise-pesin-ac}.  Finally,
Lemma~\ref{lem:app-relative-intersection} gives part~(3) whenever the clean
intersection hypothesis holds.  This proves the lemma.
\end{proof}

\section{Topological perturbations and the rigidity input}
\label{app:topological-perturbations}

Let $A$ be an abelian group and let
$\alpha_0:A\to\Diff(X)$ be a restriction of the diagonal action.  Write
$\mathcal W^i_{\alpha_0}$ for its coarse Lyapunov foliations,
$\chi_i$ for the corresponding Lyapunov functionals, and
$\mathcal W^c_{\alpha_0}$ for its center foliation.

Recall that a finite set $F\subset A$ is called \emph{sufficient} if it is
symmetric, contains a generating set, each element is regular for the restricted action, and contains a common contracting element
for each pair of nonproportional coarse Lyapunov foliations.

We shall use the following definition of topological perturbation of $\alpha_0$, which is a variant of {\cite[Definition~5.1]{Wang26}}.

\begin{defi}[Topological perturbation on a sufficient set; adapted from
{\cite[Definition~5.1]{Wang26}}]
\label{def:topological-perturbation}
Fix a finite sufficient set $F\subset A$.  A
continuous action $\alpha:A\to\operatorname{Homeo}(X)$ is an
$(\epsilon,r)$-\emph{topological perturbation on $F$} of
$\alpha_0$ if it
preserves continuous, locally transverse foliations
$\mathcal W^i_\alpha$ and $\mathcal W^c_\alpha$ satisfying the following
conditions.
\begin{enumerate}
\item The dimensions of $\mathcal W^i_\alpha$ and
$\mathcal W^c_\alpha$ agree with those of their model counterparts.

\item The center foliation $\mathcal W^c_\alpha$ is uniformly
$C^0$--$\epsilon$-close to $\mathcal W^c_{\alpha_0}$.

\item Holonomies along $\mathcal W^i_\alpha$-paths of length at most
$\epsilon^{-1}$ are uniformly $C^0$--$\epsilon$-close to the
corresponding model holonomies.

\item For every $a\in F$, the map $\alpha(a)$ is projectively
$\epsilon/2$-close to $\alpha_0(a)$.

\item If $a\in F$ and $\alpha_0(a)$ is outside the $\epsilon$-cone about
$\ker\chi_i$, then $\alpha(a)$ contracts or expands
$\mathcal W^i_\alpha$ with the model signature.  Precisely, for
$x\in X$, $y\in\mathcal W^i_\alpha(x)$, and all sufficiently large
$|n|$ with
$\operatorname{sign}(n)=-\operatorname{sign}(\chi_i(a))$,
\[
d(\alpha(na)x,\alpha(na)y)\le e^{r\chi_i(a)n}.
\]

\item For $a\in F$, $x\in X$, $y\in\mathcal W^c_\alpha(x)$, and
$n\in\mathbb Z$,
\[
e^{-\epsilon^2\|\alpha_0(a)\||n|}d(x,y)
\le d(\alpha(na)x,\alpha(na)y)
\le e^{\epsilon^2\|\alpha_0(a)\||n|}d(x,y).
\]
Here $\|\alpha_0(a)\|$ is the norm of the corresponding translation
vector in the split Cartan algebra, and the distances in this clause are
the intrinsic distances on the corresponding center leaves.

\item If $\chi_i$ and $\chi_j$ are not proportional, there is
$a\in F$ for which $\chi_i(a)<0$, $\chi_j(a)<0$, and
$\alpha_0(a)$ lies outside the $\epsilon$-cones about both walls.
\end{enumerate}
\end{defi}
We note we here define topological perturbation restricted to a sufficient set, instead of for the whole action as given in \cite{Wang26}. In {\cite[Definition~5.1]{Wang26}}, $\alpha$ is said to be a topological perturbation of $\alpha_0$ if clauses (1)-(7) are satisfied for any element $a$ outside a small cone of the Weyl chamber walls of $\alpha_0$. However, as we shall see in Section \ref{app:proof-wang-rigidity}, having clauses (1)-(7) on a sufficient set is enough for the proof of the rigidity of the action.

\setcounter{thm}{1}
The rest of this section is devoted to the proof of Theorem \ref{thm:wang-rigidity}.

In Section \ref{app:verify-topological-perturbation}, we first prove Proposition~\ref{prop:alpha-topological}, which shows that for the actions $\alpha_0,\alpha$ and the finite symmetric sufficient set $F$ from
Theorem~\ref{thm:highrankgph}, $\alpha$ is a topological perturbation of $\alpha_0$ on the sufficient set $F$. 

Next, in Section \ref{app:proof-wang-rigidity}, we will recall the proof from Section 6 of \cite{Wang26}, which shows that a topological perturbation of a restriction of the diagonal action is smoothly conjugate to an affine model, this concludes the proof of Theorem \ref{thm:wang-rigidity}.

\subsection{Verification for the action constructed in
Theorem~\ref{thm:highrankgph}}
\label{app:verify-topological-perturbation}

\begin{proof}[Proof of Proposition~\ref{prop:alpha-topological}]
Let $F$ be the finite symmetric sufficient set from
Theorem~\ref{thm:highrankgph}. Since $F$ is sufficient, for each coarse Lyapunov functional
$\chi_i$, there exist a subset $F_i\subset F$ such that
\[
\mathcal W^i_{\alpha_0}(x,loc)
=
  \bigcap_{a\in F_i}\mathcal W^s_{\alpha_0(a)}(x,loc).
\]
and we define
\[
\mathcal W^i_\alpha(x,loc)
=
  \bigcap_{a\in F_i}\mathcal W^s_{\alpha(a)}(x,loc).
\]

 Let
$\mathcal W^c_\alpha$ be the common center foliation in
Theorem~\ref{thm:highrankgph}(1).

We verify the seven clauses of
Definition~\ref{def:topological-perturbation}.

\begin{enumerate}
\item By Proposition \ref{prop:g_0tog}, we have $\dim \W^c_\alpha(x)=\dim \W^c_g(x)=\dim \W^c_{g_0}(x)=\dim \W^c_{\alpha_0}(x)$. By
the proof of Theorem~\ref{thm:highrankgph}(4), the stable leaf $\W^s_{\hat \alpha(a)}(x)$ is inside a small cone about $\W^s_{\alpha_0(a)}(x)$ for each $a\in F$ and $x\in X$. This implies that the dimension of $h_c^{-1}(\W^i_{\alpha})$ is no larger than  the dimension of $\W^i_{\alpha_0}$ for any Lyapunov functional $\chi_i$, i.e. $\dim \W^i_\alpha(x)\le \dim \W^i_{\alpha_0}(x)$ for all $i$. Moreover, by construction, $$T_xX=T\W^c_\alpha(x)\bigoplus\oplus_iT\W^i_\alpha(x)=T\W^c_{\alpha_0}(x)\bigoplus\oplus_iT\W^i_{\alpha_0}(x).$$ This shows that $\dim \W^i_\alpha(x)=\dim \W^i_{\alpha_0}(x)$  for any $\chi_i$. 

\item The center leaf conjugacy in
Theorem~\ref{thm:highrankgph}(2) makes
$\mathcal W^c_\alpha$ uniformly $C^0$-close to
$\mathcal W^c_{\alpha_0}$.

\item By Theorem~\ref{thm:highrankgph}(4), stable holonomies for every
$a\in F$ are uniformly close to their model holonomies on legs of the
required bounded length.  Restricting these holonomies to the clean
intersections defining $\mathcal W^i_\alpha$ gives the required estimate for every coarse
Lyapunov foliation. 

\item Apply Theorem~\ref{thm:highrankgph}(5) with $\epsilon/2$ in
place of $\epsilon$.  It gives projective $\epsilon/2$-closeness for
every $a\in F$.

\item The set $F$ was chosen inside the interiors of the
relevant Weyl chambers.  Thus, whenever $\alpha_0(a)$ stays outside
the $\epsilon$-cone about $\ker\chi_i$, all root spaces entering
$\mathcal W^i_\alpha$ have the same model sign.  The projective
estimate in the preceding item and the all-$F$ cone argument
in the proof of Theorem~\ref{thm:highrankgph}(1), applied to
$\alpha(a)$ or its inverse, give the asserted contraction or expansion
on $\mathcal W^i_\alpha$.
Because only the finitely many pairs $(a,i)$ with
$a\in F$ occur, the constant $r$ is uniform and depends only
on $\alpha_0$ and the chosen sufficient set.

\item The common center leaves are the orbits of
$Z_{\CZ(f)^0}(g)^0$.  In orbit coordinates the derivative of every
$\alpha(a)$ along this foliation has the coboundary form
\[
D\alpha(a)|_{E^c_\alpha(x)}
=\mathcal B_c(\alpha(a)x)\mathcal B_c(x)^{-1}.
\]
The orbit-coordinate map $\mathcal B_c$ tends uniformly to the isometric model
coordinate as $f\to f_0$.  Hence the two separate bounds
\[
 \sup_x\|\mathcal B_c(x)\|<e^{\epsilon^2m_0/2},
 \qquad
 \sup_x\|\mathcal B_c(x)^{-1}\|<e^{\epsilon^2m_0/2}
\]
hold after reducing the neighborhood, where
\[
m_0=\min\{\|\alpha_0(a)\|:a\in F\}>0.
\]
The same formula for $na$ therefore gives a bi-Lipschitz constant at most
$e^{\epsilon^2m_0}$, which is bounded by
$e^{\epsilon^2\|\alpha_0(a)\||n|}$ when $n\ne0$; the case $n=0$ is
immediate.  This proves the two-sided center-distance estimate in
clause~(6), for every $a\in F$ and $n\in\mathbb Z$.

\item This follows directly from $F$ being sufficient. Simply choose $\epsilon$ below the minimum of
their positive angular distances from the relevant walls, which only depend on $\alpha_0$ and $F$. 
\end{enumerate}

The topological Lyapunov foliations are also accessible: the joint
integrations with $\chi_i(f_0)<0$ and $\chi_i(f_0)>0$ are respectively
$\mathcal W^s_f$ and $\mathcal W^u_f$, and $f$ is accessible by
Proposition~\ref{fparhyp}.  This is the accessibility used in the proof
of Theorem~\ref{thm:wang-rigidity}.  All clauses of the definition are
now verified, proving the proposition.
\end{proof}

\subsection{Proof of Theorem~\ref{thm:wang-rigidity}}
\label{app:proof-wang-rigidity}

We now prove that if $\alpha$ is a topological perturbation of $\alpha_0$
on a sufficient set $F$, then $\alpha$ is smoothly conjugate to a restriction
of the diagonal action.  The proof follows the same geometric rigidity
argument as Section~6 of \cite{Wang26}; the underlying Lyapunov-cycle
results are from \cite{VinWang}.

\begin{proof}[Proof of Theorem~\ref{thm:wang-rigidity}]
We use Theorem~6.1, Proposition~6.2, and Section~6.5 of
\cite{Wang26}, in the untwisted setting needed here.

Let $h_c$ be the center leaf conjugacy supplied by
Theorem~\ref{thm:highrankgph} and put
\[
\widehat\alpha(a)=h_c^{-1}\alpha(a)h_c.
\]
Put
\[
N_0=Z_{G^c}(g_0)^0
=\bigl(Z_G(f_0)^0\cap Z_G(g_0)^0\bigr)^0.
\]
The lifted common center leaves are the $N_0$-orbits.  Both $f_0$ and
$g_0$ are central in $N_0$.

Choose the lifts of the maps $\widehat\alpha(a)$ which preserve each
lifted model center leaf.  Their displacement in that leaf defines
\begin{equation}\label{eq:appendix-center-cocycle}
\beta_c(a,x)
=
\widetilde{\widehat\alpha(a)}(\tilde x)\tilde x^{-1}
\in N_0.
\end{equation}
The choice is compatible with products, so $\beta_c$ is a cocycle over
$\widehat\alpha$.  The leaf conjugacy and the perturbed foliations are
H\"older, and projective closeness makes $\beta_c$ $C^0$-close on the
generators to the constant center-displacement cocycle of $\alpha_0$.

We first prove cocycle rigidity.

\begin{thm}[{\cite[Theorem~6.1]{Wang26}}]
\label{thm:cocrig}
Any H\"older cocycle of $\widehat\alpha$ with values in $N_0$ that is
$C^0$-close to a constant cocycle on a generating set is H\"older
cohomologous to a constant.  Thus there exist a H\"older map
$U_c:X\to N_0$ and a homomorphism
$\rho:\mathbb Z^2\to N_0$ such that
\begin{equation}\label{eq:appendix-cohomology}
\beta_c(a,x)
=
U_c(\widehat\alpha(a)x)\rho(a)U_c(x)^{-1}.
\end{equation}
\end{thm}

Theorem~6.1 of \cite{Wang26} is stated with target $G^c$.
Its proof applies without change to the closed subgroup $N_0$: all stable
holonomies, Lyapunov cycle functionals, and the resulting transfer map
remain in $N_0$.
\begin{proof}
    For completeness, we recall the notation entering that argument.  Let
$\{\mathcal W^i_{\widehat\alpha}\}$ be the topological coarse Lyapunov
foliations of $\widehat\alpha$.  A \emph{Lyapunov path} is a finite path
\[
\gamma=[x_0,x_1,\ldots,x_k]
\]
such that each leg $[x_i,x_{i+1}]$ lies in one
$\mathcal W^j_{\widehat\alpha}$.  For a fixed base point $x_*$, let
$\mathcal C_{x_*}(\widehat\alpha)$ be the group of closed Lyapunov paths
based at $x_*$.  Let
$\mathcal C^0_{x_*}(\widehat\alpha)$ be the subgroup consisting of cycles
whose lifts to the universal cover are closed, and let
$\mathcal S_{x_*}(\widehat\alpha)$ be the closed normal subgroup generated
by \emph{stable cycles}, namely cycles contained in
$W^s_{\widehat\alpha(a)}(x_*)$ for some $a\in\mathbb Z^2$.

If $y\in\mathcal W^i_{\widehat\alpha}(x)$, choose $a\in F$ with
$\chi_i(a)<0$, and let $0<\lambda_i(a)<1$ be a uniform contraction
rate on this foliation.  Since the model center displacement is central
in $N_0$, by making the perturbation smaller we have
\[
\sup_{z\in X,\ \delta=\pm1}
\|\operatorname{Ad}(\beta_c(a,z)^\delta)\|
\lambda_i(a)^\kappa<1,
\]
where $\kappa$ is a H\"older exponent of $\beta_c$.  Hence the stable
cocycle holonomy
\begin{equation}\label{eq:appendix-cocycle-holonomy}
p_c(x,y)
=
\lim_{n\to\infty}
\beta_c(na,x)^{-1}\beta_c(na,y)
\end{equation}
exists; this is the stable holonomy of
\cite[Definition~6.1]{Wang26}.  The estimate is needed because $N_0$
need not be abelian, in particular in the case $q=2$.  The value is
independent of the choice of contracting $a$.

For a Lyapunov path
$\gamma=[x_0,\ldots,x_k]$, define
\[
P_c(\gamma)
=
\prod_{i=0}^{k-1}p_c(x_i,x_{i+1}).
\]
On closed paths this is the Lyapunov cycle functional.  By
\cite[Lemma~6.4]{Wang26}, it is trivial on
$\mathcal S_{x_*}(\widehat\alpha)$, and by
\cite[Proposition~6.3]{Wang26} it is trivial on all cycles if and only if
$\beta_c$ is cohomologous to a constant.

It therefore remains only to show that $P_c$ is trivial.  Lift the action
and the Lyapunov paths to the universal cover.  By
\cite[Proposition~6.5]{Wang26}, which is
\cite[Theorem~7.2]{VinWang}, the quotient of the model contractible-cycle
group by stable cycles is minimally almost periodic.  By
\cite[Proposition~6.6]{Wang26}, which is
\cite[Theorem~12.2]{VinWang}, the canonical correspondence between model
and perturbed contractible cycles preserves stable cycles.  The hypotheses
needed for this correspondence are precisely the $C^0$ control of the
foliations and the bounded holonomies in
Definition~\ref{def:topological-perturbation}.  It follows that $P_c$
vanishes on contractible cycles and hence factors through a homomorphism
\begin{equation}\label{eq:appendix-lattice-hom}
\overline P_c:\widetilde\Gamma\longrightarrow N_0.
\end{equation}
By \cite[Lemma~6.10]{Wang26}, this homomorphism is arbitrarily close to
the trivial homomorphism when the perturbation is sufficiently small, and
\cite[Lemma~6.9]{Wang26} (equivalently
\cite[Lemma~11.1]{VinWang}) implies that every sufficiently small
homomorphism $\widetilde\Gamma\to N_0$ is trivial.  Thus $P_c$ is trivial,
and Proposition~6.3 of \cite{Wang26} gives Theorem
\eqref{eq:appendix-cohomology}.
\end{proof}

We now apply Theorem~\ref{thm:cocrig} to show that $\alpha$ is H\"older conjugate to an affine action.

Let
\[
\alpha_\rho(a)(x):=\rho(a)\cdot x.\]

\begin{prop}[{\cite[Proposition 6.2]{Wang26}}]
    Let $\alpha,\alpha_0$ be the $\Z^2$-actions given by Theorem \ref{thm:highrankgph}. Then $\alpha$ is H\"older conjugate to the affine action $\alpha_\rho$.
\end{prop}
\begin{proof}
The proof is the same as the proof of \cite[Proposition 6.2]{Wang26}.

Let
\[
h_1(x):=U_c(x)^{-1}\cdot x.
\]
Since $U_c(x)\in N_0$, the point $h_1(x)$ lies on the same model center
leaf as $x$.  Equation~\eqref{eq:appendix-cohomology} gives
\begin{equation}\label{eq:h1-semiconjugacy}
h_1\circ\widehat\alpha(a)
=
\alpha_\rho(a)\circ h_1,
\qquad a\in\mathbb Z^2.
\end{equation}
Thus $h_1$ is a H\"older semiconjugacy from $\widehat\alpha$ to the
homogeneous action $\alpha_\rho$.  Since $X$ is compact and $U_c$ is
continuous, the displacement of $h_1$ inside each lifted center leaf is
uniformly bounded.

We recall the holonomy fact needed to see that this semiconjugacy is a
conjugacy.  If $\gamma$ is a Lyapunov path whose initial and terminal
points lie on center leaves, successive holonomies along its Lyapunov legs
give a map between the corresponding center leaves.  On the universal
cover, the maps obtained from Lyapunov paths beginning and ending on a
fixed lifted center leaf form its \emph{Lyapunov holonomy group}.
By \cite[Propositions~6.6 and~6.8]{Wang26}, the holonomy group of
$\widehat\alpha$ is canonically identified with the model holonomy group
and acts freely and transitively on every lifted center leaf.

In the present algebraic model this holonomy group is the simply connected
group covering $N_0$, acting by left translation.  We shall use the
elementary fact that every nonidentity element of this group has unbounded
powers.  Indeed, up to finite central quotients, the relevant factors of
$N_0$ are vector groups and split tori; when $q=2$ there may also be an
$\SL_2(\mathbb R)$ factor.  After passage to the universal cover, the
vector and split factors clearly have no nontrivial bounded cyclic
subgroups.  The same is true for
$\widetilde{\SL_2(\mathbb R)}$: hyperbolic and parabolic elements have
unbounded powers, while the lift of the elliptic circle is $\mathbb R$,
so a nontrivial elliptic lift also has unbounded powers.

Suppose that $h_1(x)=h_1(y)$ with $x\ne y$.  Since $h_1$ preserves model
center leaves, $x$ and $y$ lie on the same lifted center leaf.  By simple
transitivity there is a nontrivial holonomy element $H$ with $H(x)=y$.
Equation~\eqref{eq:h1-semiconjugacy}, applied along the Lyapunov legs
defining $H$, implies that the entire orbit
$\{H^n(x):n\in\mathbb Z\}$ lies in a single fiber of $h_1$.  That fiber is
bounded in the lifted center leaf because $U_c$ is uniformly bounded.
On the other hand, the preceding paragraph shows that the orbit of a
nontrivial holonomy element is unbounded.  This contradiction proves that
$h_1$ is injective.  By invariance of domain and compactness, $h_1$ is a
homeomorphism.

Consequently
\[
h:=h_1h_c^{-1}
\]
is a H\"older conjugacy from the original action $\alpha$ to
$\alpha_\rho$.  This is precisely the conjugacy conclusion of
\cite[Proposition~6.2]{Wang26}. Normalize $U_c$ to be the
identity at the chosen base point.  The construction of $U_c$ by
Lyapunov paths, together with the uniformly bounded accessible paths from
Proposition~\ref{fparhyp} and the estimate of
\cite[Lemma~6.10]{Wang26}, shows that
\[
\sup_{x\in X}d_{N_0}(U_c(x),e)\longrightarrow0
\]
with the size of the perturbation.  It follows from
\eqref{eq:appendix-cohomology} that $\rho$ is as close to $\alpha_0$ on
the generators as desired.
\end{proof}

It remains to show that $\alpha_\rho$ is conjugate to a restriction of the diagonal action and to upgrade smoothness of the conjugacy.  Write
\[
f_1=\rho(1,0),
\qquad
g_1=\rho(0,1).
\]
Conjugation by $h$ defines
\[
\Theta:\CZ(f)^0\longrightarrow G^c,
\qquad
\Theta(g)=hgh^{-1}.
\]
By \cite[Proposition~6.2]{Wang26}, its image is a Lie subgroup of
\[
G^c=\CZ(f_0)^0.
\]
Here the Lie-group topology agrees with the topology induced by the action
on a lifted center leaf; equivalently, this follows from the free and
proper center-leaf actions.  Thus $\Theta$ is a continuous homomorphism of
finite-dimensional Lie groups and hence is smooth.

By the hypothesis
\[
\mathcal Z(f)\doteq\mathcal Z(f_0),
\]
the image $\Theta(\CZ(f)^0)$ has the same dimension as the connected group
$G^c$.  Therefore
\[
\Theta(\CZ(f)^0)=G^c,
\]
so $\Theta$ is a Lie-group isomorphism.  Since every element of
$\CZ(f)^0$ commutes with $f$, the element $f$ is central in
$\CZ(f)^0$.  Hence
\[
f_1=\Theta(f)\in Z(G^c).
\]
In the present case
\[
G^c\doteq
\SL_p(\mathbb R)\times\SL_q(\mathbb R)\times\mathbb R,
\]
and $Z(G^c)^0$ is precisely the one-parameter diagonal subgroup containing
$f_0$.  Since $f_1$ is close to $f_0$, it lies in this identity component;
in particular $f_1$ is diagonal. 
No regularity assumption on $f_1$ as an element of
$\SL_n(\mathbb R)$ is needed.

We next identify $g_1$.  Recall that $g\in\CZ(f)^0$ was chosen by
\[
g
=
\exp_{\CZ(f)}
\bigl(\Xi^{-1}(\log g_0)\bigr),
\]
where $\Xi$ is the Lie-algebra isomorphism constructed above.  Hence $g$
is real-split semisimple in the Lie group $\CZ(f)^0$.  Since $\Theta$ is
a Lie-group isomorphism and $\Theta(g)=g_1$,
\[
d\Theta\circ\operatorname{Ad}(g)
=
\operatorname{Ad}(g_1)\circ d\Theta,
\]
so $g_1$ is real-split semisimple in $G^c$.  It is therefore contained in
a maximal $\mathbb R$-split torus of $G^c$ after conjugation inside
$G^c$.  Such a conjugation fixes the central element $f_1$.  Thus
$f_1$ and $g_1$ lie in a common split Cartan subgroup of $G^c$.
Closeness of $g_1$ to the positive model element $g_0$ places it in the
corresponding identity component.  Therefore
$\rho(\mathbb Z^2)$ is a restriction of the diagonal action.  This proves
the required algebraic part of Theorem~\ref{thm:wang-rigidity}.

It remains only to upgrade the conjugacy from H\"older to smooth.  We use
the standard smoothness argument from the final part of
\cite{Wang26}.  More precisely, \cite[Proposition~6.11]{Wang26} shows
that the preimage under $h$ of every generic real-split element of the
homogeneous action is partially hyperbolic and that the corresponding
topological coarse Lyapunov foliations are its smooth stable and unstable
foliations.  The normal-form argument of
\cite[Section~2.2.3, Step~4]{KS} then implies that $h$ is
$C^\infty$ along every coarse Lyapunov foliation.  Along center leaves,
$h$ conjugates the smooth transitive actions of the Lie groups
$\CZ(f)^0$ and $G^c$, and is therefore smooth there as well.  Successive
applications of Journ\'e's lemma \cite{Journe} to complementary joint
foliations give $
h\in C^\infty(X).
$

Thus $\alpha$, and in particular $f$, is smoothly conjugate to the
claimed restriction of the diagonal action.
\end{proof}

\raggedbottom
\bibliography{cenrig}
\bibliographystyle{acm}

\end{document}